\documentclass[10pt]{article} 

\usepackage[utf8]{inputenc}

\usepackage{geometry, amsmath, amssymb,wasysym}
\usepackage{amsthm}
\usepackage{tikz-cd} 
\usepackage{graphicx} 
\usepackage{booktabs,listings}
\usepackage{array} 
\usepackage{paralist} 
\usepackage{verbatim} 
\usepackage{subfig} 
\usepackage{bm}
\usepackage{hyperref} 
\usepackage[mathscr]{euscript}
\usepackage{enumerate}
\usepackage{float}
 \usepackage{wrapfig} 
\usepackage[romanian, english]{babel}
\usepackage{combelow}
\usepackage{algpseudocode}
\numberwithin{equation}{section}

\usepackage{chngcntr}
\counterwithin{figure}{section}

\newtheorem{thm}{Theorem}[section]
\newtheorem{newdef}[thm]{Definition}
\newtheorem{prop}[thm]{Proposition}
\newtheorem{lemma}[thm]{Lemma}
\newtheorem{remark}[thm]{Remark}

\newtheorem{question}[thm]{Question}
\newtheorem{claim}[thm]{Claim}
\newtheorem{cor}[thm]{Corollary}
\newtheorem{example}[thm]{Example}
\let\conjugatet\overline 

\newcommand{\R}{\mathbb{R}}
\newcommand{\C}{\mathbb{C}}
\newcommand{\Z}{\mathbb{Z}}

\newcommand{\M}{\mathcal{M}}
\newcommand{\tR}{{\widetilde{\mathcal{R}}}}
\newcommand{\tM}{{\widetilde{\mathcal{M}}}}
\newcommand{\brill}{\mathcal{B}}

\newcommand{\bk}{{\bf k}}
\newcommand{\bfr}{{\bf r}}
\newcommand{\bfm}{{\bf m}}
\newcommand{\bn}{{\bf n}}

\newcommand{\bq}{{\bf q}}

\newcommand{\bK}{{\bf K}}

\newcommand{\bM}{{\bf M}}
\newcommand{\bv}{{\bf v}}
\newcommand{\bw}{{\bf w}}
\newcommand{\bx}{{\bf x}}

\newcommand{\bX}{{\bf X}}

\newcommand{\by}{{\bf y}}
\newcommand{\bA}{{\bf A}}
\newcommand{\bB}{{\bf B}}
\newcommand{\bC}{{\bf C}}
\newcommand{\bkappa}{{\bf \kappa}}

\newcommand{\abs}[1]{\left\lvert#1\right\rvert}

\newcommand{\D}{\partial}

\newcommand{\eps}{\varepsilon}
\newcommand{\nit}{\noindent}
\newcommand{\nn}{\nonumber}

\usepackage{color}
\definecolor{darkblue}{rgb}{0,0,0.4}
\definecolor{darkgreen}{rgb}{0,0.6,0}

\usepackage{scalerel}
\newcommand\reallywidehat[1]{\arraycolsep=0pt\relax%
\begin{array}{c}
\stretchto{
  \scaleto{
    \scalerel*[\widthof{\ensuremath{#1}}]{\kern-.5pt\bigwedge\kern-.5pt}
    {\rule[-\textheight/2]{1ex}{\textheight}} 
  }{\textheight} %
}{0.5ex}\\           
#1\\                 
\rule{-1ex}{0ex}
\end{array}
}

\title{Spectral band degeneracies of $\frac{\pi}{2}-$rotationally invariant\\ periodic Schr\"{o}dinger operators}
\author{R.T. Keller\thanks{Department of Applied Physics and Applied Mathematics, Columbia University.} 
\and 
J.L. Marzuola\thanks{Department of Mathematics, University of North Carolina, Chapel Hill.}
\and
B. Osting\thanks{Department of Mathematics, University of Utah.}
\and
M.I. Weinstein\thanks{Department of Applied Physics and Applied Mathematics and Department of Mathematics, Columbia University.}
}

\begin{document}

 \pagestyle{myheadings}
 \thispagestyle{plain}
\markboth{ $\frac{\pi}{2}-$invariant Schr\"{o}dinger Operators}{R.T.  Keller, J.L. Marzuola, B. Osting, M.I. Weinstein}
\maketitle

\begin{abstract}
{This article was published in Multiscale Model. Simul., 16(4), 1684--1731 (2017). In this updated arXiv version we correct the statement and proof of Corollary \ref{rho-inv}. Clarifying edits  were also made 
 in the statements of Corollaries \ref{small-eps-disp} and \ref{appF-albega}. }\medskip

The dynamics of waves in periodic media is determined by the band structure of the underlying periodic 
Hamiltonian.
Symmetries of the Hamiltonian can give rise to novel properties of the band structure. Here we consider a class of periodic Schr\"{o}dinger operators, $H_V=-\Delta+V$, where $V$ is periodic with respect to the lattice of translates $\Lambda=\Z^2$. The potential is also assumed to be real-valued, sufficiently regular and such that, with respect to some origin of coordinates, inversion symmetric (even) and invariant under $\pi/2$ rotation.
\begin{enumerate} 
\item We present general conditions ensuring that the band structure of $H_V$ contains dispersion surfaces which touch at multiplicity two eigenvalues at the vertices (high-symmetry quasi-momenta) of the Brillouin zone. Locally, the band structure consists of two intersecting dispersion surfaces described by a normal form which is $\pi/2-$rotationally invariant, and to leading order homogeneous of degree two. Furthermore,  the effective dynamics of wave-packets, which are spectrally concentrated near high-symmetry quasi-momenta, is given by a system of coupled  Schr\"odinger equations with indefinite effective mass tensor.  
\item For small amplitude potentials, $\eps V$ with $\eps$ small or weak coupling, certain distinguished Fourier coefficients of the potential control which of the low-lying dispersion 
surfaces (first four) of $H^\eps=H_{_{\eps V}}$ intersect and have the above local behavior.
\item The existence of quadratically touching dispersion surfaces with the above properties persists for all real $\eps$,
 without restriction on the size of $\eps$, except for $\eps$ in a discrete set. 
\end{enumerate}
Our results apply to periodic superpositions of spatially localized ``atomic potentials'' centered on the square  ($\Z^2$) and  Lieb lattices.  We show, in particular, that the well-known conical plus flat-band structure of the 3 dispersion surfaces of the Lieb lattice  tight-binding model  does not persist in the corresponding Schr\"odinger operator with finite depth potential wells. Finally, we corroborate our  analytical results with extensive numerical simulations.  The present results are the  $\Z^2-$ analogue of results obtained for conical degenerate points (Dirac points) in the band structure for honeycomb structures. 
\end{abstract}

\clearpage

\section{Introduction}\label{sec:intro}
The dynamics of waves in periodic media are determined by the band structure of the underlying 
Hamiltonian; see, for example, \cite{Ashcroft-Mermin:76,Joannopoulos,RS4,Kuchment:12,Kuchment:16}.  
Symmetries of the underlying Hamiltonian give rise to novel properties of the band structure. An important example is the band structure of the single electron model of graphene and its artificial analogues. Here,  $H_V=-\Delta+V$, with $V$ a real-valued potential with the symmetries of a hexagonal tiling of the plane. It is well-known that the band structure contains Dirac points, conical singularities at the intersections of dispersion surfaces which occur at the vertices (high-symmetry quasi-momenta) of the hexagonal Brillouin zone; see, for example, \cite{RMP-Graphene:09,FW:12,FLW-CPAM:17,BC:18}.
A consequence is the massless Dirac dynamics of wave-packets (quasi-particles) which evolve from initial data which are  spectrally localized near Dirac points \cite{RMP-Graphene:09,FW-CMP:14}. 

In this article, we consider a class of periodic Schr\"odinger operators on $\R^2$,  whose underlying period lattice is $\Z^2$ and such that the potential is real, inversion symmetric, and invariant under $\pi/2$ rotation. We call such potentials {\it admissible}; see Definition \ref{def:sq-pot}.    The class of potentials to which our results apply includes those which are superpositions of localized potentials (say potential wells or potential barriers) centered on a discrete structure with the appropriate symmetries. Two such examples are illustrated in  Figure \ref{fig:Lattices}; the square lattice (left) and the Lieb lattice (right) are displayed together with corresponding choices of fundamental cells.
We call these two types of potentials {\it square lattice potentials} and {\it Lieb lattice potentials}; see Examples~\ref{square-pot} and \ref{lieb-pot} and the potentials in Figures \ref{figure11}-\ref{figure10}.

 Our goal is to study symmetry-induced characteristics in the band structure of such operators and to explore these in the context of the above two examples. %
 The present results are the  $\Z^2-$ analogue of results obtained
 in honeycomb structures  \cite{Colin-de-Verdiere:91,Grushin:09,FW:12,FLW-MAMS:17,Lee:16,BC:18}. Our proofs make use of the framework developed in \cite{FW:12,FLW-MAMS:17}.\medskip
 
For the class of potentials we consider, the nature of band degeneracies at high-symmetry quasi-momenta is described by 
two intersecting dispersion surfaces which are locally characterized by a normal form which is $\pi/2-$rotationally invariant, and to leading order homogeneous of degree two. A consequence is that the dynamics of wave-packets, which are spectrally concentrated near such high-symmetry quasi-momenta, is given by an effective system of coupled linear time-dependent Schr\"odinger equations with indefinite effective mass tensor.  This is in contrast to the case of honeycomb structures where the band degeneracies at the high symmetry points are conical (so-called {\it Dirac points}) and the effective dynamics is given by a time-dependent system of Dirac equations.

\subsection{A quick review of Floquet-Bloch theory;\  \cite{RS4,Kuchment:12,Kuchment:16,Eastham:74,Joannopoulos}}\label{FBquick}

Consider the periodic Schr\"odinger operator, $H_V=-\Delta +V$, where $V$ is real-valued and periodic with respect to a lattice $\Lambda=\Z\bv_1\oplus\Z\bv_2$;  for all $\bx\in\R^2$ and $\bv\in\Lambda$, we have $V(\bx+\bv)=V(\bx)$. 
Introduce the dual lattice, $\Lambda^*=\Z\bk_1\oplus\Z\bk_2$, 
such that $\bk_l\cdot\bv_m=2\pi\delta_{lm}$, and  spaces of $\Lambda-$ periodic and $\bk-$ pseudo-periodic functions:
\[L^2(\R^2/\Lambda)\ =\ \Big\{  f\in L^2_{\rm loc}(\R^2)\ :\  f(\bx+\bv)=f(\bx)\quad \textrm{almost everywhere in $\bx$, for all $\bv\in\Lambda$ }\ \Big\}
\]
and 
\[
L^2_\bk\ =\ \Big\{  f\in L^2_{\rm loc}(\R^2)\ :\ e^{-i\bk\cdot\bx}f(\bx)\in L^2(\R^2/\Lambda)\ \Big\} .
\]
Functions $f\in L^2_\bk$ satisfy $f(\bx+\bv)=e^{i\bk\cdot\bx}f(\bx)$ almost everywhere in $\bx$ for all $\bv\in\Lambda$. 
 The inner product on $L^2(\R^2/\Lambda)$ is given by $\left\langle f,g\right\rangle_{L^2(\R^2/\Lambda)}=\int_\Omega \overline{f}g$, where 
 $\Omega\subset\R^2$ is a period cell. Since $f, g\in L^2_\bk$ implies that $\overline{f}g\in L^1(\R^2/\Lambda)$, the same expression defines an inner product on $L^2_\bk$ .

Floquet-Bloch states are solutions of the self-adjoint $\bk-$pseudo-periodic eigenvalue problem:
\begin{equation}
H_V\Phi=\mu\Phi,\quad \Phi\in L^2_\bk\ .
\label{k-evp}
\end{equation}
 Since this boundary condition satisfied by $f\in L^2_\bk$ is invariant under $\bk\mapsto\bk+\widetilde\bk$, 
 for any $\widetilde\bk\in\Lambda^*$, we may restrict  to $\bk$ varying over a fundamental period cell in the dual variable. This {\it Brillouin zone}, $\mathcal{B}$, is often taken to be the set of all points in $\bk\in\R^2$ which are closer to the origin
 than to any other point in $\Lambda^*$. For the case of the square lattice, $\Lambda=\Z^2$
  \begin{equation}\label{brillouin}
  \brill = \{ \bk = (k^{(1)}, k^{(2)})\colon -\pi \leq k^{(1)} \leq \pi, -\pi \leq k^{(2)} \leq \pi\}\ ;
\end{equation}
see Figure \ref{fig:squareBZ}.
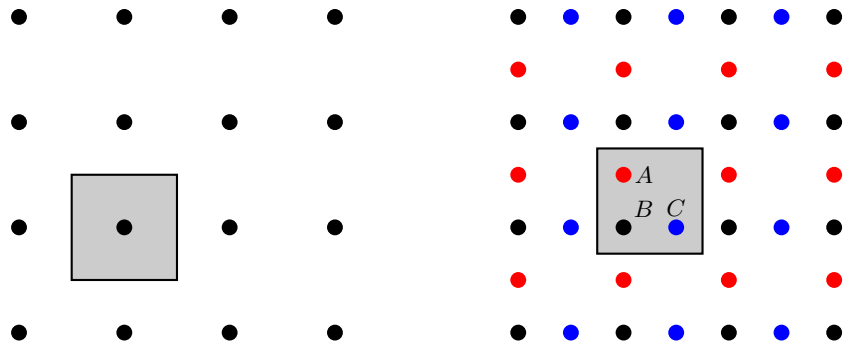
\begin{figure}
\begin{center}
\begin{tikzpicture}[thick,scale=0.7]
\tikzset{minimum size=6pt}
    \draw[fill=black!20!white] (1,1) -- (1,3) -- (3,3) -- (3,1) -- cycle;   

    \draw (0,0) node[circle,fill,inner sep=1pt] (A){};
    \draw (2,0) node[circle,fill,inner sep=1pt] (B){};
    \draw (4,0) node[circle,fill,inner sep=1pt] (C){};
    \draw (6,0) node[circle,fill,inner sep=1pt] (D){};
      
    \draw (0,2) node[circle,fill,inner sep=1pt] (E){};
    \draw (2,2) node[circle,fill,inner sep=1pt] (F){};
    \draw (4,2) node[circle,fill,inner sep=1pt] (G){};
    \draw (6,2) node[circle,fill,inner sep=1pt] (H){};

    \draw (0,4) node[circle,fill,inner sep=1pt] (I){};
    \draw (2,4) node[circle,fill,inner sep=1pt] (J){};
    \draw (4,4) node[circle,fill,inner sep=1pt] (K){};
    \draw (6,4) node[circle,fill,inner sep=1pt] (L){};

    \draw (0,6) node[circle,fill,inner sep=1pt] (M){};
    \draw (2,6) node[circle,fill,inner sep=1pt] (N){};
    \draw (4,6) node[circle,fill,inner sep=1pt] (O){};
    \draw (6,6) node[circle,fill,inner sep=1pt] (P){};
    

        
\end{tikzpicture}
\qquad  \qquad \qquad 
\begin{tikzpicture}[thick,scale=0.7]
\tikzset{minimum size=6pt}
    \draw[fill=black!20!white] (1.5,1.5) -- (1.5,3.5) -- (3.5,3.5) -- (3.5,1.5) -- cycle;   

    \draw (0,0) node[circle,fill,inner sep=1pt] (A){};
    \draw (2,0) node[circle,fill,inner sep=1pt] (B){};
    \draw (4,0) node[circle,fill,inner sep=1pt] (C){};
    \draw (6,0) node[circle,fill,inner sep=1pt] (D){};
      
    \draw (0,2) node[circle,fill,inner sep=1pt] (E){};
    \draw (2,2) node[circle,fill,inner sep=1pt,label={[label distance=-.11cm]above right:{\small $B$}}] (F){};
    \draw (4,2) node[circle,fill,inner sep=1pt] (G){};
    \draw (6,2) node[circle,fill,inner sep=1pt] (H){};

    \draw (0,4) node[circle,fill,inner sep=1pt] (I){};
    \draw (2,4) node[circle,fill,inner sep=1pt] (J){};
    \draw (4,4) node[circle,fill,inner sep=1pt] (K){};
    \draw (6,4) node[circle,fill,inner sep=1pt] (L){};

    \draw (0,6) node[circle,fill,inner sep=1pt] (M){};
    \draw (2,6) node[circle,fill,inner sep=1pt] (N){};
    \draw (4,6) node[circle,fill,inner sep=1pt] (O){};
    \draw (6,6) node[circle,fill,inner sep=1pt] (P){};
    

    
    \draw (1,0) node[circle,fill=blue,inner sep=1pt] (A2){};
    \draw (3,0) node[circle,fill=blue,inner sep=1pt] (B2){};
    \draw (5,0) node[circle,fill=blue,inner sep=1pt] (C2){};

    \draw (1,2) node[circle,fill=blue,inner sep=1pt] (E2){};
    \draw (3,2) node[circle,fill=blue,inner sep=1pt,label={[label distance=-.1cm]above:{\small $C$}}] (F2){};
    \draw (5,2) node[circle,fill=blue,inner sep=1pt] (G2){};

    \draw (1,4) node[circle,fill=blue,inner sep=1pt] (I2){};
    \draw (3,4) node[circle,fill=blue,inner sep=1pt] (J2){};
    \draw (5,4) node[circle,fill=blue,inner sep=1pt] (K2){};

    \draw (1,6) node[circle,fill=blue,inner sep=1pt] (M2){};
    \draw (3,6) node[circle,fill=blue,inner sep=1pt] (N2){};
    \draw (5,6) node[circle,fill=blue,inner sep=1pt] (O2){};

    \draw (0,1) node[circle,fill=red,inner sep=1pt] (A3){};
    \draw (2,1) node[circle,fill=red,inner sep=1pt] (B3){};
    \draw (4,1) node[circle,fill=red,inner sep=1pt] (C3){};
    \draw (6,1) node[circle,fill=red,inner sep=1pt] (D3){};
      
    \draw (0,3) node[circle,fill=red,inner sep=1pt] (E3){};
    \draw (2,3) node[circle,fill=red,inner sep=1pt,label={[label distance=-.1cm]right:{\small $A$}}] (F3){};
    \draw (4,3) node[circle,fill=red,inner sep=1pt] (G3){};
    \draw (6,3) node[circle,fill=red,inner sep=1pt] (H3){};

    \draw (0,5) node[circle,fill=red,inner sep=1pt] (I3){};
    \draw (2,5) node[circle,fill=red,inner sep=1pt] (J3){};
    \draw (4,5) node[circle,fill=red,inner sep=1pt] (K3){};
    \draw (6,5) node[circle,fill=red,inner sep=1pt] (L3){};
    
\end{tikzpicture}
\caption{ An illustration of the square {\bf (left)} and Lieb  {\bf(right)}  lattices with corresponding fundamental cells, $\Omega$,  (shaded). The Lieb lattice is the union of $3$ interpenetrating sublattices, labeled $A$ (red), $B$ (black), and $C$ (blue). } \label{fig:Lattices}
  \end{center}
\end{figure}

An alternative formulation is to  write $\Phi(\bx)=e^{i\bk\cdot\bx}\phi(\bx)$ and seek, for all $\bk\in\mathcal{B}$, solutions of the self-adjoint periodic  eigenvalue problem:
\begin{equation}
 H_V(\bk)\phi(\bx)\equiv \left(\ -(\nabla_\bx+i\bk)^2\ +\ V(\bx)\ \right)\phi\ =\ \mu \phi(\bx),\qquad \phi\in L^2(\R^2/\Lambda).
 \label{per-evp}
 \end{equation}
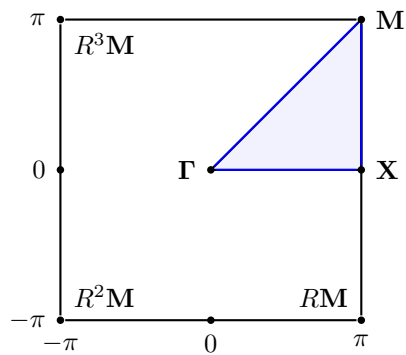
\begin{figure}[t!]
\begin{center}
\begin{tikzpicture}[thick,scale=1.0]
\tikzset{minimum size=1pt}
    \draw [fill=blue!5!white] (2,2) -- (4,2) -- (4,4) -- (2,2);        

    \draw (0,0) node[circle,fill,inner sep=1pt,label=below:{$-\pi$},label=left:{$-\pi$},label=above right:{$R^2{\bf M}$}] (A){};
    \draw (2,0) node[circle,fill,inner sep=1pt,label=below:{$0$}] (B){};
    \draw (4,0) node[circle,fill,inner sep=1pt,label=below:{$\pi$},label=above left:{$R{\bf M}$}] (C){};
      
    \draw (0,2) node[circle,fill,inner sep=1pt,label=left:{$0$}] (E){};
    \draw (2,2) node[circle,fill,inner sep=1pt,label=left:{${\bf \Gamma}$}] (F){};
    \draw (4,2) node[circle,fill,inner sep=1pt,label=right:{${\bf X}$}] (G){};

    \draw (0,4) node[circle,fill,inner sep=1pt,label=left:{$\pi$},label=below right:{$R^3{\bf M}$}] (I){};
    \draw (4,4) node[circle,fill,inner sep=1pt,label=right:{${\bf M}$}] (K){};
    
    \draw (A) -- (C) -- (K) -- (I) --(A) ;   
    \draw [blue] (F) -- (G) -- (K) -- (F);        
\end{tikzpicture}
\end{center}
\caption{Brillouin zone, $\brill$, for the square lattice $\Gamma = \mathbb Z^2$. The points ${\bf \Gamma}=(0,0)$, ${\bf M} = (\pi,\pi)$, and ${\bf X} = (\pi,0)$ are labeled, along with successive $\pi/2$ rotations of ${\bf M}$ by $R$; see \eqref{Rdef}. An irreducible Brillouin zone is shaded in blue. In later figures, the graphs of dispersion relation are plotted along the cyclic path ${\bf \Gamma} \to {\bf X} \to {\bf M} \to {\bf \Gamma}$. }
\label{fig:squareBZ}
\end{figure}
 For each $\bk\in\mathcal{B}$, the self-adjoint elliptic eigenvalue problem \eqref{per-evp} has a discrete set of eigenpairs $(\mu_b(\bk),\phi_b(\bx;\bk))$, $b=1,2,3,\dots$, where the eigenvalues may be listed with multiplicity, in order:
 \[ \mu_1(\bk)\le\mu_2(\bk)\le\cdots\le\mu_b(\bk)\le\cdots\ ,\]
 where $\{\phi_b(\bx;\bk)\}_{b\ge1}$  can be taken to be a complete orthonormal sequence in $L^2(\R^2/\Lambda)$.  Moreover, the family of states $\Phi_b(\bx;\bk)=e^{i\bk\cdot\bx}\phi_b(\bx;\bk)$ where $b\ge1$ and $\bk\in\mathcal{B}$ are complete in $L^2(\R^2)$. The eigenvalue mappings $\mu_b\colon \mathcal{B}\to\R$ are Lipschitz continuous \cite{avron-simon:78,FW-CMP:14}
 and are called {\it dispersion relations}, and their graphs are called  {\it dispersion surfaces}. The collection of all dispersion relations   is called the {\it band structure} of the periodic Schr\"{o}dinger operator, $H_V$. The spectrum of $H_V$ acting in $L^2(\R^2)$ is the union of the closed real intervals: $\sigma(H_V)=\mu_1(\mathcal{B})\cup \mu_2(\mathcal{B})\cup\cdots \cup \mu_b(\mathcal{B})\cup\cdots$.  
 
 \subsection{Summary of results}\label{sec:summ}

We  summarize our  main results:
\begin{enumerate} 
\item {\it Theorem \ref{quad-disp}}:
\subitem(A)\ We present general conditions on the admissible potential, $V$,  for the following scenario:
 there exists  $\mu_S\in\R$, such that for all vertices, $\bM_\star$  of $\brill$,  $\mu_S$ is an $L^2_{\bM_\star}-$ eigenvalue  of $H_V$ of geometric multiplicity two, with $L^2_\bM-{\rm kernel}(H_V-\mu_S I)={\rm span}\{\Phi_1,\Phi_2\}$. 
\subitem(B)\ There exist dispersion maps $\bk\mapsto\mu_-(\bk)$ and $\bk\mapsto\mu_+(\bk)$, defined for $\bk\in\brill$ with  $\mu_-(\bk)\le \mu_+(\bk)$ and such that for all vertices, $\bM_\star$, of $\brill$
 we have $\mu_\pm(\bM_\star)=\mu_S$; the dispersion surfaces associated with $\mu_-$ and $\mu_+$ touch at the energy / quasi-momentum pairs $(\bM_\star,\mu_S)$. 
\subitem(C)\ The two touching dispersion surfaces are locally described  by a normal form which is $\pi/2-$rotationally invariant and, at leading order in $|\bk-\bM_\star|$, homogeneous of degree two. In particular, for $|\bkappa|=|(\kappa_1,\kappa_2)|=\sqrt{\kappa_1^2+\kappa_2^2}$ small:
  \begin{equation} 
\mu_{\pm}(\bM+\bkappa) - \mu_S = 
  (1-\alpha)|\kappa|^2\ +\ \mathscr{Q}_6(\kappa)\ \pm \sqrt{\Big|\ \gamma(\kappa_{1}^{2} - \kappa_{2}^{2})+ 2\beta \kappa_{1}\kappa_{2}\ \Big|^2\ +\ \mathscr{Q}_8(\kappa)}\ , \label{nf} 
\end{equation}
where $\alpha\in \mathbb{R}$ and $\beta, \gamma \in \mathbb{C}$ are constants, and 
 $\mathscr{Q}_6(\kappa)=\mathcal{O}(|\kappa|^6)$ and $\mathscr{Q}_8(\kappa)=\mathcal{O}(|\kappa|^8)$ are analytic functions of  $\kappa$ and invariant under $\pi/2-$rotation: $(\kappa_1,\kappa_2)\mapsto(-\kappa_2,\kappa_1)$.
\subitem(D)\ Corollary \ref{rho-inv}:\ If  $V$ is, in addition, assumed to be \emph{reflection-invariant} with respect to the diagonal of the fundamental cell, $\Omega$ (see Figure \ref{fig:Lattices}), then  \eqref{nf} reduces to
\begin{align}
\mu_{\pm}(\bM+\bkappa) - \mu_S &= 
(1-\alpha)|\kappa|^2+\mathscr{Q}_6(\kappa)\ \nn\\
&\quad  \pm \sqrt{  \tilde{\gamma}^2(\kappa_1^2-\kappa_2^2)^2 + 4\beta^2 \kappa_{1}^2\kappa_{2}^2\ +\ \mathscr{Q}_8(\kappa) },
\label{nf-ref}\end{align}
 where $\alpha,\beta\in\R$ and $\tilde\gamma\in\R$, and where $\mathscr{Q}_6(\kappa)$ and $\mathscr{Q}_8(\kappa)$ are now also invariant under the reflection: $(\kappa_1,\kappa_2)\mapsto(\kappa_2,\kappa_1)$.

\item {\it Theorem \ref{mu-epsilon-evals}:}\ Consider $H^\eps\equiv H_{\eps V}=-\Delta +\eps V$, where $V$ is admissible and $\eps$ is real. Let $V_{1,1}$ and $V_{1,0}$ denote the $(1,1)$ and $(1,0)$ indexed Fourier coefficients of  $V$ (see \eqref{FS}) and assume the (generically satisfied) non-degeneracy condition: $V_{1,1}\ne V_{1,0}$. Then, for all non-zero and sufficiently small $\eps$,  there are  $2$ dispersion surfaces of $H^\eps$, among the lowest $4$,  that touch at the vertices of $\brill$. In a neighborhood of each vertex,  the local character given in \eqref{nf} or \eqref{nf-ref}. 
\item  
 {\it Theorem \ref{eps-generic}}:\ There exists a discrete set $\widetilde{\mathcal C}\subset\R$,  such that if $\eps\notin\widetilde{\mathcal C}$, then $2$ dispersion surfaces of $H^\eps$ touch at the vertices of $\brill$ with local behavior described by \eqref{nf}. 
The constants  $\alpha^\eps\in\R,$ $\beta^\eps \in \C$, and $\gamma^\eps \in \C$ in Theorem \ref{mu-epsilon-evals}, and 
$\alpha^\eps$ and $\beta^\eps$ in Theorem \ref{eps-generic}, displayed in \eqref{albega}, depend on the degenerate eigenspace, ${\rm span}\{\Phi_1^\eps,\Phi_2^\eps\}$, for quasi-momentum. %
 Hence, the property of quadratically touching dispersion surfaces with local behavior given by \eqref{nf-ref} holds for generic, even arbitrarily large, values of $\eps$.
\item {\it Lieb lattice potentials: tight-binding vs strong binding}: The well-known conical plus flat band structure dispersion surfaces of the 3-band tight-binding model for the Lieb lattice \underline{does not} persist in the Schr\"odinger operator with periodic Lieb lattice potential consisting of potential wells centered on the Lieb lattice; see Section \ref{LiebLattice}.
\item {\it Numerical studies:}\  In Section \ref{numerical} we discuss numerical simulations for various admissible potentials for a full range of coupling parameters, $\eps$ small to $\eps$ large. These include 
 potentials which are superpositions of spatially localized potentials, centered on vertices of the square lattice or the vertices of a  Lieb lattice. Our simulations corroborate our analytical results  and are discussed in this context. 
\item {\it Wave-packet dynamics; Appendix \ref{sec:wave-pkts}:} A multiscale analysis demonstrates that the envelope of the solution of the time-dependent Schr\"{o}dinger equation, $i\D_t\psi(\bx,t)=(-\Delta_\bx+V(\bx))\psi(\bx,t)$, for wave-packet initial data:
$ \psi(\bx,0)=\ C_{10}(\bX)\ \Phi_1(\bx)\ +\ C_{20}(\bX)\ \Phi_2(\bx),\ {\rm where}\ \bX\equiv\delta\bx=(X_1,X_2)$, and $C_{j0}(\bX),\ j=1,2$ are in Schwartz class,
 evolves on large but finite time scales according
to a coupled system of Schr\"{o}dinger equations ($T=\delta^2 t$):
\begin{equation}
 i\frac{\D}{\D  T} C_{p}(\bX,T) =  
 -\sum_{q=1}^2\sum_{r,s=1}^2  
 \frac{\D}{\D X_r}\ \Upsilon^{p,q}_{r,s}\ \frac{\D}{\D X_s}\ C_q(\bX,T)  ,\quad p=1,2.
\label{eff-env}  \end{equation}
Here,  $ \Upsilon^{p,q}_{r,s} $ are the matrix elements of an (indefinite) inverse effective mass tensor;
 see Section \ref{pf-quad-disp}. The branches of the dispersion relation of \eqref{eff-env} are given by the expression in \eqref{nf}. \label{wave-pack-summ}

A derivation of \eqref{eff-env} is presented in Appendix \ref{sec:wave-pkts}. A rigorous proof of the long (finite) time validity would be along the lines of  \cite{allaire} or \cite{FW-CMP:14}, for example.
\end{enumerate}

\subsection{The Lieb lattice; tight-binding versus continuum models, a band structure instability} \label{LiebLattice}

Our original motivation for the present study was to contrast the band structure of the Lieb lattice tight-binding Hamiltonian (see, for example, \foreignlanguage{romanian}{Ni\cb{t}}\u{a} et al in \cite{nita-lieb}, Weeks and Franz in \cite{weeks-lieb}) with that of the corresponding continuum
Schr\"odinger operator with Lieb lattice potential in the high-contrast (strong binding) regime, as studied via simulation in Guzm\'{a}n-Silva et al in \cite{guzman-lieb}. 

\bigskip

\nit{\it  Our results demonstrate that the well-known flat-band plus conical structure near 
 the vertices of $\brill$ does not persist in the regime of strong-binding (finite deep atomic wells at each lattice site) for the continuum Schr\"odinger operator.}
 \bigskip
 
 \nit This is in contrast to the case of honeycomb structures, where the Dirac (conical) points of the tight-binding model of Wallace \cite{Wallace:47,RMP-Graphene:09}  persist in the continuum honeycomb Schr\"odinger operators \cite{FLW-CPAM:17,FW:12}. 
  We now explain this in some detail.
  
The  tight-binding model for the Lieb lattice is given by a Hamiltonian acting on wave functions $\psi\in l^2(\Z^2;\C^3)$;
for each $(m,n)\in\Z^2$, $\psi_{m,n}\in\C^3$ is a vector of $3$ complex amplitudes assigned to the $A, B$ and $C$ sites in the
$(m,n)^{th}$ cell; see Figure \ref{fig:Lattices} (right). This model has $\pi/2-$rotational invariance about any $B-$ site. The detailed setup  is presented in Appendix \ref{tb-lieb}. The band structure  has three band dispersion relations, whose graphs (dispersion surfaces)  intersect
at the vertices of the Brillouin zone, $\brill$. 
Figure \ref{fig:lieb-tb} reveals that the three dispersion surfaces which meet at the vertex, $\bM$, of $\brill$ consists of 
 a constant energy dispersion surface (a {\it flat band}, $\bk\mapsto E_0^{^{\rm TB}}(\bk)\equiv0$), and two other dispersion surfaces, $\bk\mapsto E_\pm^{^{\rm TB}}(\bk)$, which meet conically at $\bM$. An analogous picture applies in a neighborhood of each vertex of $\brill$, as  seen in the left plot of Figure \ref{fig:lieb-tb}. The latter observation follows from symmetry considerations. 
\begin{figure}
  \begin{center}
 \includegraphics[width=0.48\linewidth]{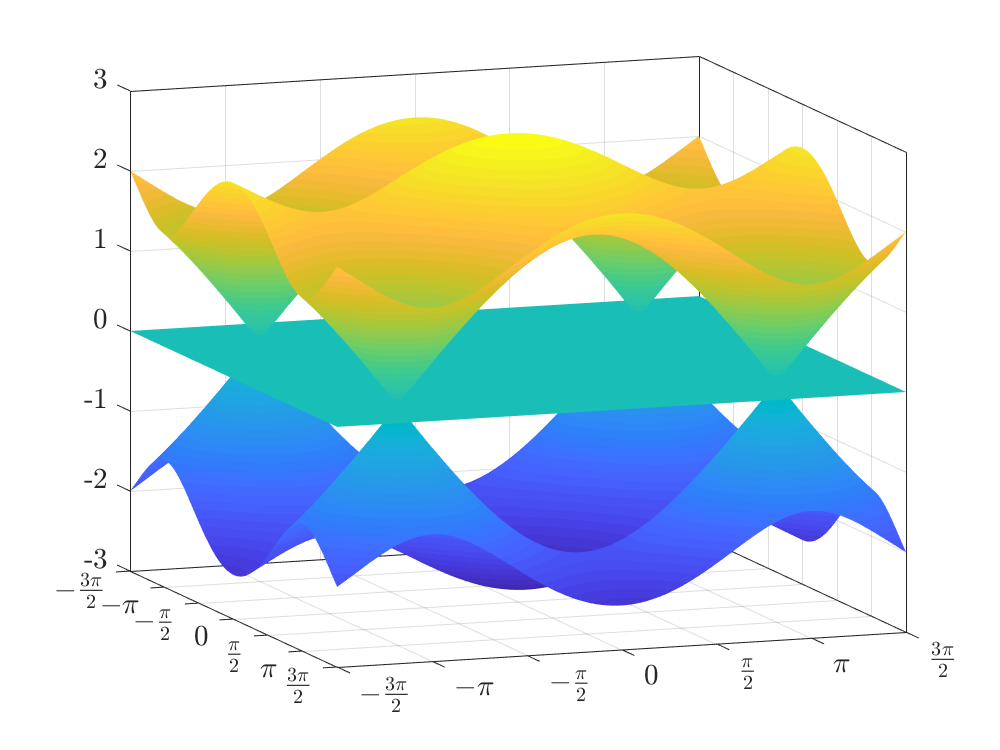}
 \includegraphics[width=0.48\linewidth]{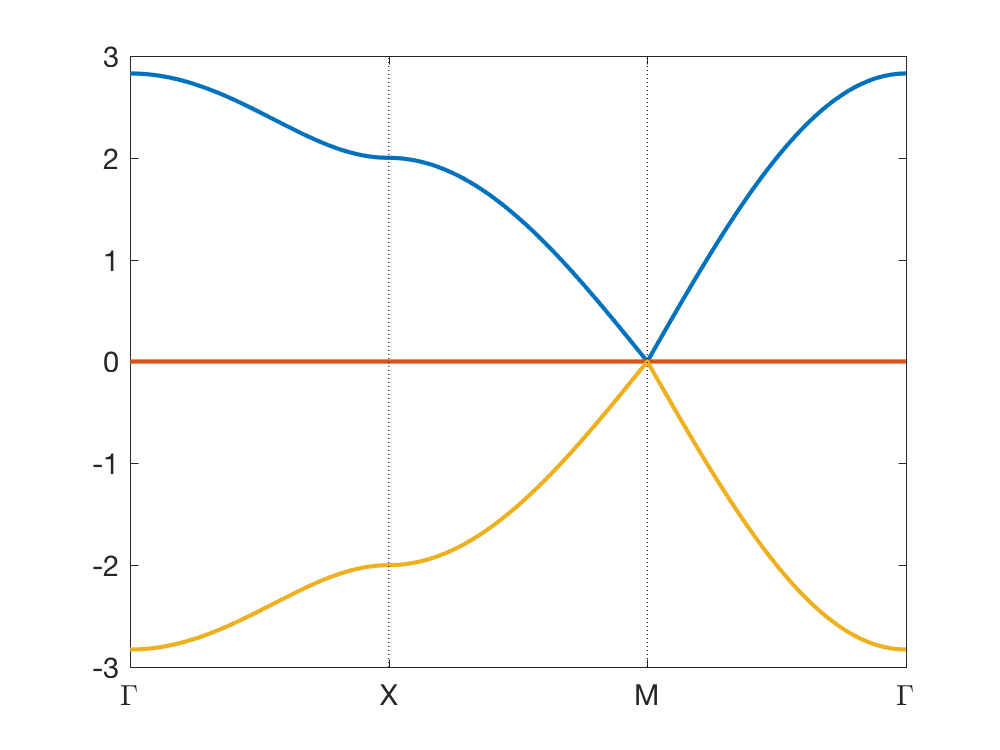}
   \caption{{\bf (left)} Dispersion surfaces of the 3-band tight-binding model for the Lieb lattice.   The Brillouin zone is $\left[-\pi, \pi \right]^2$ {\bf (right)}.  A plot of the dispersion surface along the path ${\bf \Gamma} \to  {\bf X} \to {\bf M} \to {\bf \Gamma}$, as described in Figure \ref{fig:squareBZ}. At the vertices of the Brillouin zone, the dispersion relation has the structure of a conical intersection and a flat band. Our analysis shows that this structure does not persist in the continuum Schr\"odinger operator which limits to the tight-binding model. } \label{fig:lieb-tb}
  \end{center}
\end{figure}

 It is natural to  contrast this behavior with that of the 
  two-band tight-binding model of P.R. Wallace (1947),  introduced in his pioneering work on graphite
   \cite{Wallace:47,RMP-Graphene:09}. In this model, there are two dispersion surfaces which meet conically at the high-symmetry quasi-momenta, located at the vertices of the (hexagonal) Brillouin zone. These conical points are called {\it Dirac points}.

The tight-binding model gives an approximation to the low-lying spectrum of the continuum 
Schr\"{o}dinger operator $-\Delta + \lambda^2 V$,  where $V$ is a superposition of  identical  potential wells, centered at the sites of a honeycomb structure and $\lambda$ is sufficiently large. It was proved in \cite{FLW-CPAM:17} for this {\it strong binding regime} that Dirac points occur at the vertices of the Brillouin zone and that, after a rescaling, the first two dispersion surfaces converge uniformly to those of Wallace's two-band tight-binding model. It had earlier been proved in \cite{FW:12,FLW-MAMS:17} that generic Schr\"{o}dinger operators, for the class of honeycomb lattice potentials, have Dirac points within their band structure and that these Dirac points  persist against small perturbations of the potential which break inversion or complex-conjugate symmetries; see Definitions \ref{time-rev}(ii) and \ref{even}(iii).

\medskip

\begin{question}
\label{tbquest}
\nit Consider a potential $V_L(\bx)$, formed as a superposition of identical deep potential wells centered at the points of the Lieb lattice. Does the local band structure near the vertices of $\brill$  (a flat band plus two conically touching surfaces) of the tight-binding model persist in the band structure of $-\Delta + V_L$, {\it i.e.} does this local structure persist into the strong binding regime? 
\end{question}
\medskip

\nit The answer is \underline{no}, and the precise character of the local band structure  is a consequence of our analysis of $\pi/2-$rotationally invariant potentials. Figure \ref{fig:lieb-curves1} displays the family of 3 curves obtained by sampling 3 dispersion surfaces of $-\Delta+V_L$,  two surfaces that touch at the vertices of $\brill$ and the nearest to these among all others, for a choice of deep atomic potential wells. 
The two curves which intersect are locally described by \eqref{nf}. 
\begin{figure}[t!]
\begin{center}
\includegraphics[width = 2.5in]{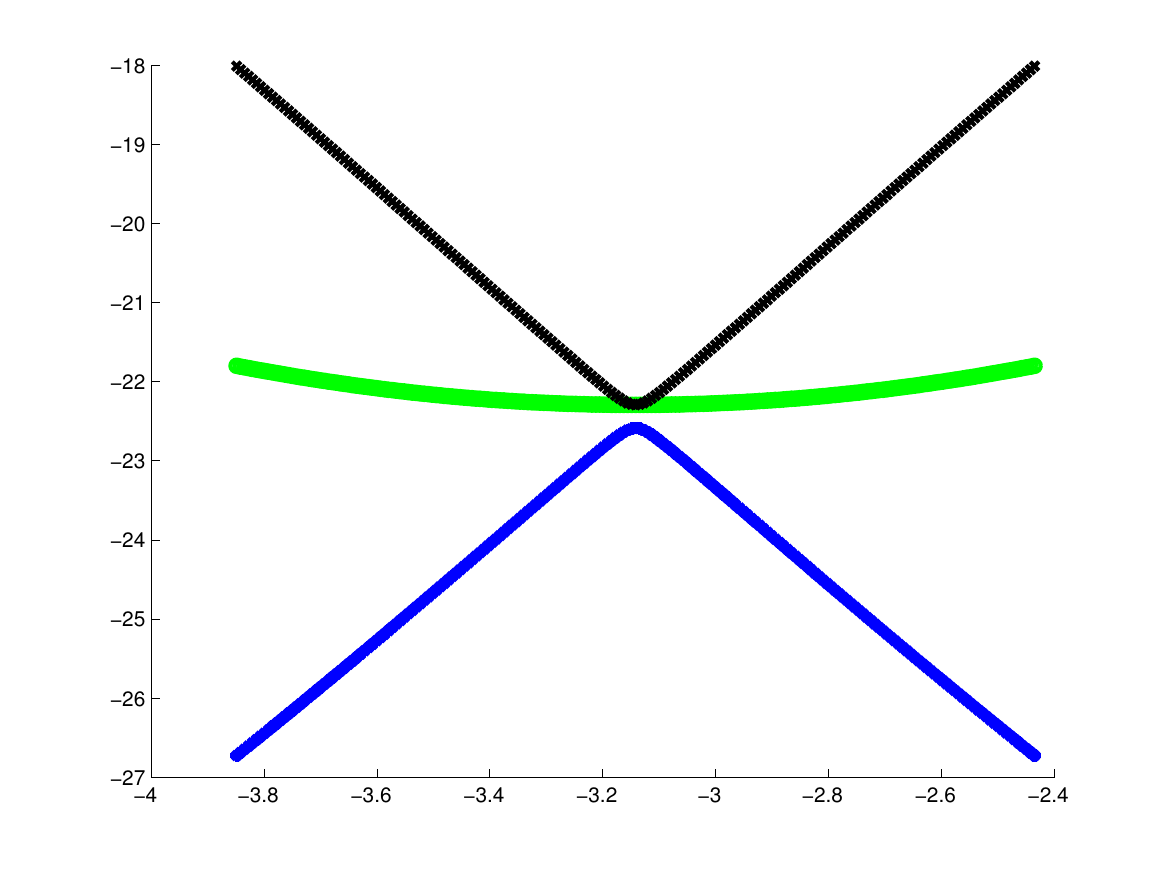}
\qquad \qquad 
\includegraphics[width = 2.5in]{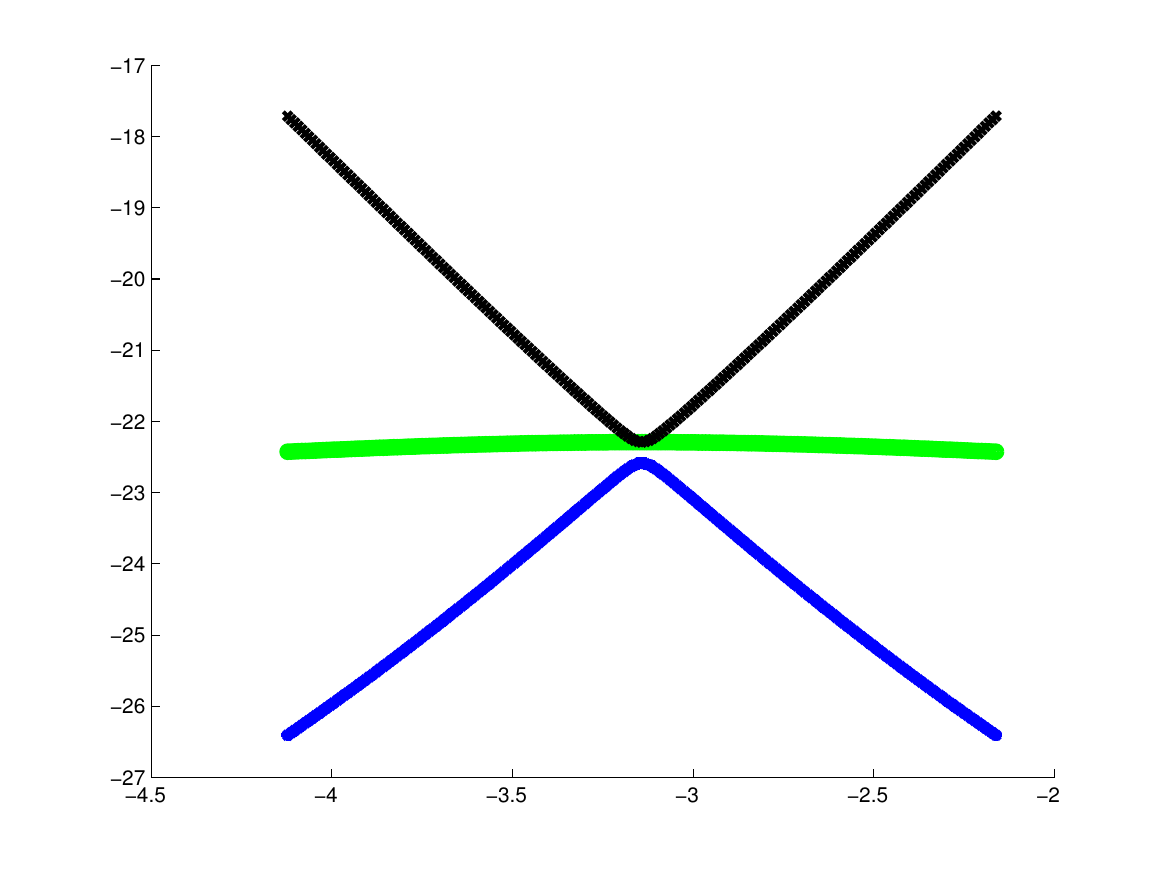}
\caption{Dispersion surfaces for $-\Delta+V_L$  sampled along a two different quasi-momentum segments through the high symmetry quasi-momentum, $\bM$. Here, $V_L$ is the periodic potential whose restriction to the primitive cell is:
$
	\tilde{V}_L (\bx) = -V_0 (e^{-\abs{\bx}^2/\sigma} + e^{-\abs{\bx - (1/2, 0)}^2/\sigma} + e^{-\abs{\bx - (0, 1/2)}^2/\sigma} ) .
$
Segments are of the form  
$\bk = \bM \pm \lambda ( \cos \theta,  \sin \theta )^t$ 
for $\lambda \in (-\lambda_0,\lambda_0)$ with $\theta = \frac{\pi}{4}$ {\bf (left)} and $\theta = \frac{15\pi}{16}$ {\bf (right)}. }
\label{fig:lieb-curves1}
\end{center}
\end{figure}

\begin{remark}\label{allaire}
In \cite[Section 6]{allaire} results on homogenized effective equations were obtained for the dynamics of wave packets, which are spectrally localized near an isolated (quadratic) spectral band edge.
The authors also consider both the cases of simple and degenerate eigenvalues occurring at a band edge, derive a system of coupled Schr\"{o}dinger envelope equations, and remark on the non-genericity of degenerate eigenvalues at a band edge.
Although non-generic in the space of all potentials, the results of this article show in the presence of symmetry,
for example the symmetries of our $\mathcal{P}\circ\mathcal{C}$ and $\pi/2-$rotationally invariant (admissible) potentials, such 
systems are realized in many physical settings of interest.
\end{remark}

\subsection{Outline of the paper}\label{outline}
In Section~\ref{s:2}, we define the class of {\it admissible potentials} considered here, characterize the   Fourier series of such potentials, and develop tools for Fourier analysis in the spaces, $L^2_\bk$ for $\bk\in\brill$. In particular, $L^2_\bM$ has direct sum (orthogonal) decomposition into eigenspaces, $L^2_{\bM,\sigma}$ ($\sigma=\pm1,\pm i$),    
of the $\pi/2-$rotation operator, $\mathcal{R}$, which acts unitarily on $L^2_\bM$.\medskip

\nit In Section~\ref{sec:freeH} we study spectral properties of the free Laplacian, $-\Delta$, in $L^2_\bM$. In particular, we show that $-\Delta$ has four-fold degenerate $L^2_\bM$ eigenvalues (we consider the least energy such), which correspond to four simple eigenvalues, one in each of the invariant subspaces, $L^2_{\bM,\sigma}$ .\medskip

\nit In Section \ref{2fold-touch} we state and prove Theorem  \ref{quad-disp} and Corollary \ref{rho-inv} which give  sufficient conditions for the quadratic touching of two dispersion surfaces at the high-symmetry quasi-momenta, situated at the vertices, $\bM_\star$, of $\brill.$  These results also
display the detailed expansion of the dispersion maps in a neighborhood of these quasi-momenta.\medskip

\nit In Section \ref{sec:eps-small} we verify the hypotheses of Theorems  \ref{quad-disp} and Corollary \ref{rho-inv} for Hamiltonians of the form $H^\eps=-\Delta+\eps V$, where $V$ is admissible, for all $\eps$ real, sufficiently small and non-zero. This then proves the existence of quadratic degeneracies at high-symmetry quasi-momenta as described in Theorem  \ref{quad-disp} and Corollary \ref{rho-inv} for the regime of sufficiently weak-coupling (low-contrast). \medskip
 
 \nit In Section \ref{eps-generic}, we consider $H^\eps=-\Delta+\eps V$, where $\eps$ is real and non-zero, but not restricted by size. Theorem \ref{generic-eps} states that the small $\eps-$results persist
  for all non-zero real $\eps$ except for a possible discrete set of $\eps-$values. The proof, based on arguments in \cite{FW:12,FLW-MAMS:17}, is sketched.  \medskip

\nit Section~\ref{numerical} summarizes a range of  computational experiments  that corroborate our analytical results. \medskip

\nit Appendix \ref{tb-lieb} summarizes the formulation of the tight-binding for the Lieb lattice;
 see Section \ref{LiebLattice}.  \medskip

\nit  Appendix \ref{sec:wave-pkts} summarizes a  multiple scales asymptotic expansion yielding a coupled system of time-dependent Schr\"odinger equations, which govern the evolution of wave-packets, which are spectrally concentrated in a neighborhood of high-symmetry quasi-momenta, $\bM_\star\in\brill$. \medskip
 
\nit Finally, in Appendix \ref{mu-eps-coeff} we provide the details of a calculation of coefficients occurring in the leading order expression for the touching dispersion surfaces in the weak coupling ($\eps$ non-zero and small) regime; see also Corollary \ref{small-eps-disp}.

\subsection{Notation}\label{oft-used} 

We remark on some frequently used notation and conventions.
\begin{enumerate}
\item Repeated indices are understood to be summed unless otherwise specified.
\item For $\bfm\in\Z^2$, $\bfm\vec{\bk}= m_1\bk_1+m_2\bk_2$, where $\bk_1$ and $\bk_2$ are the dual lattice basis vectors 
defined below in \eqref{k12-def}.
\item $\left\langle f,g\right\rangle_{L^2(\R^2/\Lambda)}=\left\langle f,g\right\rangle$; if not indicated
the inner product is understood to be taken in $L^2(\R^2/\Lambda)$. 
 \end{enumerate}

\medskip

\noindent {\bf Acknowledgments:} The authors thank M. Rechtsman and C.L. Fefferman for very stimulating discussions. M. Rechtsman also provided very helpful advice on numerical spectral calculations. M.I. Weinstein would like to thank the Department of Mathematics of Stanford University for its excellent hospitality and environment during their Winter, 2018 quarter, when part of this work was done as Bergman Visiting Professor.
This research was supported in part by the following US National Science Foundations grants: Graduate Research Fellowship Program grants:  DGE-1144155 and DGE 16-44869 (RTK);  DMS-1312874 and NSF CAREER Grant DMS-1352353 (JLM); DMS 16-19755 (BO); DMS-1412560, DMS-1620418 and RNMS Grant \#11-07444 (Ki-Net) (RTK\&MIW). This research was also supported in part by Simons Foundation Math + X Investigator Award \#376319 (MIW).

\section{The Lattice $\Lambda$  and Admissible Potentials} \label{s:2}
\subsection{The Square Lattice}

We begin with the lattice $\Lambda=\Z\bv_1\oplus\Z\bv_2$, where
\begin{equation} 
  \bv_{1} = a\begin{pmatrix}1\\0\end{pmatrix} \text{ and } \; \bv_{2} =a\begin{pmatrix}0\\1\end{pmatrix}.
\label{v12-def}
\end{equation}
The constant, $a$,  is  the lattice constant giving the distance between nearest neighbor sites.
  The dual lattice, $\Lambda^*$, is generated by the vectors 
\begin{equation}
  \bk_{1} = a^{-1} \begin{pmatrix}2\pi\\0\end{pmatrix}\text{ and }  \; \bk_{2} = a^{-1} \begin{pmatrix}0\\2\pi\end{pmatrix}.
\label{k12-def}
\end{equation}
For simplicity, we assume $a=1$. The Brillouin zone, $\brill$, is given in \eqref{brillouin}; see Figure \ref{fig:squareBZ}.
We have the relations $  \bk_{l}\cdot\bv_{m} = 2\pi \delta_{l,m}$, for $l,m=1,2$.

 Let $R$ be the $\pi/2-$clockwise rotation matrix
\begin{equation}
R = \begin{pmatrix} 0 & 1\\ 
-1 & 0 
\end{pmatrix}.
\label{Rdef}\end{equation}
We record the elementary relations:  $R^*\bv_1=\bv_2$  and $R^*\bv_2=-\bv_1$, and 
$R\bk_1=-\bk_2$ and $R\bk_2=\bk_1$. It follows that $R$ and $R^*$ map $\Lambda$ to itself and $\Lambda^*$ to itself. 
Denote the vertices of $\brill$ by: 
\[\bM\equiv\bM_{_{++}}  = (\pi, \pi)^T,\quad  \bM_{_{+-}} = (\pi, -\pi)^T,\quad  \bM_{_{--}} = (-\pi, -\pi)^T,\quad \bM_{_{-+}} =  (-\pi, \pi)^T \ . \]
Then, the set of vertices of $\brill$ is mapped by $R$ to itself, 
\[ \bM_{_{+-}} =R\bM=\bM-\bk_2,\quad \bM_{_{--}}=R^2\bM=\bM-\bk_1-\bk_2,\quad \bM_{_{-+}}=R^3\bM\  =\ \bM-\bk_1\ . \]
Furthermore,  $R$ maps the affine sublattice $\bM+\Lambda^*$ one to one and onto itself. See Figure \ref{fig:squareBZ}. 

\begin{remark}\label{allvertices}
Note that the pseudo-periodic boundary condition associated with  quasi-momenta located at the vertices of $\brill$ are 
the same and correspond to $\bM$-pseudo-periodicity, {\it i.e.} for any choice of $(a,b)\in\{+,-\}$, we have 
$\psi(\bx+\bv;\bM_{ab})=e^{i\bM\cdot\bv}\psi(\bx;\bM_{ab}),\quad \bx\in\R,\ \  \bv\in\Lambda$ 
 Therefore, 
for any $\bM_\star\in \bM+\Lambda$, the space $L^2_{\bM_\star}$ can be identified with $L^2_{\bM}$. Furthermore, 
the local character of dispersion surfaces in a neighborhood of any vertex of $\brill$ determines the local character in a neighborhood of any other vertex of $\brill$.
\end{remark}

\subsection{Admissible Potentials}\label{admissible}

For any function $f$ defined on $\R^2$, we define the $\pi/2-$rotational operator 
\[\mathcal{R}[f](\bx) \equiv f(R^{*}\bx),\] where $R$ is the $\pi/2-$clockwise rotation matrix displayed in \eqref{Rdef}.

 We consider smooth (say $C^\infty$) periodic potentials $V(\bx)=V(x_1,x_2)$ defined on $\R^2$, with fundamental period cell $\Omega=[0,1] \times [0,1]$.
Any such $V$ can be represented as a Fourier series:
 \begin{equation}
 V(\bx)\ =\ \sum_{\bfm\in\Z^2} V_\bfm e^{i\bfm\vec{\bk}\cdot\bx}\ =\ \sum_{(m_1,m_2)\in\Z^2}\ V_{m_1m_2}e^{2\pi i(m_1x_1+m_2x_2)},
\label{FS}\end{equation}
where $V_\bfm=(2\pi)^{-2}\int_\Omega e^{-i\bfm\vec\bk\cdot\bx}V(\bx)d\bx$, and  $\bfm\vec{\bk}= m_1\bk_1+m_2\bk_2$.

%

%
\begin{newdef} [$\mathcal{P}$, $\mathcal{C}$ and $\mathcal{R}$ invariance]\label{invariance}
\begin{enumerate}
\item Given a point $\bx\in\R^2$, its $\pi/2-$counterclockwise rotation about $\bx_c$, denoted $\widehat{\bx_{\mathcal{R}}}$, 
satisfies: $ \widehat{\bx_{\mathcal{R}}}-\bx_c= R^*\left(\bx-\bx_c\right)$. If
 \begin{equation}
  \mathcal{R}[V](\bx)\ \equiv\ V(\widehat{\bx_{\mathcal{R}}})=V(\bx)
  \label{Rdef}\end{equation}
 
  for all $\bx\in\R^2$ we say that $V(\bx)$ is $\mathcal{R}-$invariant, or $\pi/2-$rotationally invariant, with respect to $\bx_c$.
\item Given a point $\bx\in\R^2$, its inversion with respect to $\bx_c$, denoted $\widehat{\bx_{\mathcal{I}}}$, satisfies:
$ \widehat{\bx_{\mathcal{I}}}-\bx_c= -\left(\bx-\bx_c\right)$. If $V(\widehat{\bx_{\mathcal{I}}})=V(\bx)$ for all $\bx\in\R^2$ we say that $V(\bx)$ is $\mathcal{P}$-invariant, parity invariant, or inversion symmetric  with respect to $\bx_c$.

\item We say that $V(\bx)$ is $\mathcal{C}$-invariant or invariant under complex conjugation if $\overline{V(\bx)}=V(\bx)$ for all $\bx\in\R^2$.
\end{enumerate}
\end{newdef}

We shall study potentials which are real-valued, smooth and invariant under  $\mathcal{P}\circ\mathcal{C}$ and $\mathcal{R}-$invariant (invariant under $\pi/2$ rotation). We call such potentials {\it admissible}. 

\begin{newdef}[Admissible potentials]
\label{def:sq-pot}
An admissible potential is a smooth function, $V(\bx)$, defined on $\R^2$ with the following properties.
\begin{enumerate}
\item[i.] $\Lambda$-periodicity:\ $V(\bx+\bv) = V(\bx) \text{ for all } \bx \in \mathbb{R}^{2}$ and $\bv \in \Lambda$;\label{lattice-periodic}
\end{enumerate}
There exists  $\bx_{c} \in \mathbb{R}^{2}$ with respect to which (in the sense of Definition \ref{invariance})
\begin{enumerate}
\item[ii.] $V$ is $\mathcal{C}-$invariant; \label{time-rev}
\item[iii.]   $V$ is $\mathcal{P}-$invariant; and \label{even}
\item[iv.]   $V$ is $\pi/2-$rotationally invariant. \label{R-inv}
\end{enumerate}
\end{newdef}
\nit Throughout this paper we shall assume, with no loss of generality, $\bx_c=0$.

In  some of our results we consider admissible potentials which are also reflection invariant. Such potentials have the full symmetry of the square lattice.  
\begin{newdef}[Reflection Invariance]
\label{def:refinv}
An admissible potential, $V(\bx)$, defined on $\R^2$ is reflection invariant if
$
V(x_1, x_2) = V(x_2, x_1).
$
\end{newdef}

We introduce two basic examples of admissible potentials, each obtained as a sum of translates of 
 a fixed atomic potential:
\begin{example}[Square lattice potential]\label{square-pot}
Let $V(\bx) = \sum_{\bfm \in \mathbb Z^2} V_0( \bx + \bfm )$, where $V_0=V_0(|\bx|)$ is real-valued, radially-symmetric, sufficiently rapidly-decaying ``atomic potential'' .  We call $V(\bx)$  a \emph{ square lattice potential}. It is easily seen that this class of potentials is admissible with $\bx_c=0$. 
\end{example}
\begin{example}[Lieb lattice potential]\label{lieb-pot}
We fix the fundamental period cell to be the square with side-length one. 
Within a fixed cell are three points, labeled $\bA$, $\bB$ and $\bC$; see Figure \ref{fig:Lattices} (right).
The Lieb lattice, $\mathbb{L}$, is the union of three sublattices: $\bA+\Z^2$, $\bB+\Z^2$ and $\bC+\Z^2$. A \emph{ Lieb lattice potential} is given by 
$V(\bx) = \sum_{\bw \in \mathbb L} V_0( \bx + \bw )$, where $V_0=V_0(|\bx|)$ is  real-valued, radially-symmetric and  rapidly-decaying  atomic potential. Lieb lattice potentials are admissible with $\bx_c = 0$.  An example of an atomic Lieb lattice potential is displayed in Figure \ref{figure9}. 
\end{example}

\medskip

The next proposition states that if $V$ is an admissible potential in the sense of Definition \ref{def:sq-pot} then  the operator $H_V$ acting in $L^2_\bk$, has an additional symmetry, for $\bk\in\bM+\Lambda^*$.

\begin{prop}\label{extra-symm}
  Assume $V$ is an admissible potential in the sense of Definition \ref{def:sq-pot}.  Then $H_V=-\Delta+V$ and $\mathcal{R}$ map a dense subspace of $L^{2}_{\bM}$ to itself. Moreover, restricted to this dense subspace of $L^{2}_{\bM}$, the commutator $[H_V, \mathcal{R}] \equiv H_V\mathcal{R}-\mathcal{R}H_V= 0$. In particular, if $\Phi(\bx)$ is a solution of the $\bM-$pseudo-periodic eigenvalue problem for $H_V$ for energy $E$, then $\mathcal{R}[\Phi](\bx)$ is also a solution of  the $\bM-$pseudo-periodic eigenvalue problem for $H_V$ with energy $E$.
  \end{prop}
  \begin{proof}  Note first that $-\Delta$ commutes with rotations and the operator $\Phi\mapsto V\Phi$, 
where $V$ is an admissible potential, commutes with $\pi/2$ rotations. Furthermore, assume $\Phi(\bx)$ is $\bM-$pseudo-periodic and define $\Phi_R(\bx)\equiv \Phi(R^*\bx)$. Then, for all $\bv\in\Lambda=\Z^2$ and all $\bx\in\R^2$:
 $\Phi_{R}(\bx + \bv) = \Phi(R^{*}\bx + R^{*}\bv) = e^{i\bM \cdot R^{*}\bv} \Phi(R^{*}\bx)= e^{iR\bM \cdot \bv}\Phi(R^{*}\bx) = e^{i(\bM -\bk_{2})\cdot \bv} \Phi(R^{*}\bx)  =  e^{i\bM \cdot \bv} \Phi_{R}(\bx)
 $. For more detail, see an analogous result in \cite{FW:12}.
  \end{proof}


\subsection{Fourier series of admissible potentials}\label{FS-adm}
 The following proposition implies constraints on the Fourier coefficients of admissible potentials. 
  \begin{prop}\label{coef-cons}%
Let $V(\bx)$ be in $C^\infty(\R^2/\Lambda)$. Then, 
 for all $\bfm=(m_1,m_2)\in\Z^2$
\begin{enumerate}[i.]
\item $V(-\bx)=V(\bx)\ \implies$\ $V_\bfm= V_{-\bfm},$ \label{minus-inv}
\item $\overline{V(\bx)}=V(\bx)\ \implies$\ $V_\bfm=\overline{V_{-\bfm}},$
\item $\mathcal{R}[V](\bx)=V(\bx)\ \implies$\ $V_{m_1,m_2}= V_{-m_2,m_1}.$ \label{rot-symm-pot}
\end{enumerate}
\end{prop}
\begin{proof} Parts (i.) and (ii.) are straightforward. Part (iii.) makes use of the action of $R$ on the dual lattice basis
 $\{\bk_1,\bk_2\}$ or equivalently $R^*$ on the lattice basis $\{\bv_1,\bv_2\}$. 
   \end{proof}

  Note that we may iterate the relation in part \ref{rot-symm-pot} of Proposition \ref{coef-cons} to obtain:
  \begin{equation} 
  V_{m_{1},m_{2}} =V_{-m_2,m_1}=V_{-m_{1},-m_{2}}=V_{m_2,-m_1} .
  \label{Vcycle1}
  \end{equation}
  Introduce the mapping  $\tR: \mathbb{Z}^{2} \rightarrow \mathbb{Z}^{2}$ defined $\tR(m_{1},m_{2}) = (-m_{2},m_{1})$ and therefore $\tR^{2}(m_{1},m_{2}) = (-m_{1},-m_{2})$, $\tR^{3}(m_{1},m_{2}) = (m_{2},-m_{1})$, and  $\tR^{4}(m_{1},m_{2}) = (m_{1},m_{2})$. Thus, $\tR^{4}=\tR^0=Id$, and
   \begin{equation}
    V_\bfm =V_{\tR\bfm}=V_{\tR^{2}\bfm}=V_{\tR^{3}\bfm}\ .\ \label{Vcycle}
    \end{equation}
    Note that ${\bf 0}$ is the unique element of the kernel (and fixed point) of $\tR$ and furthermore that every $\bfm\ne0$ lies on
    a non-trivial 4-cycle of $\tR$, the set $\{ (m_{1}, m_{2}), (-m_{2},m_{1}), (-m_{1},-m_{2}), (m_{2},-m_{1})\}$.
    Let $\bfm$ and $\bn$ be elements of $\Z^2\setminus\{\bf 0\}$. We say that $\bfm\sim\bn$ if   $\bfm$ and $\bn$ lie on the same 4-cycle of $\tR$. The relation $\sim$ is an equivalence relation and partitions $\Z^2\setminus\{\bf 0\}$ into equivalence classes,
     $\left(\Z^2\setminus\{\bf 0\}\right) / \sim$. Let $\widetilde{\mathcal{S}}$ denote a set consisting of  one representative element from each equivalence class. 
         
     \begin{prop}\label{V-expan}
    \begin{enumerate}[(a.)]
    \item  Let  $V$ denote an admissible potential in the sense of Definition \ref{def:sq-pot} and let $V_\bfm=V_{m_1,m_2}$, for $\bfm\in\Z^2$, denote its Fourier coefficients; see \eqref{FS}. Then, 
     \begin{equation}
  V(x_1,x_2) = V_{0,0}+ \sum_{(m_1,m_2)\in\widetilde{\mathcal S}} 2V_{m_1,m_2} \left[\ \cos\left(2\pi(m_1x_1+m_2x_2)\right)\ +\ 
   \cos\left(2\pi(m_2x_1-m_1x_2)\right)\ \right].
   \label{Vexpand}
   \end{equation}
  \label{prop:v-exp}
 \item  If $V$ is also reflection invariant (\ $V(x_1,x_2)=V(x_2,x_1)$\ ), then 
       \begin{equation}
  V(x_1,x_2) = V_{0,0}+\sum_{m\in\mathbb{Z}} 2V_{m, m} \left[\ \cos\left(2\pi m (x_1+x_2)\right)\ +\ 
   \cos\left(2\pi m (x_1-x_2)\right)\ \right].
   \label{Vexpand-refl}
   \end{equation} 
  \label{prop:v-refl-exp}
      \end{enumerate}
           \end{prop}
  \begin{proof} 
  Expanding $V(\bx)$ in a Fourier series, and using the relations in \eqref{Vcycle}, we obtain:
  \[V(\bx)=V_{\bf 0}+\sum_{\bfm\in\widetilde{\mathcal{S}} }\ 
  V_{\bfm}\ \left(\ e^{i\bfm\vec\bk\cdot\bx}+e^{i(\tR\bfm)\vec{\bk}\cdot\bx}+e^{i(\tR^2\bfm)\vec\bk\cdot\bx}+e^{i(\tR^3\bfm)\vec\bk\cdot\bx}\ \right) .\]
  Adding this expression to its complex conjugate and dividing by two and using that the coefficients $V_\bfm$
   are real (Proposition \ref{coef-cons} (i) and (ii) ) implies:
 \begin{align*}
 V(\bx)&=V_{\bf 0}+\sum_{\bfm\in\widetilde{\mathcal{S}} }\ 
  V_{\bfm}\ \left(\ \cos(\bfm\vec\bk\cdot\bx)+\cos\left(\ (\tR\bfm)\vec{\bk}\cdot\bx\ \right)+\cos\left( ( \tR^2\bfm)\vec\bk\cdot\bx\right)+\cos\left((\tR^3\bfm)\vec\bk\cdot\bx\right)\ \right),
  \end{align*}
  which reduces to the expression in \eqref{Vexpand}. Thus Part (\ref{prop:v-exp}.) is proved.

  Suppose $V$ additionally is reflection invariant in the sense of Definition \ref{def:refinv}. By \eqref{Vexpand} this is equivalent to:
  \begin{align*}
&\sum_{(m_1,m_2)\in\widetilde{\mathcal S}} 2V_{m_1,m_2} \left[\ \cos\left(2\pi(m_2x_1+m_1x_2)\right)\ +\ 
   \cos\left(2\pi(m_2x_2-m_1x_1)\right)\ \right]\\
&= \sum_{(m_1,m_2)\in\widetilde{\mathcal S}} 2V_{m_1,m_2} \left[\ \cos\left(2\pi(m_1x_1+m_2x_2)\right)\ +\ 
   \cos\left(2\pi(m_2x_1-m_1x_2)\right)\ \right]\ .
  \end{align*}
  It follows that for all $x_1, x_2:$
    \begin{equation}\label{ref-cos}
  \cos[2\pi(m_2x_1+m_1x_2)]\ +\ 
   \cos[2\pi(m_2x_2-m_1x_1)]\ -\ \cos[2\pi(m_1x_1+m_2x_2)]\ -\ 
   \cos[2\pi(m_2x_1-m_1x_2)]\ =\ 0.
   \end{equation}
   Using trigonometric identities, \eqref{ref-cos} reduces to:
   \[
  f_1(x_1,x_2) \equiv  \sin(2\pi m_2 x_1) \sin(2\pi m_1 x_2) =    \sin(2\pi m_1 x_1) \sin(2\pi m_2 x_2) \equiv f_2(x_1,x_2).
   \]
This implies that the Fourier coefficients of $f_1$ and $f_2$ match and therefore $m_1 = m_2$. 

  \end{proof}



  \subsection{Fourier analysis in $L_{\bM_{*}}^{2}$} \label{Fourier-M}
  In this subsection, we characterize the Fourier series of functions  $\phi \in L_{\bM}^{2}$.
  Such functions may be expressed in the form $\Phi(\bx) = e^{i\bM \cdot \bx} \phi(\bx)$, where $\phi(\bx)$ is $\Lambda=\Z^2-$periodic. Thus, $\Phi$ has the Fourier representation:
\begin{equation}\label{phi-fs0}
\Phi(\bx) = e^{i\bM\cdot \bx} \sum_{(m_{1}, m_{2}) \in \mathbb{Z}^{2}} c(m_{1}, m_{2})e^{i(m_{1}\bk_{1} + m_{2}\bk_{2})\cdot \bx},
\end{equation}
which we rewrite as
\begin{equation}\label{phi-fs} 
\Phi(\bx)= \sum_{(m_{1}, m_{2}) \in \mathbb{Z}^{2}} c_\Phi(m_{1}, m_{2})e^{i(\bM + m_{1}\bk_{1} + m_{2}\bk_{2})\cdot \bx}\ = \sum_{\bfm\in \mathbb{Z}^{2}} c_\Phi(\bfm)e^{i\bM^{\bfm}\cdot \bx}, 
\end{equation}
where $\bM^\bfm = \bM+\bfm\vec\bk=\bM+ m_{1}\bk_{1} + m_{2} \bk_{2}\in \bM+\Lambda^*$. We  denote the Fourier coefficients of  a specific $\Phi \in L^{2}_{\bM}$ shown in \eqref{phi-fs} as $c_{\Phi}(\bfm)$ or $c(\bfm; \Phi)$.
\medskip
  
Next, observe that the transformation $\mathcal{R}\colon f\mapsto\mathcal{R}[f](\bx)=f(R^*\bx)$  is unitary on $L^{2}_\bk$ and so its eigenvalues lie on the unit circle in $\mathbb{C}$. Furthermore, if $\mathcal{R} \Phi = \sigma \Phi$ and $\Phi \neq 0$, then since $\mathcal{R}^{4} = Id,$ we have that $\Phi = \mathcal{R}^{4}\Phi = \sigma^{4}\Phi$, so that $\sigma^{4} = 1.$ Therefore, $\sigma \in \{+1, -1, +i, -i\}$.
\medskip

This induces a decomposition of  $L^2_\bM$ as an orthogonal sum of eigenspaces of $\mathcal{R}$: 

  \begin{equation}
  L^2_{\bM}\ =\ L^2_{\bM,1}\oplus L^2_{\bM,-1}\oplus L^2_{\bM,i}\oplus L^2_{\bM,-i}\ ,
  \label{decompL2}
  \end{equation}
where
   \begin{equation}
   L^2_{\bM,\sigma}\ =\ \{\ f\in L^2_\bM :\ \mathcal{R}[f]=\sigma f\ \},\quad \sigma\in \{1,-1,i,-i\}.
 \label{L2Msig}  \end{equation}

\begin{remark}\label{L2splits}
The spectral theory of $H$ in $L^{2}_{\bM}$, can be reduced to its independent study in each of summand subspace in the orthogonal sum \eqref{decompL2}.
\end{remark}

Our next goal is to characterize Fourier series of functions in the orthogonal summands $L^2_{\bM,\sigma}$ for $\sigma=\pm1,\pm i$. We first apply $\mathcal{R}$ to $\Phi$, represented as a  Fourier series in \eqref{phi-fs}. Note  that 
 \[ R\bM^\bfm=R(\bM+\bfm\vec\bk)= R\bM+R(m_1\bk_1+m_2\bk_2)=\bM +m_2\bk_1+(-1-m_1)\bk_2=\bM^{m_2,-1-m_1}\ ,\] 
 and define (taking some liberty with notation)
 \[ \mathcal{R}\bfm \equiv \mathcal{R}(m_1,m_2)=(m_2,-1-m_1)\]
 and hence  $\mathcal{R}^{-1}\bfm = \mathcal{R}^{-1}(m_1,m_2)=(-m_2-1,m_1)$ .
The mapping $\mathcal{R}$ acting on $L^2_\bM$ induces a decomposition of $\Z^2$ into orbits of minimal length four:
\begin{equation}
(m_1,m_2)^\mathcal{R}\mapsto (m_2,-1-m_1)^\mathcal{R}\mapsto (-1-m_1,-1-m_2)^\mathcal{R}\mapsto 
(-1-m_2,m_1)^\mathcal{R}\mapsto
(m_1,m_2)
\label{Rcycle}
\end{equation} 
and we write:  $\mathcal{R}^2\bfm=(-1-m_1,-1-m_2)$, 
$\mathcal{R}^3\bfm=(-1-m_2,m_1)$ and $\mathcal{R}^4\bfm=(m_1,m_2)$.

For $\bfm, \bn\in\Z^2$ we write $\bfm\approx\bn$ if $\bfm$ and $\bn$ lie on the same orbit under $\mathcal{R}$.
We denote by $\mathcal{S}$ any set containing exactly one representative from each equivalence class in $\Z^2/\approx$. 

\begin{remark}\label{e-class}
One such equivalence class is $\{(0,0),(0,-1),(-1,-1),(-1,0) \}$ and we choose  $(-1,0)$ as its representative in $\mathcal{S}$. In Section \ref{sec:eps-small} we shall define $\mathcal{S}^\perp=\mathcal{S}\setminus\{(0,-1)\}$ and write 
$\mathcal{S}\equiv\{(-1,0)\}\cup \mathcal{S}^\perp$. 
\end{remark}

In terms of the above notation we have 
\begin{align}
R\bM\ &=\ \bM^{m_2,-1-m_1} =\bM^{\mathcal{R}\bfm} \label{RbM}\\
\mathcal{R}[\Phi](\bx) &=   e^{i\bM\cdot\bx}\ \sum_{(m_1,m_2)\in\Z^2} c_{\Phi}(m_{1}, m_{2})
 e^{i(m_2\bk_1+(-1-m_1)\bk_2) \cdot \bx}\ =\ 
 \sum_{\bfm\in\Z^2} c_\Phi(\bfm)e^{i\bM^{\mathcal{R}\bfm}\cdot\bx}.\nn
\end{align}
Therefore, $c_{\mathcal{R}\phi}(\mathcal{R}\bfm)= c_\phi(\bfm)$.
Note that 
 \begin{align}
 \mathcal{R}^{-1}\bfm &= \mathcal{R}^{3}\bfm=(-m_2-1,m_1),\label{Rm1}\\
   \mathcal{R}^{-2}\bfm &=\mathcal{R}^2\bfm = (-1-m_1,-1-m_2) , \text{ and }\label{Rm2}\\
     \mathcal{R}^{-3}\bfm &=\mathcal{R}\bfm=(m_2,-m_1-1). \label{Rm3}
        \end{align}
Hence,
\begin{align}
c_{\mathcal{R}^j\Phi}(\bfm)=c_\Phi(\mathcal{R}^{4-j}\bfm),\quad j=0,1,2,3. 
\label{cRjm}
\end{align}
The Fourier series of $\Phi \in  L^2_\bk$, satisfying the pseudo-periodic boundary conditions may be expressed as a sum over 4-cycles of $\mathcal{R}$:
\begin{equation}\label{phi-r-decomp}
\begin{split}
  \phi(\bx)
  &= \sum_{\bfm \in \mathcal{S}} c_{\Phi}(\bfm)e^{i\bM^{\bfm} \cdot \bx} + c_{\Phi}(\mathcal{R}\bfm)e^{iR\bM^{\bfm} \cdot \bx} + c_{\Phi}(\mathcal{R}^{2}\bfm)e^{iR^{2}\bM^{\bfm} \cdot \bx} + c_{\Phi}(\mathcal{R}^{3}\bfm)e^{iR^{3}\bM^{\bfm} \cdot \bx}\ .\end{split}
\end{equation}
\medskip

\nit We next study the Fourier representation  \eqref{phi-r-decomp} in the case where $\Phi\in L^2_{\bM,\sigma}$ for $\sigma=\pm 1,\pm i$. 

\begin{prop}\label{Fco-sig}
Let $\Phi\in L^2_\bM$. Then, 
\[\Phi\in L^2_{\bM,\sigma}\ \iff\ c_\Phi(\mathcal{R}^j\bfm)\ =\ \sigma^{4-j}\ c_\Phi(\bfm),\quad  j=0,1,2,3.\]
In particular, 
\begin{align}
\mathcal{R}\Phi &= \Phi\ \iff\ c(\bfm)=c(\mathcal{R}\bfm)=c(\mathcal{R}^2\bfm)=c(\mathcal{R}^3\bfm)\nn\\
\mathcal{R}\Phi &= -\Phi\ \iff\ c(\mathcal{R}\bfm)=-c(\bfm),\quad c(\mathcal{R}^2\bfm)=c(\bfm),\quad c(\mathcal{R}^3\bfm)=-c(\bfm)\nn\\
\mathcal{R}\Phi &= i \Phi\ \iff\ c(\mathcal{R}\bfm)=-ic(\bfm),\quad c(\mathcal{R}^2\bfm)=-c(\bfm),\quad c(\mathcal{R}^3\bfm)=+ic(\bfm)\nn\\
\mathcal{R} \Phi &= -i \Phi\ \iff\ c(\mathcal{R}\bfm)=+ic(\bfm),\quad c(\mathcal{R}^2\bfm)=-c(\bfm),\quad
 c(\mathcal{R}^3\bfm)=-ic(\bfm)\nn. 
\end{align}
\end{prop}
\begin{proof} Suppose $\Phi\in L^2_{\bM,\sigma}$ and  $\mathcal{R}\Phi=\sigma\Phi$. Then, $\mathcal{R}^2\Phi=\sigma^2\Phi$
 and $\mathcal{R}^3\Phi=\sigma^3\Phi$. Correspondingly, $c_{\mathcal{R}^j\Phi}(\bfm)=\sigma^j c_\Phi(\bfm)$ for $j=0,1,2,3$.
 By relations \eqref{cRjm} we have  $c_\Phi(\mathcal{R}^{4-j}\bfm)=c_{\mathcal{R}^j\Phi}(\bfm)=\sigma^j c_\Phi(\bfm)$ for $j=0,1,2,3$.
Replacing $j$ by $4-j$ completes the proof.
\end{proof}

Applying Proposition \ref{Fco-sig} to \eqref{phi-r-decomp} we obtain:

\begin{prop}\label{phi_sig} Let $\sigma \in \{+1,-1,+i,-i\}$. 
 Then,   $\Phi\in L_{\bM,\sigma}^{2}$
 if and only if  there exists $\{c(\bfm)\}_{\bfm\in \mathcal{S}}$ in $l^2(\mathcal{S})$ such that
\begin{equation}
\Phi(\bx) = \sum_{\bfm \in \mathcal{S}}  c(\bfm) \sum_{j=0}^{3} \sigma^{4-j}e^{i R^{j}\bM^{\bfm} \cdot \bx}.
\label{phi-sig-gen}
\end{equation}
In detail, 
\begin{enumerate}
\item $\Phi \in L_{\bM,i}^{2} \iff$ there exists $\{c(\bfm)\}_{\bfm\in \mathcal{S}}\in l^2(\mathcal{S})$ such that
\begin{equation}
\begin{split}
  \Phi(\bx) &= \sum_{\bfm \in \mathcal{S}} c(\bfm) \left( e^{i \bM^{\bfm} \cdot \bx} -i e^{i R\bM^{\bfm} \cdot \bx} - e^{i R^{2} \bM^{\bfm} \cdot \bx} + i e^{i R^{3} \bM^{ \bfm} \cdot \bx}\right).
\end{split}
\end{equation}
\item $\Phi \in L_{\bM,-i}^{2} \iff$ there exists $\{c(\bfm)\}_{\bfm\in \mathcal{S}} \in l^{2}(\mathcal{S})$ such that
\begin{equation}
  \Phi(\bx) = \sum_{\bfm \in \mathcal{S}} c(\bfm) \left( e^{i \bM^{\bfm} \cdot \bx} +i e^{i R\boldsymbol{\bM}^{\bfm} \cdot \bx} - e^{i R^{2} \bM^{\bfm} \cdot \bx} -i e^{i R^{3} \bM^{ \bfm} \cdot \bx}\right).
\end{equation}
\item $\Phi \in L_{\bM,1}^{2} \iff$ there exists $\{c(\bfm)\}_{\bfm\in \mathcal{S}} \in l^{2}(\mathcal{S})$ such that
\begin{equation}
  \Phi(\bx) = \sum_{\bfm \in \mathcal{S}} c(\bfm) \left( e^{i \bM^{\bfm} \cdot \bx} + e^{i R\bM^{\bfm} \cdot \bx} + e^{i R^{2} \bM^{\bfm} \cdot \bx} + e^{i R^{3} \bM^{ \bfm} \cdot \bx}\right).
\end{equation}

\item $\Phi \in L_{\bM,-1}^{2} \iff$ there exists $\{c(\bfm)\}_{\bfm\in S} \in l^{2}(\mathcal{S})$ such that
\begin{equation}
  \Phi(\bx) = \sum_{\bfm \in \mathcal{S}} c(\bfm) \left( e^{i \bM^{\bfm} \cdot \bx} - e^{i R\bM^{\bfm} \cdot \bx} + e^{i R^{2} \bM^{\bfm} \cdot \bx} - e^{i R^{3} \bM^{ \bfm} \cdot \bx}\right).
\end{equation}
\end{enumerate}
Finally, $\mathcal{P}\circ\mathcal{C}$ is a bijection between $L^2_{\bM,i}$ and $L^2_{\bM,-i}$. If $c_\Phi(\bfm),\bfm\in\mathcal{S}$ are the Fourier coefficients of $\Phi\in L^2_{\bM,i}$ , then
  $c_{(\mathcal{P}\circ\mathcal{C})\Phi}(\bfm)=\overline{c_\Phi(\bfm)},\ \bfm\in\mathcal{S}$ are the
  Fourier coefficients of $(\mathcal{P}\circ\mathcal{C})\Phi\in L^2_{\bM,-i}$ .
\end{prop}

 \section{$H^{(0)}=-\Delta$ on $L^{2}_{\bM}$: A Four-fold Degenerate Eigenvalue}\label{sec:freeH}

We consider the eigenvalue problem \eqref{k-evp} for the case $V\equiv0$.
Let $H^{(0)}=-\Delta$. 
\begin{equation}\label{H0-eval}
H^{(0)} \Phi^{(0)} = \mu^{(0)}(\bk) \Phi^{(0)}, \; \Phi^{(0)} \in L_{\bk}^{2}.
\end{equation}

Equivalently, take $\Phi^{(0)}(\bx; \bk) = e^{i\bk \cdot \bx}\phi^{(0)}(\bx)$, where $\phi^{(0)}(\bx) \in L^{2}(\mathbb{R}^{2}/\Lambda)$. We have (see \eqref{per-evp})
\begin{equation}\label{H0-eval-p-form}
\begin{split}
&H^{(0)}(\bk) \phi^{(0)} = -(\nabla + i\bk)^{2}\phi^{(0)} = \mu^{(0)}(\bk) \phi^{(0)},\\
& \phi^{(0)}(\bx + \bv) = \phi^{(0)}(\bx), \;\; \bv \in \Lambda.
\end{split}
\end{equation}
For $m_1,m_2 \in \Z$, the eigenvalue problem \eqref{H0-eval-p-form} has solutions of the form:
\[
\phi^{(0)}_{m_{1}, m_{2}}(\bx; \bk) = e^{i(m_{1}\bk_{1} + m_{2}\bk_{2})\cdot \bx},\ \ \bk  \in\brill ,
\]
with corresponding eigenvalues
\[
\mu_{m_{1}, m_{2}}^{(0)}(\bk) = \vert \bk + m_{1}\bk_{1} + m_{2}\bk_{2}\vert^{2}, \ \bk \in \brill .
\]
The dispersion relation for the free Hamiltonian, $H^{(0)}$, is plotted in Figure \ref{fig:free-hamiltonian}.

\nit The following result concerns the spectral problem for the high symmetry quasi-momentum $\bk=\bM$
 (and by Remark \ref{allvertices} all vertices of $\brill$):
\begin{thm}\label{H0-mult4}
Let $\bk=\bM$ and $\sigma=\pm1, \pm i$. Then, 
\begin{enumerate}
\item $\mu^{(0)}_S\equiv |\bM|^2=2\pi^2$ is an $L^2_{\bM}-$eigenvalue of multiplicity four with corresponding four-dimensional eigenspace given by
\begin{equation}
L^2_\bM-\ {\rm Kernel}\left(\ H^{(0)}-\mu^{(0)}_S{\rm Id}\ \right)\ =\  {\rm span}\left\{\ e^{i\bM\cdot\bx},\ \ e^{iR\bM\cdot\bx}, \ \ e^{iR^2\bM\cdot\bx},\ \ e^{iR^3\bM\cdot\bx}\ \right\}.\label{list0}
 \end{equation}
 \label{thm-e-val}
\item $H^{(0)}$ acting in $L^2_{\bM,\sigma}$  has simple eigenvalue $\mu_S^{(0)} = \vert \bM \vert^{2}$ with
corresponding eigenspace: 
\[L^2_{\bM,\sigma}-\ {\rm Kernel}\left(\ H^{(0)}-\mu^{(0)}_S{\rm Id}\ \right)\ 
= {\rm span}\left\{\Phi^{(0)}_\sigma\right\},\]
where $\Phi^{(0)}_\sigma$ are defined as follows.
\begin{align}
\Phi_{+1}^{(0)}(\bx)\ &=\ e^{i\bM\cdot\bx}\ +\ e^{iR\bM\cdot\bx}\ +\ e^{iR^2\bM\cdot\bx}\ +\
 e^{iR^3\bM\cdot\bx}\ \in\  L^2_{\bM,+1}\label{Phi10}\\
 \Phi_{-1}^{(0)}(\bx)\ &=\ e^{i\bM\cdot\bx}\ -\ e^{iR\bM\cdot\bx}\ +\ e^{iR^2\bM\cdot\bx}\ -\
 e^{iR^3\bM\cdot\bx}\ \in\  L^2_{\bM,-1}\label{Phi-10}\\
 \Phi_{+i}^{(0)}(\bx)\ &=\ e^{i\bM\cdot\bx}\ -i\ e^{iR\bM\cdot\bx}\ -\ e^{iR^2\bM\cdot\bx}\ +i\
 e^{iR^3\bM\cdot\bx}\ \in\  L^2_{\bM,+i}\label{Phi+i0}\\
 \Phi_{-i}^{(0)}(\bx)\ &=\ e^{i\bM\cdot\bx}\ +\ i\ e^{iR\bM\cdot\bx}\ -\ e^{iR^2\bM\cdot\bx}\ -i\
 e^{iR^3\bM\cdot\bx}\ \in\  L^2_{\bM,-i}\label{Phi-i0} 
 \end{align}
 Furthermore,
 \begin{align}
&L^2_\bM-\ {\rm Kernel}\left(\ H^{(0)}-\mu^{(0)}_S{\rm Id}\ \right)\nn\\
& \qquad=\  
{\rm span}\left\{\ \Phi_{+1}^{(0)}(\bx)\ \right\}\oplus{\rm span}\left\{\ \Phi_{-1}^{(0)}(\bx)\ \right\}
\oplus{\rm span}\left\{\ \Phi_{+i}^{(0)}(\bx)\ \right\}\oplus{\rm span}\left\{\ \Phi_{-i}^{(0)}(\bx)\ \right\}.\nn
 \end{align}
 \label{list1}
  \item $\mu_S^{(0)}=2\pi^2$ is the lowest eigenvalue of $H^{(0)}$ in $L^{2}_{\bM}$.\label{lowest-eval}
 \end{enumerate}
\end{thm}

\begin{proof}
The function  $e^{i\bk\cdot \bx}$ is an $L^2_\bk-$eigenvalue of $-\Delta$ with eigenvalue $|\bk|^2$. 
Since vertices of the Brillouin zone, $\bM, R\bM, R^2\bM$ and $R^3\bM$, are equidistant from the origin and are all equivalent modulo $\Lambda^*$, we have   $\mu^{(0)}_S = |\bM|^2=2\pi^{2}$
 is an $L^2_\bM-$eigenvalue of multiplicity at least four with eigenspace contained in the span of the functions 
 $e^{i\bM\cdot\bx}$, $e^{iR\bM\cdot\bx}$, $e^{iR^2\bM\cdot\bx}$ and $e^{iR^3\bM\cdot\bx}$.
 To show that $\mu^{(0)}_S$  is of multiplicity exactly four, we seek to find $\bfm$ for which $\vert \bM^{\bfm} \vert^{2} = \vert \bM \vert^{2}.$  Using $\bM^{\bfm} = \bM + m_{1}\bk_{1} + m_{2}\bk_{2}$, we have
$
|\bM^\bfm|^2-|\bM|^2=(2\pi)^2\left[m_{1}^{2} + m_{2}^{2} + m_{1} + m_{2}\right] ,
$
which vanishes only if $\bfm = (0,0), (0, -1), (-1, -1)$ or $ (-1, 0).$  These four possibilities correspond to the four vertices of $\brill$.
 Thus, $\mu_S^{(0)}$ is of multiplicity exactly four. This proves part \ref{thm-e-val}. 
Part \ref{list1} is a consequence of Proposition \ref{phi_sig} and its proof.
Part \ref{lowest-eval} holds because $m_{1}^{2} + m_{2}^{2} + m_{1} + m_{2} \geq 1>0$ for $\bfm = (m_{1}, m_{2}) \notin \{(0,0), (0, -1), (-1, -1), (-1, 0)\}$ and therefore 
$|\bM^\bfm|^2\ge |\bM|^2+(2\pi)^2> |\bM|^2$.
 \end{proof}

\begin{figure}[t!]
\begin{center}
\includegraphics[width=.48\textwidth]{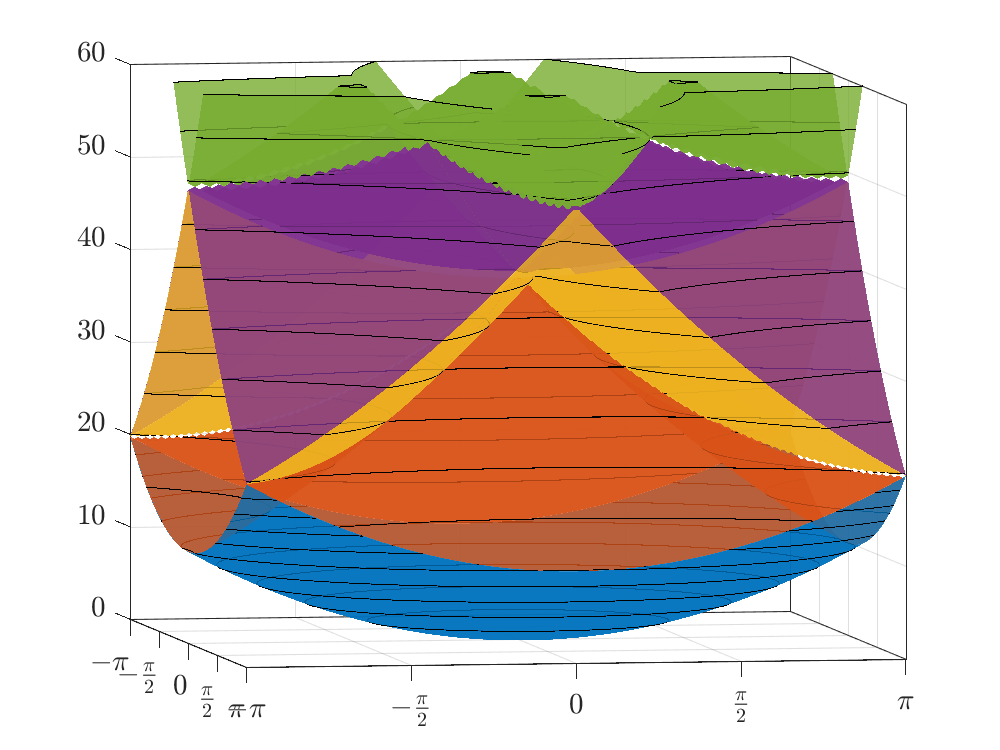}
\includegraphics[width=.48\textwidth]{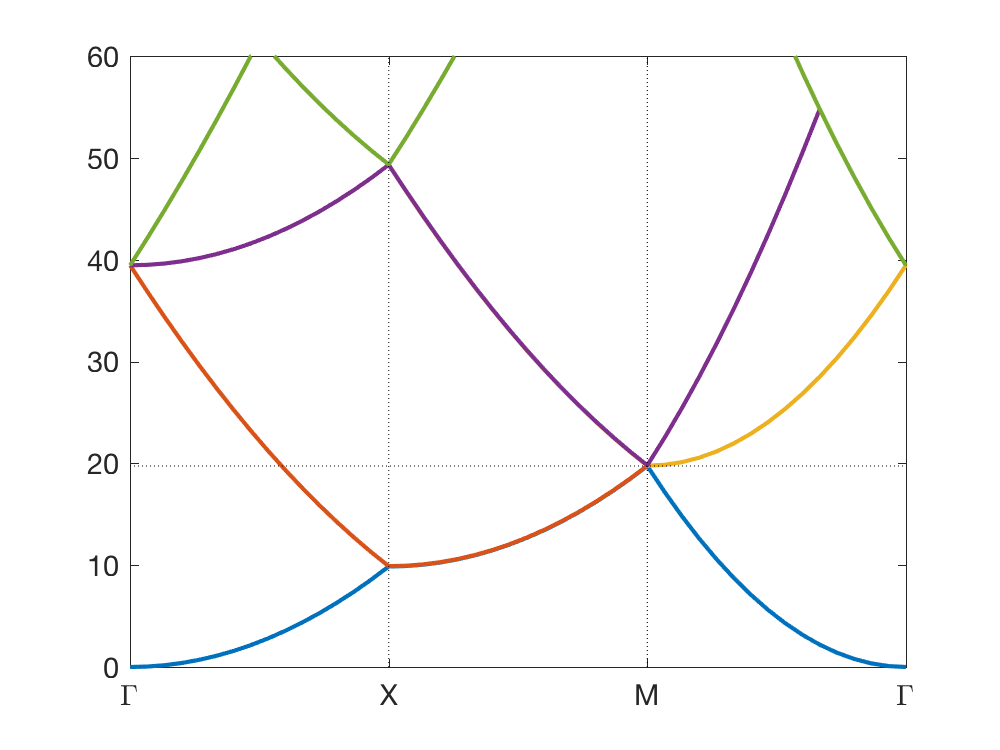}
\end{center}
\caption{ Dispersion surfaces of $H^{(0)}=-\Delta$. {\bf (left)} The first five dispersion surfaces are plotted over the Brillouin zone, $\left[-\pi, \pi \right]^2$. Each surface is plotted using a different color. (The first is blue, the second is red, {\it etc}.) Some level sets of the dispersion surfaces are indicated with black lines. 
{\bf (right)} The same dispersion surfaces are  plotted along the circuit ${\bf \Gamma} \to {\bf X} \to {\bf M} \to {\bf \Gamma}$, displayed in Figure \ref{fig:squareBZ}. 
The colors match the plot on the left. 
As shown in Theorem \ref{H0-mult4}, there is a multiplicity four $L^2_\bM-$eigenvalue $\mu^{(0)}=|\bM|^2=2\pi^2$.  }
\label{fig:free-hamiltonian}
\end{figure}


\section{Two-fold Degenerate $L^{2}_{\bM}$ Eigenvalues Imply Quadratic Touching of Dispersion Surfaces}
\label{2fold-touch}
%
\begin{thm}\label{quad-disp}
Let $H = - \Delta_{\bx} + V(\bx)$, where $V(\bx)$ is an admissible  potential in the sense of Definition \ref{def:sq-pot}.
 Assume that $\mu_\star$ is a two-fold degenerate $L^2_\bM$ eigenvalue of $H$. More specifically,  
\begin{enumerate}
\item[(H1)]\ $H$ has a simple $L^{2}_{\bM, +i}$ eigenvalue $\mu_S $ with corresponding  normalized eigenfunction
 $\Phi_{1}(\bx)=e^{i\bM\cdot\bx}\phi_1(\bx)$.
\item[(H2)]\ $H$ has a simple $L^{2}_{\bM, -i}$ eigenvalue $\mu_S $ with corresponding  normalized eigenfunction
\[\Phi_{2} =\left(\mathcal{P}\circ\mathcal{C}\right)[\Phi_1](\bx)=\overline{\Phi_1(-\bx)}\equiv e^{i\bM\cdot\bx}\phi_2(\bx).
\]
We shall also use the notation  $\Phi_{1}=\Phi_{(+i)}$ and $\Phi_{2}=\Phi_{(-i)}$. 
\item[(H3)]\ $\mu_S $ is neither a  $L^{2}_{\bM, +1}$ nor a $L^{2}_{\bM, -1} $eigenvalue of $H$.
\end{enumerate}

Then, there exist dispersion relations: $\bk\mapsto\mu_\pm(\bk)$ associated with the $L^2_\bk-$eigenvalue problem 
 for $H$, whose local character in a neighborhood of the high symmetry quasi-momentum, $\bM$ (and therefore all vertices of 
  $\brill$), is given by:
  \begin{equation}\label{mu-pert}
\mu_{\pm}(\bM+\bkappa) - \mu_S = 
  (1-\alpha)|\kappa|^2+\mathscr{Q}_6(\kappa)\ \pm \sqrt{\Big|\ \gamma(\kappa_{1}^{2} - \kappa_{2}^{2})+ 2\beta \kappa_{1}\kappa_{2}\ \Big|^2\ +\ \mathscr{Q}_8(\kappa)}\ ,
\end{equation}
for $|\bk-\bM|=\sqrt{\kappa_1^2+\kappa_2^2}$ small. The constants $\alpha\in \mathbb{R}$ and $\beta, \gamma \in \mathbb{C}$ are  inner product expressions which are quadratic in the 
 the entries of $\nabla_{\bx}\Phi_{1}$ and $\nabla_{\bx}\Phi_{2}$; see equations \eqref{albega}.
  The functions $\mathscr{Q}_6(\kappa)=\mathcal{O}(|\kappa|^6)$ and $\mathscr{Q}_8(\kappa)=\mathcal{O}(|\kappa|^8)$ are analytic functions of  $\kappa$ and invariant under $\pi/2$ rotation: $(\kappa_1,\kappa_2)\mapsto(-\kappa_2,\kappa_1)$. 
   
\end{thm}

\nit The proof of Theorem \ref{quad-disp} is given in Section \ref{pf-quad-disp}.

\begin{cor}\label{rho-inv}
Assume hypotheses of Theorem \ref{quad-disp}.
 Assume further that with respect to the origin of coordinates, $\bx_c=0$,
we have, in addition, that $V$ is reflection invariant in the sense of Definition \ref{def:refinv}, {\it i.e.} $V(x_1, x_2) = V(x_2, x_1)$. Then, the coefficients $\beta$ and $\gamma$ in \eqref{mu-pert} are  constrained to satisfy: 
 $\beta\in\R$ and $\gamma=i\tilde{\gamma}\in i\R$ and we have:
\begin{align}
\mu_{\pm}(\bM+\bkappa) - \mu_S &= 
(1-\alpha)|\kappa|^2+\mathscr{Q}_6(\kappa)\ \nn\\
&\quad  \pm \sqrt{  \tilde{\gamma}^2(\kappa_1^2-\kappa_2^2)^2 + 4\beta^2 \kappa_{1}^2\kappa_{2}^2\ +\ \mathscr{Q}_8(\kappa) }\ .
\label{mu-rho}\end{align}
Here,  $\mathscr{Q}_6(\kappa)$ and $\mathscr{Q}_8(\kappa)$ are now also invariant under the reflection: $(\kappa_1,\kappa_2)\mapsto(\kappa_2,\kappa_1)$.
\end{cor}
\bigskip

Before presenting the proofs of Theorem \ref{quad-disp} and Corollary \ref{rho-inv}, we state a result on the 
instability or non-persistence of the quadratic degeneracies of Theorem \ref{quad-disp} against a class of real-valued perturbations
which preserve $\Z^2-$periodicity and inversion symmetry, but break $\pi/2-$rotational invariance.

\begin{thm}[Non-persistence of quadratic degeneracy]
\label{deformation}
 Consider $H^\eta = -\Delta + V + \eta W$, where $V$ is admissible. By Theorem \ref{quad-disp},

\nit $\bullet$\quad $H^0$ has an $L^2_{\bM}-$eigenvalue $\mu_S$ of geometric multiplicity two, and \\
\nit $\bullet$\quad  $\mu_S$ has an associated orthonormal basis $\{\Phi_1, \Phi_2\}$ with $\Phi_1 \in L^2_{\bM, i}$ and $\Phi_2(\bx) = \overline{\Phi_1(-\bx)}$.

We introduce a class of perturbations, $W$, consisting of real-valued functions which are $\Z^2-$periodic and  even,  but which do not respect $\pi/2-$rotational invariance, {\it i.e.} $\mathcal{R}[W] \neq W$. In particular, we assume that 
\begin{equation}
\langle \Phi_1, W\Phi_2 \rangle \neq 0\ .
 \label{hyp-non-degen}\end{equation}
\nit  Then, the two-fold degenerate eigenvalue splits into two simple eigenvalues, $\nu_{\pm}$, given by:
\begin{equation}\label{w-pert-eigs}
\nu_{\pm}  = \mu_S + \eta \langle \Phi_1, W \Phi_1 \rangle \pm \eta\abs{ \langle \Phi_1, W \Phi_2 \rangle} + \mathcal{O}(\eta^2).
\end{equation}
\end{thm}

\nit We omit the proof of Theorem \ref{deformation}, which follows from degenerate perturbation theory argument; see Section 9 (particularly, Remark 9.2) of \cite{FW:12} and  Section 5 of \cite{LWZ:18}.

\begin{remark}
It is easy to verify that if $W$ is  real-valued, $\Z^2-$periodic, even \underline{and} $\pi/2-$rotationally invariant, then 
$ \langle \Phi_1, W\Phi_2 \rangle = 0$. \end{remark}

\begin{remark}
 We provide an example of such potential $W$ that is even, is \underline{not} $\pi/2-$rotationally invariant and has the property:
$
\langle \Phi_1^\eps , W \Phi_2^\eps \rangle \neq 0.$  We set  
\[W_0(\bx) = 2\cos((\bk_1 + \bk_2)\cdot \bx).\] 
Then, 
$
\mathcal{R}[W_0] =  2\cos((\bk_1 + \bk_2)\cdot R^*\bx) = 2\cos(R(\bk_1 + \bk_2)\cdot \bx)  = 2\cos((\bk_1 - \bk_2)\cdot \bx)  \neq W_0(\bx).
$
We obtain $\langle \Phi_1^\eps , W_0 \Phi_2^\eps \rangle = -2 +\mathcal{O}(\eps)\neq 0$ for $\eps$ small.

Other examples which satisfy the hypotheses of Theorem \ref{deformation}:
$W_{1}(\bx) =2 \cos(\bk_1 \cdot \bx)$,  $W_{2}(\bx) = 2 \cos(\bk_2 \cdot \bx)$  and 
$W_{3}(\bx) =  2 \cos((\bk_1-\bk_2) \cdot \bx)$.
We omit the lengthy but elementary verification.  A numerical illustration of Theorem \ref{deformation} is presented in Figure \ref{deform-w}.
\end{remark}

\begin{figure}[h!]
\begin{center}
\caption{The admissible potential $V$ in Figure \ref{figure9} with the additional perturbation $W = 2\cos((\bk_1+\bk_2)\cdot\bx)$.}\label{deform-w}
\subfloat[Plot of Potential $V + \eta W$]{\includegraphics[width=.5\columnwidth, keepaspectratio]{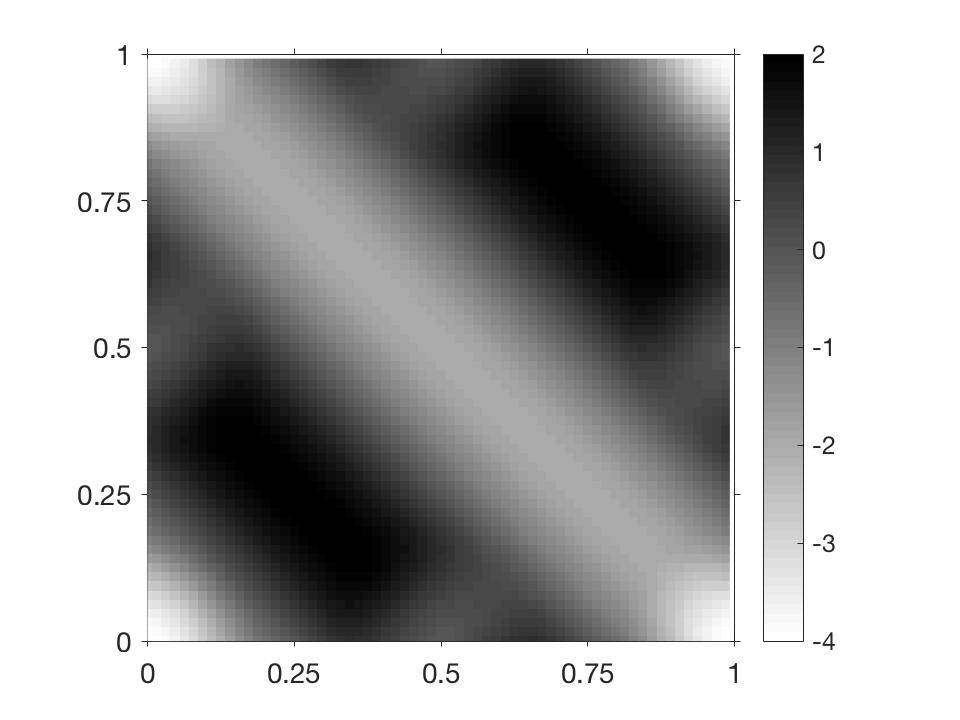}}
\subfloat[Dispersion Curves; Inset Shows Splitting at $\bM$]{\includegraphics[width=.5\columnwidth, keepaspectratio]{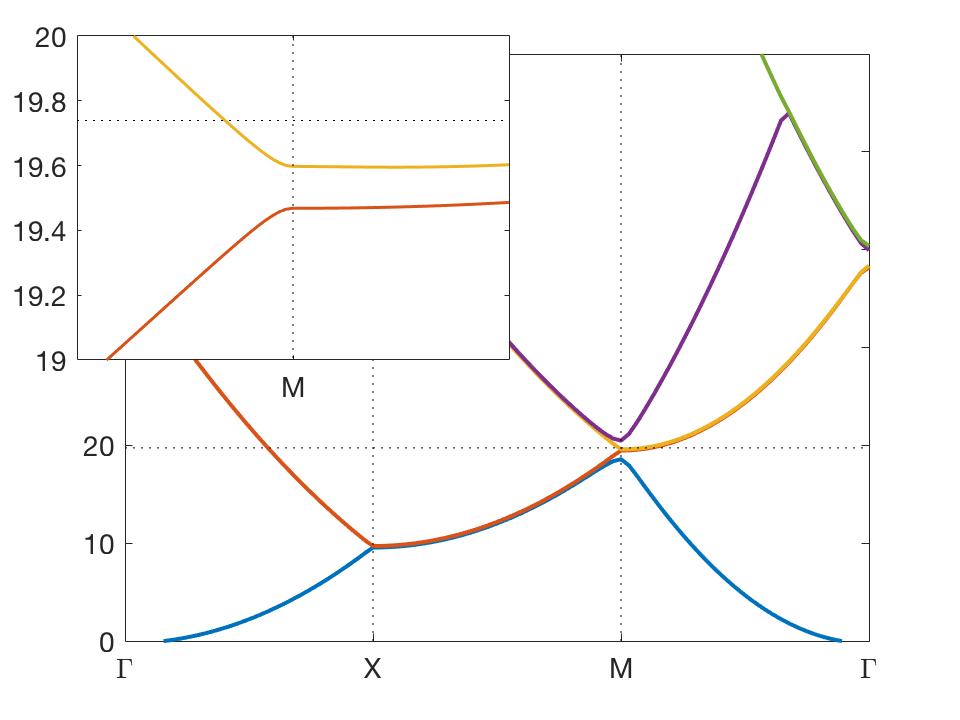}}
\end{center}
\end{figure}

\subsection{Proof of Theorem \ref{quad-disp} on conditions for quadratic degeneracy}\label{pf-quad-disp}

\subsubsection{Reduction to the study of $\det\mathcal{M}(\mu,\bkappa)=0$, for a $2\times2-$matrix-valued analytic function $(\mu,\bkappa)\mapsto\mathcal{M}(\mu,\bkappa)$ in a neighborhood of $(0,{\bf 0})$}\label{reduction}

 The proof follows closely that of Theorem 4.1 of \cite{FW:12}.  For   $\Phi \in L_{\bk}^{2}$ as $\Phi(\bx; \bk) = e^{i \bk \cdot \bx} \phi(\bx; \bk),$ where $\phi(\bx; \bk)$ is $\Lambda-$periodic. Let  $H(\bk) =(-(\nabla_{\bx} + i\bk)^{2} + V(\bx))$.    We  study the eigenvalue problem $H(\bk)\phi(\bx;\bk)=\mu \phi(\bx;\bk)$ for $\bk=\bM + \kappa$ and  $\mid \kappa \mid \ll 1$. In particular,
\begin{equation}\label{pert-evp}
\begin{split}
&\left[ - (\nabla_{\bx} + i (\bM + \kappa) )^{2} + V(\bx) \right] \phi = \mu\ \phi,\\
&\phi(\bx + \bv) = \phi(\bx), \text{ for all } \bv \in \Lambda, \ \ \bx\in\R^2.
\end{split}
\end{equation}
We seek a solution of \eqref{pert-evp}, $\mu=\mu(\bM + \kappa)$ and $\phi=\phi(\bx; \bM + \kappa)$, in the form
\begin{equation}\label{evals-extended} 
  \mu(\bM + \kappa) = \mu_S +  \mu^{(1)}, \hspace{5mm} \phi(\bx; \bM + \kappa) = \phi^{(0)} + \phi^{(1)},
\end{equation}
where 
\[\phi^{(0)}\in {\rm kernel}(H(\bM)-\mu_S I),\ \   \phi^{(1)}\perp {\rm kernel}(H(\bM)-\mu_S I),\]
 and  $\mu^{(1)}$  are to be determined.
Substituting \eqref{evals-extended} into\eqref{pert-evp}, we obtain:
\begin{equation}\label{eqn:pert-sys}
\begin{split}
(H(\bM) - \mu_S I) \phi^{(1)} =& \left( 2i \kappa \cdot (\nabla + i \bM) - \kappa \cdot \kappa + \mu^{(1)} \right)
\ \left(\phi^{(0)}+\phi^{(1)}\right)\ \equiv F^{(1)}. 
\end{split}
\end{equation}
%
The right-hand side of \eqref{eqn:pert-sys} depends on $\phi^{(0)} \in  {\rm kernel}(H(\bM) - \mu_{0}I)$, which, by hypothesis, is expressed for constants $\alpha_1, \alpha_2$ to be determined,
\begin{equation}\label{psi0-exp}
\begin{split}
\phi^{(0)} &= \alpha_1 \phi_{1}(\bx) +  \alpha_2 \phi_{2}(\bx),\\ 
\phi_{j}(\bx) &= e^{-i\bM \cdot \bx}\Phi_{j}(\bx), j = 1, 2,
\end{split}
\end{equation}

We next construct $\phi^{(1)}$. Introduce the orthogonal projections $Q_{\parallel}$  onto the two-dimensional kernel of $H(\bM) - \mu_S I$ and $Q_{\perp} = I - Q_{\parallel}$. Note that:
 $Q_{\parallel} \psi^{(1)}=0$, $Q_{\perp} \psi^{(0)} = 0$, and $Q_{\perp} \psi^{(1)} = \psi^{(1)}$.
Equation \eqref{eqn:pert-sys} is of the form $(H(\bM)-\mu_S I)\phi^{(1)}=F^{(1)}(\alpha_1, \alpha_2, \kappa, \mu^{(1)}, \phi^{(1)} )$ 
 and can be expressed as an equivalent system for  $(\phi^{(1)},\mu^{(1)})$:
 \begin{align}
&(H(\bM) - \mu_S I) \phi^{(1)} = Q_{\perp} F^{(1)}(\alpha_1, \alpha_2, \kappa, \mu^{(1)}, \phi^{(1)} ), \label{Qperp}\\
& 0 = Q_{\parallel} F^{(1)}(\alpha_1, \alpha_2, \kappa, \mu^{(1)}, \phi^{(1)} ).\label{Qpar}
\end{align}
We proceed as in \cite{FW:12}. Introduce the resolvent $\mathscr{R}_{\bM}(\mu_S) = (H(\bM) - \mu_S I)^{-1}$, which  is a bounded linear map from $Q_{\perp}L^{2}(\mathbb{R}^{2}/\Lambda)$ to $Q_{\perp}H^{2}(\mathbb{R}^{2}/\Lambda)$.  Equivalently, 
$\mathscr{R}(\mu_S) = (H - \mu_S I)^{-1}$ is a bounded linear map from $\widetilde{Q}_{\perp}L^{2}_\bM$ to $\widetilde{Q}_{\perp}H^{2}_\bM$, where $\widetilde{Q}$ and $\widetilde{Q}_\perp$ are the orthogonal projections onto $\textrm{span}\{\Phi_1,\Phi_2\}$ and its orthogonal complement. For $\abs{\kappa}$ and $\abs{\mu^{(1)}}$ sufficiently small, we have:
\begin{equation}\label{psi1}
\phi^{(1)} = \alpha_1\ c^{(1)}[\kappa, \mu^{(1)}](\bx) + \alpha_2\ c^{(2)}[\kappa, \mu^{(1)}](\bx),
\end{equation}
where 
\begin{equation}\label{cj}
\begin{split}
c^{(j)}[\bkappa,\mu^{(1)}](\bx) =& \left( I + \mathscr{R}_{\bM}(\mu_S)  Q_{\perp} \left( -2i \kappa  \cdot (\nabla + i \bM) + \kappa \cdot \kappa - \mu^{(1)} \right) \right)^{-1} \\
&\hspace{15mm} \circ  \left( \mathscr{R}_{\bM}(\mu_S) Q_{\perp} \left( 2i \kappa \cdot (\nabla + i \bM)  \right) \phi_j \right).
\end{split}
\end{equation}
where $(\kappa, \mu^{(1)}) \rightarrow c^{(j)}[\kappa, \mu^{(1)}]$ is a smooth mapping from a neighborhood $(0,0) \in \mathbb{R}^{2} \times \mathbb{C}$ into $H^{2}(\mathbb{R}^{2}/\Lambda)$ satisfying the bound
$
  \| c^{(j)} \|_{H^{2}} \leq C(\ |\kappa|^2 +\abs{\mu^{(1)}}\abs{\kappa});  j = 1,2.
$

Substituting \eqref{psi1} into  \eqref{Qpar}, 
we obtain a homogeneous system 
\[
  \mathcal{M}(\mu^{(1)}, \kappa)  \begin{pmatrix} \alpha_1  \\ \alpha_2 \end{pmatrix} = 0\ .
\]
 We therefore have the following characterization of eigenvalues, $\mu=\mu_S+\mu^{(1)}$ 
 for $|\mu^{(1)}|$ small and $\bk=\bM+\kappa$, with $\kappa$ near zero. 
 
 \begin{prop}
 Let $\bk=\bM+\kappa$ with $|\kappa|<\kappa_{\rm max}$ sufficiently small.
 Then, for  with ,  $\mu=\mu_S+\mu^{(1)}$, with $|\mu^{(1)}|$ in a small neighborhood of $0$, is an $L^2_\bM-$eigenvalue if and only if  $\det \mathcal{M}(\mu^{(1)}, \kappa)=0 $.
 \end{prop}
 The matrix $\mathcal{M}(\mu^{(1)}, \kappa)$ is given by
\begin{equation}\label{Mdef}
\mathcal{M}(\mu^{(1)}, \kappa)  \equiv \mathcal{M}^{(0)}(\mu^{(1)}, \kappa)\ +\ \mathcal{M}^{(1)}(\mu^{(1)}, \kappa),
 \end{equation}
where
 \begin{equation}\label{M0def}
\mathcal{M}^{(0)}(\mu^{(1)}, \kappa)  \equiv 
 \begin{pmatrix} 
  \mu^{(1)} -\kappa \cdot \kappa + \left\langle \Phi_{1}, 2i \kappa \cdot \nabla \Phi_{1} \right\rangle & \left\langle \Phi_{1}, 2i\kappa \cdot  \nabla \Phi_{2} \right\rangle \\ 
 \left\langle \Phi_{2}, 2i\kappa  \cdot \nabla \Phi_{1} \right\rangle &  \mu^{(1)} -\kappa \cdot \kappa + \left\langle \Phi_{2}, 2i \kappa \cdot \nabla \Phi_{2}\right\rangle
 \end{pmatrix} ,
\end{equation}
and
 \begin{equation}\label{M1-def}
\mathcal{M}^{(1)}(\mu^{(1)}, \kappa)  \equiv 
 \begin{pmatrix} 
   \left\langle \Phi_{1}, 2i \kappa \cdot \nabla C^{(1)}(\kappa, \mu^{(1)}) \right\rangle &\left\langle \Phi_{1}, 2i \kappa \cdot \nabla C^{(2)}(\kappa, \mu^{(1)})\right\rangle \\ 
    \left\langle \Phi_{2}, 2i \kappa \cdot \nabla C^{(1)}(\kappa, \mu^{(1)})\right\rangle  & \left\langle \Phi_{2}, 2i \kappa \cdot \nabla C^{(2)}(\kappa, \mu^{(1)}) \right\rangle
    \end{pmatrix} .
\end{equation}
Using the relations
$ (\nabla+i\bM) \phi_{j} = e^{-i \bM \cdot \bx} \nabla_{\bx} e^{i\bM\cdot\bx}\phi_{j}=e^{-i \bM \cdot \bx} \nabla_{\bx} \Phi_{j}$, we have $C^{(j)}[\kappa, \mu^{(1)}](\bx) \equiv e^{i\bM \cdot \bx}c^{(j)}[\kappa, \mu^{(1)}](\bx)$, where $\langle \Phi_{i}, C^{(j)} \rangle = 0$ for $i, j = 1, 2$. 

\begin{remark}
Given real $\mu^{(1)}$ and $\kappa$, the matrices $\mathcal{M}, \ \mathcal{M}^{(0)},$ and $\mathcal{M}^{(1)}$ are Hermitian matrices.
\end{remark}


We conclude this section with an elementary lemma which we use, along with symmetry,  to simplify the matrix entries of 
 $\mathcal{M}(\mu^{(1)}, \kappa) $. We denote $\mathcal{M}_{j_1,j_2}^{(1)}(\mu^{(1)},\bkappa) =   \left\langle \Phi_{j_1}, 2i \kappa \cdot \nabla C^{(j_2)}(\kappa, \mu^{(1)}) \right\rangle .$

  \begin{lemma}\label{rot-lemma} Suppose $f\colon \mathbb{R}^{2} \rightarrow \mathbb{R}^{2}, f \in L^{2}(\Omega)$ satisfies 
  $f(R^*\bx)=f(\bx)$, where $R^{*}$ is the counter-clockwise rotation matrix by $\pi/2$.
 Then,
\begin{align}\label{rot-lemma-result}
\nabla_\bx \mathcal{R}[f](\bx) &= R\ \mathcal{R}[\nabla_\by f](\bx)\ \  \textrm{or}\ \ \D_{x_\alpha}\ \mathcal{R}[f](\bx)=R_{\alpha r}\ \mathcal{R}[\D_{y_r}f](\bx)
\end{align}
 \end{lemma}
\begin{proof}
Let $\by = R^{*}\bx$ or $\by_{r} = R_{ir}\bx_{i}$. Therefore,  $\frac{\partial \by_{r}}{\partial \bx_{\alpha}} = R_{\alpha r}.$
 Fix $\alpha \in \{1, 2\}$. Then,
\begin{align}
\D_{x_\alpha} \mathcal{R}[f](\bx) = \D_{x_\alpha} f(R^{*}\bx) &= R_{\alpha r} \mathcal{R}[\D_{y_r} f](\bx)\ = \{R \mathcal{R} (\nabla f)\}_{_\alpha}.
\end{align}
\end{proof}

\subsubsection{Symmetry implies $\det\mathcal{M}(\mu^{(1)},\bkappa)$ has no  linear in $\bkappa$ terms for $|\bkappa|\ll1$}\label{no-lin}

\begin{prop}\label{no-linear}
\begin{equation} \mathcal{M}^{(0)}(\mu^{(1)},\bkappa) = \left(\ \mu^{(1)} - \kappa \cdot \kappa\ \right)\ \times\ {\rm I}_{_{2\times2}}
\label{M0-1}\end{equation}
and therefore
\begin{equation} 
\mathcal{M}(\mu^{(1)},\bkappa) = \left(\ \mu^{(1)} - \kappa \cdot \kappa\ \right)\ \times\ {\rm I}_{_{2\times2}}
+ \mathcal{M}^{(1)}(\mu^{(1)},\bkappa) \ .
\label{M-1}
\end{equation}
 \end{prop}

Recall from the hypotheses of Theorem \ref{quad-disp} that $\Phi_1\in L^2_{\bM,+i}$ and $\Phi_2\in L^2_{\bM,-i}$ and therefore
 $\mathcal{R}[\Phi_q](\bx)=i^{2q-1}\ \Phi_q(\bx)$, $q=1,2$.

\begin{prop}\label{f-gradf} For $j_{1}, j_{2} = 1, 2$,
\[
\langle \Phi_{j_{1}}, \nabla \Phi_{j_{2}} \rangle_{L^{2}(\Omega)} = {\bf 0}.\]
\end{prop}
\begin{proof} Choose $j_{1}, j_{2} \in \{1, 2\}$. Using that  $\mathcal{R}$ is unitary and Lemma \ref{rot-lemma}, we have
\begin{align*}
\langle \Phi_{j_{1}}, \nabla_{\by}\Phi_{j_{2}} \rangle_{_{L^2(\Omega_\by)}} &=
 \langle \mathcal{R}[\Phi_{j_{1}}], \mathcal{R}[\nabla_{\by}\Phi_{j_{2}}] \rangle_{_{L^2(\Omega_\bx)}}
= \langle \mathcal{R}[\Phi_{j_{1}}], R^{*}\nabla_{\bx} \mathcal{R}[\Phi_{j_{2}}]\rangle_{_{L^2(\Omega_\bx)}} 
= R^{*}\langle \mathcal{R}[\Phi_{j_{1}}], \nabla_{\bx} \mathcal{R}[\Phi_{j_{2}}] \rangle_{_{L^2(\Omega_\bx)}}\\
&= R^{*}\langle i^{2j_{1}-1}\Phi_{j_{1}}, \nabla_{\bx} i^{2j_{2}-1}\Phi_{j_{2}}\rangle_{_{L^2(\Omega_\bx)}} =  i^{2(j_{2}-j_{1})}R^{*}\langle \Phi_{j_{1}}, \nabla_{\bx} \Phi_{j_{2}} \rangle_{_{L^2(\Omega_\bx)}}.
\end{align*}
It follows that  either $i^{2(j_{2}-j_{1})}$ is an eigenvalue of $R$ or $\langle \Phi_{j_{1}}, \nabla_{\bx} \Phi_{j_{2}} \rangle={\bf 0}$.
But the eigenvalues of $R$ are $\pm i$, and since $j_2-j_1$ is an integer,  $i^{2(j_{2}-j_{1})}$ is real. We conclude that
$\langle \Phi_{j_{1}}, \nabla_{\bx} \Phi_{j_{2}} \rangle={\bf 0}$ for all $j_1, j_2=1,2$. The proof of Proposition \ref{f-gradf} is complete. 

\begin{remark}\label{interim1}
As observed in Section \ref{reduction}, $\|C^{(j)}(\bkappa,\mu^{(1)})\|_{H^1}\lesssim |\bkappa|+|\bkappa|^2+|\mu^{(1)}|$, we have from Proposition \ref{no-linear} that $\mathcal{M}(\mu^{(1)},\bkappa) = (\mu^{(1)}-|\bkappa|^2)\ I_{2\times2}+\mathcal{O}_{2\times2}(|\bkappa|^2+|\bkappa|^3+|\mu^{(1)}|\ |\bkappa|)$. Therefore, $\det\mathcal{M}(\mu^{(1)},\bkappa)=0$ has solutions $\mu^{(1)}_\pm=\mathcal{O}(|\kappa|^2)$. We next obtain the precise quadratic dependence on $\bkappa$ of  $\mathcal{M}^{(1)}(\mu^{(1)},\bkappa)$ and then give a more precise expansion of solutions to $\det\mathcal{M}(\mu^{(1)},\bkappa)=0$.
\end{remark}


\subsubsection{Quadratic in $\kappa$ terms of $\det\mathcal{M}(\mu^{(1)},\kappa)$ for $|\bkappa|\ll1$}
\label{quad-kap}
 
 We next expand $\mathcal{M}_{j_1,j_2}^{(1)}(\mu^{(1)},\bkappa)$ for $|\bkappa|$ and $|\mu^{(1)}|$ small. Recall first \eqref{cj}, 
the relations listed after \eqref{M1-def}.  Then,   $\widetilde{Q}_{\perp} C^{(j)}[\bkappa,\mu^{(1)}]=C^{(j)}[\bkappa,\mu^{(1)}]$, where 
\begin{align}
C^{(j)}[\bkappa,\mu^{(1)}] =& \left( I + \mathscr{R}(\mu_S)  \widetilde{Q}_{\perp} \left( -2i \kappa  \cdot \nabla  + \kappa \cdot \kappa - \mu^{(1)} \right) \right)^{-1} \circ  \left( \mathscr{R}(\mu_S) \widetilde{Q}_{\perp} \left( 2i \kappa \cdot \nabla \right) \Phi_j \right)\nn\\
&=\ \left(I + \mathcal{O}_{_{L^2\to L^2}}(\mid \kappa \mid + \mid \mu^{(1)} \mid )\right)\circ \left( \mathscr{R}(\mu_S) \widetilde{Q}_{\perp} \left( 2i \kappa \cdot \nabla \right) \Phi_j \right).
\label{Cj-exp}\end{align}
Furthermore, recalling that  $\widetilde{Q}_\perp\partial_{x_l}\Phi_m=\partial_{x_l}\Phi_m$ (Proposition \ref{f-gradf}), we have
\[ \mathcal{M}_{j_1,j_2}^{(1)}(\mu^{(1)},\bkappa)\ =\ 
\left\langle \Phi_{j_1}, 2i \kappa \cdot \nabla C^{(j_2)}(\kappa, \mu^{(1)}) \right\rangle 
=\left\langle \widetilde{Q}_{\perp} 2i \kappa \cdot \nabla \Phi_{j_1}, \widetilde{Q}_{\perp} C^{(j_2)}(\kappa, \mu^{(1)}) \right\rangle\ .
\]
Therefore, by \eqref{Cj-exp} we have for $j_{1}, j_{2} = 1, 2$ and $\kappa\in\R^2$:
\begin{align}
\mathcal{M}_{j_1,j_2}^{(1)}(\mu^{(1)},\bkappa)\ &=\ 
 4 \sum_{l,m=1}^{2} \langle  \widetilde{Q}_{\perp} \partial_{x_l} \Phi_{j_{1}}, \mathscr{R}(\mu_S)  \widetilde{Q}_{\perp}\partial_{x_m} \Phi_{j_{2}}\rangle\ \kappa_{l}\ \kappa_{m}+ \mathcal{O}\left(\ |\kappa|^{3} + |\mu^{(1)}|\  |\kappa|\ \right)\nn\\
&=4 \sum_{l,m=1}^{2} \langle \partial_{x_l} \Phi_{j_{1}}, \mathscr{R}(\mu_S) \partial_{x_m} \Phi_{j_{2}}\rangle\ \kappa_{l}\ \kappa_{m}+  \mathcal{O}\left(\ |\kappa|^{3} + |\mu^{(1)} |\  |\kappa|\ \right)\label{s-last}\\
&\equiv\ 4 \ \kappa^T\ A^{j_1,j_2}\ \kappa\ +\ \mathcal{O}\left(\ |\kappa|^{3} + |\mu^{(1)}|\  |\kappa|\ \right)\ , \label{Aj12_def}
 \end{align}
 where \eqref{Aj12_def} defines the matrix $A^{j_1,j_2}$ with entries:
 
\begin{equation}\label{A-def}
a_{l,m}^{j_{1},j_{2}} = \langle \partial_{x_l} \Phi_{j_{1}}, \mathscr{R}(\mu_S) \partial_{x_m} \Phi_{j_{2}}\rangle,
\end{equation}


We proceed now to use symmetry to deduce the structure of the matrices $A^{j_{1},j_{2}}$. 
\begin{lemma}\label{claim:M1entries} For fixed $j_{1}, j_{2} \in \{1,2\}$, we have the following:
\begin{equation}
R^{T} A^{j_{1},j_{2}}R\ =\ i^{2(j_{2}-j_{1})}A^{j_{1},j_{2}}  ,
\label{Aj1j2}
\end{equation}
where $R$ denotes the $\pi/2-$rotation matrix displayed in \eqref{Rdef}. Therefore, 
\begin{align*}
&j_1=j_2 \equiv j \ \quad \implies\quad R^{T}\ A^{j,j}\ R\ =\ A^{j,j}\\
&j_1\ne j_2\ \quad \implies\quad  R^{T}\ A^{j_1,j_2}\ R\ =\ -A^{j_1,j_2}.
\end{align*}
\end{lemma}
\begin{proof}
We will use Lemma \ref{rot-lemma}, 
$\mathcal{R}[\D_{x_l}f](\bx)=R_{nl}\mathcal{R}[\D_{y_n}f](\bx)$.  Since $\mathcal{R}$ is unitary and commutes with 
$\mathscr{R}(\mu_S)$ , 
\begin{align*}
\kappa^{T} A^{j_{1},j_{2}} \kappa &= \langle\partial_{y_l} \Phi_{j_{1}} , \mathscr{R}(\mu_S) \partial_{y_m} \Phi_{j_{2}}\rangle_{_{L^2(\Omega_\by)}} \kappa_{l}\kappa_{m}\nn\\   
&=  \langle \mathcal{R}[\partial_{y_l} \Phi_{j_{1}}],  \ \mathscr{R}(\mu_S)\ \mathcal{R}[\partial_{y_m} \Phi_{j_{2}}]\rangle_{_{L^2(\Omega_\bx)}}\  \kappa_{l} \kappa_{m} \\
&= \langle R_{nl}  \partial_{x_n}  \mathcal{R}[\Phi_{j_{1}}] ,  \mathscr{R}(\mu_S)\ R_{qm}\mathcal{R}[  \partial_{x_q} \Phi_{j_{2}}]\rangle_{_{L^2(\Omega_\bx)}}\  \kappa_{l} \kappa_{m}   \\
&=    \langle  \partial_{x_n}i^{2j_{1}-1} \Phi_{j_{1}}, \mathscr{R}(\mu_S) \partial_{x_q} i^{2j_{2}-1} \Phi_{j_{2}}\rangle_{_{L^2(\Omega_\bx)}}\ R_{nl} \kappa_{l} R_{qm} \kappa_{m} \\
&=  i^{2(j_{2}-j_{1})}   \langle  \partial_{x_n}\Phi_{j_{1}}  \mathscr{R}(\mu_S) \partial_{x_q} \Phi_{j_{2}}\rangle_{_{L^2(\Omega_\bx)}}
(R\kappa)_{n}  (R\kappa)_{q}\\
&= i^{2(j_{2}-j_{1})} (R\kappa)^{T}  A^{j_{1},j_{2}}  (R\kappa ).
\end{align*}
Since $\kappa$ is arbitrary,  $ A^{j_{1},j_{2}}  =  i^{2(j_{2}-j_{1})} R^T\ A^{j_{1},j_{2}}\ R$.
\end{proof}


\begin{lemma}\label{Astructure} Assume $R$ is the $\pi/2-$rotation matrix, \eqref{Rdef}, and $A=(a_{ij})$. Then, 
\begin{align}
(j_1=j_2)\quad &\  R^T\ A\ R\ =\ A\qquad  \implies\qquad A=\begin{pmatrix}  a_{11} & a_{12}\\ -a_{12} & a_{11}\end{pmatrix}\nn\\
(j_1\ne j_2)\quad &\  R^T\ A\ R\ =\ -A\qquad  \implies\qquad A=\begin{pmatrix}  a_{11} & a_{12}\\ a_{12} & -a_{11}\end{pmatrix}. \nn
\end{align}
\end{lemma}


\begin{claim}\label{claim:pwiseH} Let $A^\dagger$ conjugate-transpose of $A$. Then, 
\[
  (A^{1,1})^{\dagger} = A^{1,1},\; (A^{2,2})^{\dagger} = A^{2,2},\text{ and } (A^{2,1})^{\dagger} = A^{1,2}.
\]
 \end{claim}
\begin{proof} Pick ${j_{1}, j_{2}} \in \{1,2\}$ and $l, m \in \{1, 2\}$. 
\begin{align*} 
(A^{j_{1}, j_{2}})^{\dagger}_{l,m} = (\overline{a^{j_{1},j_{2}}_{m,l}} ) &= \left[\ \overline{\langle \partial_{x_m} \Phi_{j_{1}}, \mathscr{R}(\mu_S)  \partial_{x_l}  \Phi_{j_{2}}\rangle}\ \right]\nn\\ 
&=\ \left[\ \langle \partial_{x_l}  \Phi_{j_{2}},  \mathscr{R}(\mu_S) \partial_{x_m} \Phi_{j_{1}}\rangle\ \right] 
= (a^{j_{2},j_{1}}_{l,m}) = (A^{j_{2}, j_{1}})_{l,m},
\end{align*} 
and therefore $(A^{j_{1}, j_{2}})^{\dagger} = A^{j_{2},j_{1}}$.
\end{proof}

\begin{claim} $A^{1,1} =(A^{2,2})^{T}$. In particular, $a^{1,1}_{11}=a^{2,2}_{11}$.
\end{claim}
\begin{proof}
Recall $\Phi_{1}(\bx) = (\mathcal{P}\circ\mathcal{C})[\Phi_1](\bx)=\conjugatet{\Phi_{2}(-\bx)}$. For  $l, m \in \{1, 2\}$,
using that $\D_{x_l}(\mathcal{P}\circ\mathcal{C})=-(\mathcal{P}\circ\mathcal{C})\D_{y_l}$, we have
\begin{align*}
a^{1,1}_{l,m} 
&= \left\langle \partial_{x_l}  \Phi_{1},  \mathscr{R}(\mu_S)\  \partial_{x_m} \Phi_{1} \right\rangle \\
&= \left\langle \partial_{x_l}  (\mathcal{P}\circ\mathcal{C})[\Phi_2],  \mathscr{R}(\mu_S)\  \partial_{x_m} (\mathcal{P}\circ\mathcal{C})[\Phi_2] \right\rangle\\
&= \left\langle  (\mathcal{P}\circ\mathcal{C})[\partial_{y_l} \Phi_2],  \mathscr{R}(\mu_S)\   (\mathcal{P}\circ\mathcal{C})[\partial_{y_m}\Phi_2] \right\rangle\\
&= \left\langle  \mathcal{C}[\partial_{y_l} \Phi_2],  \mathscr{R}(\mu_S)\  \mathcal{C}[\partial_{y_m}\Phi_2] \right\rangle
\quad (\ =\overline{a^{2,2}_{lm}}\ )\\
&= \left\langle \partial_{y_m} \Phi_2,  \mathscr{R}(\mu_S)\  \partial_{y_l}\Phi_2 \right\rangle\ =\ a^{2,2}_{m,l} .
\end{align*} 
Therefore, $A^{1,1} = \conjugatet{A^{2,2}} = (A^{2,2})^{T}$.
\end{proof}
\medskip

By \eqref{Aj12_def} we have
\begin{equation}
\mathcal{M}^{(1)}(\mu^{(1)},\kappa)\ =\
4\ \begin{pmatrix}
     \kappa^{T}A^{1,1} \kappa  & \kappa^{T}A^{2,1} \kappa \\
     \kappa^{T} A^{1,2} \kappa & \kappa^{T} A^{2,2} \kappa
\end{pmatrix}\ +\ \mathcal{O}_{2\times2}(\ |\kappa|^{3} + |\mu^{(1)}| |\kappa|\ )
\label{M1-expand}
\end{equation}

\noindent Simplifying the leading term in \eqref{M1-expand} we observe, by the above claims, that $a_{1,1}^{1,1} \in \R$ and
 \[A^{1,1} =\begin{pmatrix}
    a_{1,1}^{1,1} & a_{1,2}^{1,1} \\
    -a_{1,2}^{1,1}  & a_{1,1}^{1,1}
\end{pmatrix}; \;\;  A^{1,2} = \overline{A^{2,1}} = \begin{pmatrix}
   a_{1,1}^{1,2} & a_{1,2}^{1,2}  \\
    a_{1,2}^{1,2}  & -a_{1,1}^{1,2}
\end{pmatrix}; \;\;
 A^{2,2} = \begin{pmatrix}
    a_{1,1}^{1,1}& -a_{1,2}^{1,1} \\
    a_{1,2}^{1,1} & a_{1,1}^{1,1}
\end{pmatrix}.\]

Hence, 
\begin{align}\label{M1-matrix}
 \M^{(1)}(\mu^{(1)},\kappa) &= \ \begin{pmatrix}
    \alpha\ (\kappa_1^2+\kappa_2^2)  & \gamma\ (\kappa_{1}^{2} - \kappa_{2}^{2}) + 2\beta\ \kappa_{1}\kappa_{2} \\
    & \\
    \overline{\gamma}\ (\kappa_{1}^{2} - \kappa_{2}^{2}) + 2\overline{\beta}\ \kappa_{1}\kappa_{2}
    & \alpha\ (\kappa_1^2+\kappa_2^2)
\end{pmatrix}
 +\ \mathcal{O}_{_{2\times2}}(\ |\kappa|^{3} + |\mu^{(1)}| |\kappa|\ ), 
\end{align}
where, by \eqref{A-def}, 
\begin{align}
\alpha\ &=\  4a_{1,1}^{1,1}\ =\ 
4\ \langle \partial_{x_1} \Phi_{1}, \mathscr{R}(\mu_S) \partial_{x_1} \Phi_{1}\rangle\ ,\label{albega}\\
 \beta\ &=\ 4a_{1,2}^{1,2}\ =\ 4\ \langle \partial_{x_1} \Phi_{1}, \mathscr{R}(\mu_S) \partial_{x_2} \Phi_{2}\rangle\ ,\nn\\
\gamma\ &=\ 4a_{1,1}^{1,2}\ =\ 4\ \langle \partial_{x_1} \Phi_{1}, \mathscr{R}(\mu_S) \partial_{x_1} \Phi_{2}\rangle .
\nn
\end{align}
\begin{remark} Observe by Claim \eqref{claim:pwiseH}, the terms $\alpha = 4a^{1,1}_{1,1}$ are real.
\end{remark}


\subsection{Symmetries of dispersion maps $\bk\mapsto \mu_\pm(\bk)$ for $\bk$ near a vertex of $\brill$}\label{s-disp}

Before deriving a detailed picture of the local character of dispersion surfaces near vertices of $\brill$, we prove a general result on the structure of dispersion surfaces in a neighborhood of a vertex quasi-momentum of $\brill$ for which the eigenvalue problem has a degenerate eigenvalue. 

Assume that the potential $V$ is admissible in the sense of Definition \ref{def:sq-pot}, where without loss of generality we take 
the centering $\bx_c=0$. 
Suppose that in addition that $V$ is reflection invariant, {\it i.e.} $ V(\bx)=V(x_1,x_2)=V(x_2,x_1)=V(\rho\bx)$, where 
\begin{equation}
\rho\bx= \begin{pmatrix} 0&1\\1&0\end{pmatrix}\ \begin{pmatrix}x_1\\ x_2\end{pmatrix}
= \begin{pmatrix}x_2\\ x_1\end{pmatrix}\ .
\label{rhodef}\end{equation}
To prove  Corollary \ref{mu-rho}  we need to show that $\gamma = 4\ \langle \partial_{x_1} \Phi_{1}, \mathscr{R}(\mu_S) \partial_{x_1} \Phi_{2}\rangle$ (see \eqref{albega}) vanishes.  We first deduce that either $\beta = 4\ \langle \partial_{x_1} \Phi_{1}, \mathscr{R}(\mu_S) \partial_{x_2} \Phi_{2}\rangle $ or $\gamma$ vanishes based on the following symmetry argument; then, we prove that, in fact, $\gamma = 0$ using the properties of $\rho$. We write
\begin{equation}
T_\rho[f] = f(\rho^{\ast} \bx) = f(\rho \bx) =f(x_2, x_1).
\label{Trho}
\end{equation}

\begin{prop}\label{near-M-sym}
 Let $V$ be an admissible potential in the sense of Definition \ref{def:sq-pot}; we take $\bx_c=0$ without any loss of generality. In particular, $V(R^*\bx)=V(\bx)$ and $V(-\bx)=V(\bx)$.  Assume the hypotheses of Theorem \ref{quad-disp} which imply that 
$H_V$ has a degenerate (multiplicity two) $L^2_\bM-$eigenvalue, $\mu_S\in\R$.

Then, for all  $\bk=\bM+\bkappa$ with $0<|\bkappa|<\kappa_0$ sufficiently small there exist two eigenvalues
given by $\mu_\pm(\bM+\bkappa)=\mu_S\pm \mu^{(1)}(\kappa)$. 

\begin{enumerate}
\item For all $0<|\bkappa|<\kappa_0$, we have
\begin{equation}
 \{\mu_-(\bM+\bkappa),\mu_+(\bM+\bkappa)\}\ =\ \{\mu_-(\bM+R\bkappa),\mu_+(\bM+R\bkappa)\}\ .
 \label{nearM-R}
 \end{equation}
 \item Suppose in addition that $V(\rho^*\bx)=V(\rho \bx)$\ (recall $\rho=\rho^*$). Then, for $0<|\bkappa|<\kappa_0$, 
   we have
  \begin{equation}
 \{\mu_-(\bM+\bkappa),\mu_+(\bM+\bkappa)\}\ =\ \{\mu_-(\bM+\rho\bkappa),\mu_+(\bM+\rho\bkappa)\}\ .
 \label{nearM-rho}
 \end{equation}
\end{enumerate}
\end{prop}

\nit{\it Proof of Proposition \ref{near-M-sym}:} Let $(\mu_\kappa,\psi)$ denote an $L^2_{\bM+\kappa}$ eigenpair of $-\Delta+V$, where $\mu_\kappa$ is assumed to be near $\mu_S$. Then,  $\mu_\kappa=\mu_-(\bM+\kappa)$ or 
$\mu_\kappa=\mu_+(\bM+\kappa)$. Consider now $\widetilde{\psi}(\bx)\equiv \mathcal{R}[\psi](\bx)=\psi(R^*\bx)$. Note that 
 $(-\Delta+V)\widetilde\psi=\mu_\kappa\widetilde\psi$ since $\mathcal{R}$ commutes with $-\Delta+V$.  Moreover, for all $\bv\in\Z^2$, we have
  $\widetilde{\psi}(\bx+\bv)=\psi(R^*(\bx+\bv))=\psi(R^*\bx+R^*\bv)=e^{i(\bM+\kappa)\cdot R^*\bv}\widetilde\psi(\bx)=e^{i(R\bM+R\kappa)\cdot\bv}\widetilde\psi(\bx)=e^{i(\bM+R\kappa)\cdot\bv}\widetilde\psi(\bx)$, where we have used that $R\bM\in \bM+\Lambda$. Therefore, $\mu_\kappa$ is 
  a $L^2_{\bM+R\kappa}$ eigenvalue in a neighborhood of $\mu_S$. Hence, 
  $ \{\mu_-(\bM+\bkappa),\mu_+(\bM+\bkappa)\}\subset  \{\mu_-(\bM+R\bkappa),\mu_+(\bM+R\bkappa)\}$. To prove the reverse inclusion, assume $(\hat\mu_{\kappa},\phi)$ is an $L^2_{\bM+R\kappa}$ eigenpair with $\hat\mu_{\kappa}$ near $\mu_S$. 
  Now let $\widetilde{\phi}=\mathcal{R}^3[\phi](\bx)$ and note that $\widetilde{\phi}\in L^2_{\bM+\kappa} $.
Hence $(\hat\mu_{\kappa},\widetilde{\phi})$ is an $L^2_{\bM+\kappa}-$eigenpair of $-\Delta+V$.
   Therefore, 
  $\{\mu_-(\bM+R\bkappa),\mu_+(\bM+R\bkappa)\}
  \subset\{\mu_-(\bM+\bkappa),\mu_+(\bM+\bkappa)\}$ 
  and the proof of Part 1 is complete. 
   The proof of Part 2 is analogous. This completes the proof of Proposition \ref{near-M-sym}.
\subsection{Local behavior of degenerate dispersion surfaces near the $\bM-$point}\label{local-bhv}
 
 We need to study the solutions of $\det(\M(\mu^{(1)}, \bkappa) )=0$ for $\kappa$ in a neighborhood of zero.  Our strategy is based on a general approach used in \cite{FLW-CPAM:17} (Section 13). We extend $\M(\mu^{(1)}, \bkappa)$ to be defined as a matrix-valued analytic function
  of $(\mu^{(1)},\kappa)$ in a neighborhood of the origin in $\C\times\C^2$ and which agrees
   with $\M(\mu^{(1)}, \bkappa)$ as defined above for $(\mu^{(1)},\kappa)\in\R\times\R^2$:

\begin{align}
\M(\mu^{(1)}, \bkappa) = 
\begin{pmatrix}
   m_{11}(\kappa,\mu^{(1)}) & m_{12}(\kappa,\mu^{(1)}) \\
    &\\
    \overline{\ m_{12}(\overline{\kappa},\overline{\mu^{(1)}})\ } &  m_{22}(\kappa,\mu^{(1)})\end{pmatrix} 
. \label{M-ext} \end{align}

Here, 
\begin{align*}
m_{11}(\kappa,\mu^{(1)})\ &=\ (\alpha-1) \mid \kappa \mid^{2}  + \mu^{(1)}\ +\ \mathcal{O}_{_{2\times2}}(\ |\kappa|^{3} + |\mu^{(1)}| |\kappa|\ )\\
m_{12}(\kappa,\mu^{(1)})\ &=\ \gamma(\kappa_{1}^{2} - \kappa_{2}^{2})+ 2\beta \kappa_{1}\kappa_{2}
+\ \mathcal{O}_{_{2\times2}}(\ |\kappa|^{3} + |\mu^{(1)}| |\kappa|\ )\\
m_{21}(\kappa,\mu^{(1)})\ &=\ \overline{\ m_{12}(\overline{\kappa},\overline{\mu^{(1)}})\ }\ =\ 
 \overline{\gamma}(\kappa_{1}^{2} - \kappa_{2}^{2})+ 2\overline{\beta} \kappa_{1}\kappa_{2}\ 
+\ \mathcal{O}_{_{2\times2}}(\ |\kappa|^{3} + |\mu^{(1)}| |\kappa|\ )\\
m_{22}(\kappa,\mu^{(1)})\ &=\ m_{11}(\kappa,\mu^{(1)})
\end{align*}
Note in particular that $\M(\mu^{(1)}, \kappa)$ given by \eqref{M-ext} is Hermitian for real $\mu^{(1)}$ and $\kappa$.

\medskip

We first make the simple change of variables 
\begin{equation}
\nu\ \equiv\ \left(\alpha-1\right)\left(\kappa_1^2+\kappa_2^2\right)\ +\ \mu^{(1)}. \label{nu-def}\end{equation}
Define
\begin{equation}
\tM(\nu,\kappa)=\tM(\nu,\kappa_1,\kappa_2)\equiv \M(\ \nu-\left(\alpha-1\right)\left(\kappa_1^2+\kappa_2^2\right),\kappa_1,\kappa_2\ ),
\label{tM-def}\end{equation}
and study the equivalent equation of $\det(\M(\mu,\kappa))=0$ for $\nu$:
\begin{equation}
\det(\tM(\nu,\kappa))=0.\label{det-tM}
\end{equation}
The entries of $\tM(\nu, \bkappa)$ are analytic functions of $(\nu,\kappa)$ in a neighborhood of the origin
 in $\C_\nu\times\C_\kappa^2$.
The matrix $\tM(\mu,\kappa)$ has the expansion
\begin{align}
\tM(\nu, \bkappa) = \tM_{\rm{app}}(\nu, \bkappa)
+\ \mathcal{O}_{_{2\times2}}(\ |\kappa|^{3} + |\nu|\ |\kappa|\ ) \ ,
\label{tM-expand}\end{align} 
where
\begin{equation} 
\tM_{\rm{app}}(\nu,\bkappa) =
\begin{pmatrix}
    \nu & q(\kappa_1,\kappa_2)  \\
    &\\
    q_\natural(\kappa_1,\kappa_2) &  \nu 
     \end{pmatrix} 
\label{tM-approx}
\end{equation}
and
\begin{align}
q(\kappa_1,\kappa_2)\ &\equiv\ \gamma(\kappa_{1}^{2} - \kappa_{2}^{2})+ 2\beta \kappa_{1}\kappa_{2}
\label{q-def}\\
q_\natural(\kappa_1,\kappa_2)\ &\equiv\ \overline{q(\overline{\kappa_1},\overline{\kappa_2})}=
\overline{\gamma}(\kappa_{1}^{2} - \kappa_{2}^{2})+ 2\overline{\beta} \kappa_{1}\kappa_{2}.
\label{qnat-def}\end{align}
Calculating the determinant of $\tM(\nu,\kappa)$ we obtain:
\begin{align}
D(\nu,\kappa)\ \equiv\ \det\left(\tM(\nu,\kappa)\right)\ =\ \nu^2\ -\ q(\kappa)\ q_\natural(\kappa)\ +\ g(\nu,\kappa)\ ,
\label{D-def}\end{align}
where $g(\nu,\kappa)$ and hence $D(\nu,\kappa)$ are  analytic in a neighborhood of the origin in $\C\times\C^2$. Note also that $g(\nu,\kappa)$ and $\partial_\nu g(\nu,\kappa)$ satisfy the following bounds for $|\kappa|$ and $|\nu|$ small:
\begin{align}
\label{g-bound}
|g(\nu,\kappa)|\ &\le\ C_g\ \left(\ |\kappa|^5\ +\ |\nu|\ |\kappa|^3\ +\ |\nu|^2\ |\kappa|\ \right)\\
|\partial_\nu g(\nu,\kappa)|\ &\le\ C^\prime_g\ \left(\ |\kappa|^3\ +\ |\nu|\ |\kappa|\ \right)
\label{gp-bound}
\end{align}
for some positive constants $C_g$ and $C^\prime_g$.

The problem of finding eigenvalues $\mu=\mu_S+\nu$ near $\mu_S$ for $\bk=\bM+\kappa$ near $\bM$  has been reduced to the study of the solutions to the equation 
\[D(\nu,\kappa)\ =\ 0\]
 for $\kappa$ near $(0,0)\in\R^2_\kappa$. We shall study the roots of $D(\nu,\kappa)$ using Rouch\'e's Theorem. 
 
 Consider $\kappa\in\R^2$ such that $|\kappa|<\kappa_{\rm max}$. We shall eventually take 
 $\kappa_{\rm max}$ to be small. For such $\kappa$ we have
 \[ \Big|\nu^2-q(\kappa)\ q_\natural(\kappa)\Big|\ \ge\ \nu^2-C_{\gamma,\beta}|\kappa_{\rm max}|^2\ ,\]
 where $C_{\gamma,\beta}$ is a positive constant depending only on $\gamma$ and $\beta$.
 Note also that for $\nu$ constrained to the circle $|\nu|=2C_{\gamma,\beta}\kappa_{\rm max}^2$ we have the lower bound:
 \begin{equation}
  \Big|\nu^2-q(\kappa)\ q_\natural(\kappa)\Big|\ \ge\ C_{\gamma,\beta}|\kappa_{\rm max}|^2\ .\label{lowerb}\end{equation} 
 Thus, if $|\nu|=2C_{\gamma,\beta}\kappa_{\rm max}^2$, then 
 \[ |g(\nu,\kappa)|\le C_{\gamma,\beta,g}\kappa_{\rm max}^5\ .\]
Note also that $\ C_{\gamma,\beta}\ \kappa_{\rm max}^2>C_{\gamma,\beta,g}\ \kappa_{\rm max}^5$
  provided $\kappa^3_{\rm max}<C_{\gamma,\beta}/C_{\gamma,\beta,g}$.
Therefore,  if we choose $\kappa_{\rm max}$ to be any  constant satisfying  
\begin{equation}
0<\kappa_{\rm max}<\frac12\left(C_{\gamma,\beta}/C_{\gamma,\beta,g}\right)^{\frac13}\ ,
\label{km-choose}
\end{equation}
 then for $|\kappa|\le\kappa_{\rm max}$ we have:
   \begin{equation}
    |\nu|=2C_{\gamma,\beta}\kappa_{\rm max}^2\quad \implies \quad |g(\nu,\kappa)|\ <\ \Big|\ \nu^2\ -\ q(\kappa)\ q_\natural(\kappa)\ \Big|\ \label{compare}
    \end{equation}
   for all $|\nu|=2C_{\gamma,\beta}\kappa_{\rm max}^2$. Therefore by Rouch\'e's Theorem, the functions 
    \[ \nu^2-q(\kappa)\ q_\natural(\kappa)\quad {\rm and}\quad D(\nu,\kappa)=\nu^2-q(\kappa)\ q_\natural(\kappa)+g(\nu,\kappa)\]
     have the same number of zeros in the disc: $|\nu|<2C_{\gamma,\beta} \kappa_{\rm max}^2$.
     We denote these zeros: $\nu_+(\kappa)$ and $\nu_-(\kappa)$.   For real $\kappa$ these zeros are real by self-adjointness and we have:
          \[ \nu_+(\kappa),\ \nu_-(\kappa)\ \in\ [-C_{\gamma,\beta} \kappa_{\rm max}^2,C_{\gamma,\beta} \kappa_{\rm max}^2]\ ,\ \ |\kappa|<\kappa_{\rm max},\ \kappa\in\R^2.\]
     
Next, observe by a residue calculation that  for $l=1,2$:
\begin{equation}\label{residue}
\left( \nu_+(\kappa) \right)^l\ +\ \left( \nu_-(\kappa) \right)^l\ =\ 
\frac{1}{2\pi i}\ \int_{|\nu|=2C_{\gamma,\beta} \kappa_{\rm max}^2}\ \frac{\nu^l\ \partial_\nu D(\nu,\kappa)}{D(\nu,\kappa)}\ d\nu .
\end{equation} 
Since $\partial_\nu D(\nu,\kappa)\ =\ 2\nu+\partial_\nu g(\nu,\kappa)$, we have
\begin{align}\label{residue2}
&\left( \nu_+(\kappa) \right)^l\ +\ \left( \nu_-(\kappa) \right)^l\ 
\nn\\
&=\frac{1}{2\pi i}\ \int_{|\nu|=2C_{\gamma,\beta} \kappa_{\rm max}^2}\ 
\frac{2\nu^{l+1}}{\nu^2\ -\ q(\kappa)\ q_\natural(\kappa)\ +\ g(\nu,\kappa)}\ d\nu \\
& \qquad +\ \frac{1}{2\pi i}\ \int_{|\nu|=2C_{\gamma,\beta} \kappa_{\rm max}^2}\ 
\frac{\nu^l\partial_\nu g(\nu,\kappa)\ }{\nu^2\ -\ q(\kappa)\ q_\natural(\kappa)\ +\ g(\nu,\kappa)}\ d\nu\nn\\
&=\frac{1}{2\pi i}\ \int_{|\nu|=2C_{\gamma,\beta} \kappa_{\rm max}^2}\ 
\frac{2\nu^{l+1}}{\nu^2\ -\ q(\kappa)\ q_\natural(\kappa)}\ d\nu\nn\\
&\qquad -\ \frac{1}{2\pi i}\ \int_{|\nu|=2C_{\gamma,\beta} \kappa_{\rm max}^2}\ 
\frac{2\nu^{l+1}\ g(\nu,\kappa)}{\left(\nu^2\ -\ q(\kappa)\ q_\natural(\kappa)+g(\nu,\kappa)\right)\cdot
\left(\nu^2\ -\ q(\kappa)\ q_\natural(\kappa)\right) }\ d\nu\nn\\
&\qquad +\ \ \frac{1}{2\pi i}\ \int_{|\nu|=2C_{\gamma,\beta} \kappa_{\rm max}^2}\ 
\frac{\nu^l\partial_\nu g(\nu,\kappa)\ }{\nu^2\ -\ q(\kappa)\ q_\natural(\kappa)\ +\ g(\nu,\kappa)}\ d\nu\nn\\
\end{align} 
We can use the identity, for $r > a$, 
\begin{equation}
\frac{1}{2\pi i}\ \int_{|z|=r}\ \frac{z^{l+1}}{z^2-a^2}\ dz\ =\ 
\begin{cases} 
0 & \textrm{if}\ l=1\\
a^2 & \textrm{if}\ l=2,
\end{cases}
\end{equation}
to evaluate the first integral and bound the remaining two integrals using \eqref{g-bound} and \eqref{gp-bound}. This gives
\begin{equation}\label{residue1}
\left( \nu_+(\kappa) \right)^l\ +\ \left( \nu_-(\kappa) \right)^l\ =\ 
\begin{cases}
\mathcal{O}(\kappa_{\rm max}^4\ |\kappa|)\ \ &\ l=1\\
2q(\kappa)\ q_\natural(\kappa)\ +\ \mathcal{O}(\kappa_{\rm max}^6\ |\kappa|)\ \ &\ l=2
\end{cases}
\end{equation}
for $\kappa_{\rm max}$ sufficiently small. 

Now note
\begin{equation}
\nu_+(\kappa)\cdot\nu_-(\kappa)\ =\ 
\frac12  \left(\ \nu_+(\kappa)  +\  \nu_-(\kappa)\ \right)^2
\ -\ \frac12\left[\ \left( \nu_+(\kappa) \right)^2  +\  \left( \nu_-(\kappa) \right)^2\ \right]. 
\end{equation}
Therefore, $\kappa\mapsto \nu_+(\kappa)\cdot\nu_-(\kappa)$  is analytic in a $\C^2$ neighborhood of $\kappa=0$. Moreover, 
we have 
\begin{align}
\nu_+(\kappa)\cdot\nu_-(\kappa) &=  -\frac12\left(\ 2q(\kappa)\ q_\natural(\kappa)\ +\ \mathcal{O}(\kappa_{\rm max}^6\ |\kappa|) \right)\ +\ \mathcal{O}(\kappa_{\rm max}^8|\kappa|^2)\nn\\
 &=\ -q(\kappa)\ q_\natural(\kappa)+\mathcal{O}(\kappa_{\rm max}^6\ |\kappa|) .
 \label{nu-prod}
\end{align}

Consider now the equation $\left(\nu-\nu_+(\kappa)\right)\ \left(\nu-\nu_-(\kappa)\right)=0$ satisfied
by $\nu=\nu_\pm(\kappa)$. 

\begin{lemma}\label{Dequiv}
 The roots of $D(\nu,\kappa)=0$ in the disc $|\kappa|<2C_{\gamma,\beta} \kappa_{\rm max}^2$ coincide with the roots of the quadratic equation: $\nu^2-t(\kappa)\nu+d(\kappa)=0$, where
 $t(\kappa)$ and $d(\kappa)$ are analytic and 
 \begin{align}
 &t(\kappa)= \nu_+(\kappa)\ +\  \nu_-(\kappa)\  = \mathcal{O}(\kappa_{\rm max}^4\ |\kappa|)\nn\\
 &d(\kappa)= \nu_+(\kappa)\cdot\nu_-(\kappa)\ =\ -q(\kappa)\ q_\natural(\kappa)+ \mathcal{O}(\kappa_{\rm max}^6\ |\kappa|).\nn
 \end{align}
  \end{lemma}
\medskip

Solving for $\nu$ we have, for $|\kappa|<2C_{\gamma,\beta} \kappa_{\rm max}^2$,  two roots: 
\begin{align}
\nu_\pm(\kappa)\ &=\ \frac{1}{2}t(\kappa)\pm\sqrt{-d(\kappa)+\frac14 t^2(\kappa)}\\
&= c_0(\kappa)\pm \sqrt{q(\kappa)\ q_\natural(\kappa)+c_1(\kappa)}\ ,
\end{align}
where $c_0(\kappa)$ and $c_1(\kappa)$ are analytic in $\kappa$ in a $\C^2$ neighborhood of the origin  satisfying:
\begin{equation}
|c_0(\kappa)|\ \lesssim\ \kappa_{\rm max}^4\ |\kappa|,\qquad
 |c_1(\kappa)|\ \lesssim\ \kappa_{\rm max}^6\ |\kappa|
\nn\end{equation}
for all $|\kappa|\le \kappa_{\rm max}$ and $\kappa_{\rm max}$ is any constant satisfying
 \eqref{km-choose}.
 
Since  $\kappa$ can be taken arbitrary and  $|\kappa|=\kappa_{\rm max}$ can be an arbitrarily small  positive number, it follows that the analytic functions $c_0(\kappa)=\mathscr{Q}_5(\kappa)$
 and  $c_1(\kappa)=\mathscr{Q}_7(\kappa)$ which satisfy, for $|\kappa|\to0$:
$\mathscr{Q}_r(\kappa)\ =\ \mathcal{O}(|\kappa|^r)$.

 If we now restrict to real $\kappa=(\kappa_1,\kappa_2)$, then from \eqref{q-def}-\eqref{qnat-def} we have that 
\begin{equation}
q(\kappa)\ q_\natural(\kappa)\ =\ |q(\kappa)|^2\ =\ \Big|\ \gamma(\kappa_{1}^{2} - \kappa_{2}^{2})+ 2\beta \kappa_{1}\kappa_{2}\ \Big|^2\ .
\end{equation}
Therefore, 
$\nu_\pm(\kappa)\ 
= \mathscr{Q}_5(\kappa)\ \pm \sqrt{\Big|\ \gamma(\kappa_{1}^{2} - \kappa_{2}^{2})+ 2\beta \kappa_{1}\kappa_{2}\ \Big|^2\ +\ \mathscr{Q}_7(\kappa)}.$

We finally return to \eqref{nu-def} which relates $\nu$ to our eigenvalue parameter $\mu$:
 $\mu(\bM+\kappa)=(1-\alpha)|\kappa|^2+\nu(\kappa)$.  
  Proposition \ref{near-M-sym} implies constraints, due to symmetry, on  
  the mappings $\bk\mapsto \mu_\pm(\bk)$ for $\bk$ near $\bM$.
Clearly $|q(\kappa)|^2$ is invariant under the $\pi/2$ rotation: $(\kappa_1,\kappa_2)\mapsto(-\kappa_2,\kappa_1)$, and by part (1) of Proposition \ref{near-M-sym} we must have: $c_0(\kappa)=\mathscr{Q}_6(\kappa)$
 and  $c_1(\kappa)=\mathscr{Q}_8(\kappa)$.  We therefore have
 \begin{equation}
 \mu_\pm(\bM+\kappa)\ =
  (1-\alpha)|\kappa|^2+\mathscr{Q}_6(\kappa)\ \pm \sqrt{\Big|\ \gamma(\kappa_{1}^{2} - \kappa_{2}^{2})+ 2\beta \kappa_{1}\kappa_{2}\ \Big|^2\ +\ \mathscr{Q}_8(\kappa)}\ .
 \label{nf-1} \end{equation}
This completes the proof of Theorem \ref{quad-disp}.
\end{proof}

\subsection{Dispersion Surfaces Near $\bM$ for $V$ Admissible and Reflection Invariant; Proof of Corollary \ref{mu-rho}} 
\label{rho-inv-exp}

 In addition to $V$ being admissible in the sense of Definition \ref{def:sq-pot}, we now assume that
 $V(x_1,x_2)$ is reflection invariant about the line $x_1=x_2$. 
 
  \begin{lemma}\label{rho-rot-lemma} Suppose $f\colon \mathbb{R}^{2} \rightarrow \mathbb{R}^{2}, f \in L^{2}(\Omega)$ satisfies 
  $T_\rho[f](\bx)\equiv f(\rho^*\bx)=f(\bx)$, where $\rho=\rho^*$ is the reflection (permutation) matrix mapping $(x_1, x_2) \rightarrow (x_2, x_1)$.
 Then,
\begin{align}\label{rho-rot-lemma-result}
\D_{x_n}\  T_\rho[f](\bx)=\rho_{nm}\ T_\rho[\D_{y_m}f](\bx)\ =\ T_\rho[\rho_{nm}\D_{y_m}f](\bx).
\end{align}
 \end{lemma}
\begin{proof}
Let $\by = \rho^{\ast} \bx$. Then $\by_{m} = \rho_{nm}\bx_{n}$, and furthermore $ \frac{\D\by_{m}}{\D_{\bx_n}} = \rho_{nm}$. Thus,
\begin{align*}
\{\nabla \rho[ f](\bx)\}_{n} &= \frac{\D}{\D{x_n}} f(\rho \bx) = \rho_{nm}\frac{\D}{\D{y_m}}f(\by)\Big|_{\by=\rho^*\bx}=\Big\{\rho[\nabla_{\by}f]\Big\}_{n}\Big|_{\by=\rho^*\bx} =\ 
\Big\{T_\rho [\rho\ \nabla f](\bx) \Big\}_{n}.
\end{align*}
\end{proof}

\begin{claim}\label{L2M-jump}
Let $\sigma \in \{\pm 1, \pm i\}$. If $\psi(\bx)$ solves the Floquet-Bloch eigenvalue problem with $V$ admissible and reflection-invariant, and  $\psi(\bx)\in L^2_{\bM, \sigma}$, then $\tilde{\psi} = \psi(x_{2},x_{1})$ is also a solution where $\tilde{\psi} \in L^2_{\bM, \sigma^{3}}$. That is, $\rho$ maps $L^{2}_{\bM, \sigma} \rightarrow L^{2}_{\bM, \sigma^{3}}.$ 
\end{claim}
\begin{proof}
First we'll show if $\psi(\bx) \in L^2_{\bM, \sigma}$, then $\tilde{\psi} = \psi(x_{2},x_{1}) \in L^2_{\bM, \sigma^{3}}$.
 Observe that $\rho$ is self-adjoint, and note 
\begin{equation}\label{conj-sym} \rho^{\ast} R^{\ast}\rho = R \end{equation}
since
\begin{align*}
R^{\ast}\rho  = \begin{pmatrix} 0 & -1 \\ 1 & 0 \end{pmatrix} \begin{pmatrix} 0 & 1 \\ 1 & 0 \end{pmatrix}
&= \begin{pmatrix} -1 & 0 \\ 0 & 1 \end{pmatrix}
= \begin{pmatrix} 0 & 1 \\ 1 & 0 \end{pmatrix}\begin{pmatrix} 0 & 1 \\ -1 & 0 \end{pmatrix} = \rho R.
\end{align*}

We seek to show \[\mathcal{R}[\tilde{\psi}] = \sigma^3\tilde{\psi},\] or, equivalently, $\mathcal{R}[\rho[\psi]](\bx) = \sigma^3 \rho[\psi](\bx).$ If that holds, then $\tilde{\psi} \in L^2_{\bM, \sigma^{3}}$.
\begin{align*}
\mathcal{R}[\tilde{\psi}] &= \mathcal{R}[\rho[\psi]](\bx) \\&= \psi(R^{\ast}\rho \bx) \stackrel{\text{\eqref{conj-sym}}}{=} \psi(\rho R\bx) 
= \psi(\rho (R^{\ast})^{3}\bx) \\&= \rho(\psi((R^{\ast})^{3}\bx)) = \rho(\mathcal{R}^{3}\psi(\bx)) \\
&= \sigma^3\rho[\psi](\bx) = \sigma^3\tilde{\psi}.
\end{align*}
Therefore, if  $\psi(\bx) \in L^2_{\bM, \sigma}$, then $\tilde{\psi} = \psi(x_{2},x_{1}) \in L^2_{\bM, \sigma^{3}}$.

\end{proof}

{ Now consider the setting of Theorem \ref{quad-disp}; $\mu_S$ is an eigenvalue of  $H=-\Delta+V$ acting in $L^2_{\bM}$ of multiplicity $2$. In particular, $\mu_S$ is simple $L^2_{\bM,i}$ eigenvalue with corresponding eigenfunction $\Phi_1$, and $\mu_S$ is a simple $L^2_{\bM,-i}$ eigenvalue with corresponding eigenfunction $\Phi_2$, with $\Phi_2(\bx)=\overline{\Phi_1(-\bx)}$. By Claim \ref{L2M-jump}, $\rho\Phi_1\in L^2_{\bM,-i}$
and since $\rho$ commutes with $H$ we have that $\rho\Phi_1$ is an $L^2_{\bM,-i}$ eigenfunction. Thus,
$
  \rho\Phi_1= e^{i\nu}\Phi_2,\ \textrm{for some}\ \nu\in\R$
  or equivalently $\rho e^{-i \frac{\nu}{2} } \Phi_1 = e^{i \frac{\nu}{2} } \Phi_2$.  Hence,
for the case $\nu \ne 0$ replace $\Phi_1$ by $e^{-i\frac{\nu}{2}} \Phi_1$ and $\Phi_2$ by $e^{i\frac{\nu}{2}} \Phi_2$, to obtain the relation
%
%
\begin{equation}
\rho\Phi_1=\Phi_2\ \label{rphi12}
\end{equation}
in all cases.
Recall from \eqref{albega} that 
\begin{align}
\alpha\ &=\  4a_{1,1}^{1,1}\ =\ 
4\ \langle \partial_{x_1} \Phi_{1}, \mathscr{R}(\mu_S) \partial_{x_1} \Phi_{1}\rangle\ ,\label{albega}\\
 \beta\ &=\ 4a_{1,2}^{1,2}\ =\ 4\ \langle \partial_{x_1} \Phi_{1}, \mathscr{R}(\mu_S) \partial_{x_2} \Phi_{2}\rangle\ ,\nn\\
\gamma\ &=\ 4a_{1,1}^{1,2}\ =\ 4\ \langle \partial_{x_1} \Phi_{1}, \mathscr{R}(\mu_S) \partial_{x_1} \Phi_{2}\rangle .
\nn
\end{align}
We have shown that $\alpha\in\R$. Using \eqref{rphi12} we have 
the following constraints on $\beta$ and $\gamma$:

\begin{claim}\label{crossTerms0} Assume  $[\rho,H]=0$. Then, 
\begin{align}
\beta\ &\in\ \R\label{b-Re}\\
\gamma\ &=\  i\tilde{\gamma},\ \quad\tilde{\gamma}\in\R. \label{g-Im}
\end{align}
\end{claim}

\begin{proof}  We first prove \eqref{b-Re}. Since  with $[\rho,H]=0$ and $\rho\D_{x_1}=\D_{x_2}\rho$, we have:
\begin{align*}
\beta\ &=\ 4\ \langle \partial_{x_1} \Phi_{1}, \mathscr{R}(\mu_S) \partial_{x_2} \Phi_{2}\rangle\\
&=\ 4\ \langle  \rho\ \partial_{x_1} \Phi_{1}, \rho \mathscr{R}(\mu_S) \partial_{x_2} \Phi_{2}\rangle\\
&=\ 4\ \langle  \partial_{x_2} \rho\ \Phi_{1},  \mathscr{R}(\mu_S) \partial_{x_1} \rho\ \Phi_{2}\rangle\\
&=\ 4\ \langle  \partial_{x_2} \Phi_{2},  \mathscr{R}(\mu_S) \partial_{x_1}  \Phi_{1}\rangle\\
&=\ 4\ \langle  \mathscr{R}(\mu_S) \partial_{x_2} \Phi_{2},   \partial_{x_1}  \Phi_{1}\rangle\\
&=\ 4\ \overline{\langle  \partial_{x_1}  \Phi_{1}, \mathscr{R}(\mu_S) \partial_{x_2} \Phi_{2}\rangle} = \overline{\beta}.
\end{align*}

To prove \eqref{g-Im}, let $\kappa\in\R^2$ be arbitrary. Using that $\mathcal{R}[\Phi_1]=i  \Phi_1$
 and $\mathcal{R}[\Phi_2]=-i  \Phi_2$, we have for  $j_1, j_2 \in\{1,2\}$:
\begin{align*}
\kappa^{T} A^{j_{1},j_{2}} \kappa &= \langle\partial_{y_l} \Phi_{j_{1}} , \mathscr{R}(\mu_\star) \partial_{y_m} \Phi_{j_{2}}\rangle_{_{L^2(\Omega_\by)}} \kappa_{l}\kappa_{m}\nn\\   
&=  \langle \mathcal{R}(\rho[\partial_{y_l} \Phi_{j_{1}}]),  \ \mathscr{R}(\mu_\star)\ \mathcal{R}(\rho [\partial_{y_m} \Phi_{j_{2}}])\rangle_{_{L^2(\Omega_\bx)}}\  \kappa_{l} \kappa_{m} \\
&= \langle  R_{ns} \rho_{ln}  \partial_{x_s} \mathcal{R}[\rho\Phi_{j_{1}}] ,  \mathscr{R}(\mu_\star) R_{qt} \rho_{mq}  \partial_{x_t}\mathcal{R}[\rho\Phi_{j_{2}}]\rangle_{_{L^2(\Omega_\bx)}}\  \kappa_{l} \kappa_{m}   \\
&=    \langle  \partial_{x_s} i^{2j_{2}-1}\Phi_{j_{2}}, \mathscr{R}(\mu_\star) \partial_{x_q} i^{2j_{1}-1} \Phi_{j_{1}}]\rangle_{_{L^2(\Omega_\bx)}}\ R_{ns}(\rho_{ln} \kappa_l) \ R_{qt} (\rho_{mq} \kappa_{m}) \\
&= i^{2(j_{1}-j_{2})}  \left\langle  \partial_{x_s}\Phi_{j_{2}}, \mathscr{R}(\mu_\star) \partial_{x_t} \Phi_{j_{1}}\right\rangle_{_{L^2(\Omega_\bx)}}\ R_{ns} (\rho \kappa)_{n}   R_{qt} (\rho \kappa)_{q}\ \\
&=  i^{2(j_{1}-j_{2})}   ( R \rho \kappa)^T  A^{j_{2},j_{1}} R \rho \kappa,
\end{align*}
for any choice of pairs $(j_1, j_2)$ with $j_1, j_2 \in \{1,2\}$. Since $\kappa$ is arbitrary, 
\[  A^{j_{1},j_{2}}  = i^{2(j_{1}-j_{2})}  \rho \ R^T A^{j_{2},j_{1}} R \ \rho.\]
For any pair $j_1, j_2 \in \{1,2\}$, let $A^{j_1,j_2}=A = \begin{pmatrix} a & b \\ c & d \end{pmatrix}$. Then, 
\begin{align}
 \ R^T A R \ &= \begin{pmatrix}
    d & -c  \\
    -b  & a
\end{pmatrix}. \end{align}
Consider $j_1 = 1$ and $j_2 = 2$. From the above analysis, 
\[ A\ =\ A^{1,2}\  =\ -\rho \ R\ A^{2,1}\ R^T \ \rho\ =\  -\rho \ R\ (A^{1,2})^\dagger\ R^T \ \rho\ =\   -\rho \ R\ A^\dagger\ R^T \ \rho
\]
so
\begin{align*}
-\rho \ R\ {A}^\dagger\ R^T \ \rho &= -\begin{pmatrix}
    0 & 1  \\
    1  & 0
\end{pmatrix}
\begin{pmatrix}
    \bar{d} & -\bar{b}  \\
    -\bar{c}  & \bar{a}
\end{pmatrix} \begin{pmatrix}
    0 & 1  \\
    1  & 0
\end{pmatrix}\\
&= -\begin{pmatrix}
    -\bar{c}  & \bar{a} \\
    \bar{d} & -\bar{b}
\end{pmatrix}
\begin{pmatrix}
    0 & 1  \\
    1  & 0
\end{pmatrix}\\
&= -\begin{pmatrix}
    \bar{a} & -\bar{c}  \\
    -\bar{b} & \bar{d}
\end{pmatrix}
= \begin{pmatrix}
    -\bar{a} & \bar{c}  \\
    \bar{b} & -\bar{d}
\end{pmatrix}
=\ \begin{pmatrix} a & b \\ c & d \end{pmatrix}
\end{align*}
In particular, $a=-\bar{a}$ and $d=-\bar{d}$. That is,  
 $a^{1,2}_{1,1} = -\bar{a}^{1,2}_{1,1},$ and $a^{1,2}_{2,2} = -\bar{a}^{1,2}_{2,2}$, so that 
\[
\Re\gamma\ =\ 4\Re(a^{1,2}_{1,1}) = 4\Re(a^{1,2}_{2,2})=0. 
\]
\end{proof} 
   
 Corollary \ref{rho-inv} is now an immediate consequence of Part 2 of Proposition \ref{near-M-sym}. 
 \footnote{{\bf miw:} still need to fix  Corollary \ref{rho-inv}. See previous footnote.}
 }

\section{Spectral band degeneracies for small amplitude potentials}\label{sec:eps-small}
In this section we apply Theorem \ref{quad-disp} to the case of small amplitude potentials. Our analysis follows that 
of Section 6 in \cite{FW:12} in the case of honeycomb potentials.  We consider the Floquet-Bloch eigenvalue problem:
\begin{equation}\label{eps-evp}
H^{\varepsilon} \Phi(\bx)\ \equiv\ \left( - \Delta + \varepsilon V(\bx) \right) \Phi(\bx)\ =\ \mu \Phi(\bx),\ \ \Phi\in L^2_{\bM,\sigma}\ ,
\end{equation}
where  $\sigma \in \{+1,-1,+i,-i\}$,  and  $\eps$ is small and non-zero.

By Proposition \ref{phi_sig}, we may seek an $L^2_{\bM,\sigma}$  eigenstate of the form:
\[
  \Phi(\bx)  = \sum_{\bfm \in \mathcal{S}}c_{\Phi}(\bfm) \left(\sigma^{4} e^{i\bM^{\bfm} \cdot \bx} +\sigma^{3} e^{iR\bM^{\bfm} \cdot \bx} + \sigma^{2} e^{iR^{2}\bM^{\bfm} \cdot \bx} + \sigma e^{iR^{3}\bM^{\bfm} \cdot \bx} \right)\ .
  \]
  We use the notations $c(\bfm)=c_\Phi(\bfm)=c(\bfm;\Phi)$.
Applying $\left( - \Delta -\mu \right)$ to $\Phi$ and using that $R$ is an orthogonal matrix yields
\[
\left( - \Delta -\mu \right)\Phi = \sum_{\bfm \in \mathcal{S}} \left( \abs{ \bM^{\bfm}}^{2} - \mu \right) c(\bfm; \Phi) \left(\sum_{i=0}^{3}\sigma^{4-i} e^{iR^{i}\bM^{\bfm}\cdot \bx} \right)
\]
Since $V(R^*\bx)=V(\bx)$, we have that  $V(\bx) \Phi(\bx) \in L^{2}_{\bM, \sigma}$. 
Therefore, by Proposition \ref{phi_sig}, $V\Phi$ has an expansion
\begin{equation}\label{vphi-expansion}
V(\bx)\Phi(\bx) = \sum_{\bfm \in \mathcal{S}} \left[ \sum_{j=0}^{4}\sigma^{4-j}e^{iR^j\bM^\bfm\cdot \bx}\ \right]\ c(\bfm; V\Phi) ,\\
\end{equation}
where
\begin{equation}\label{vphi-exp-co}
c(\bfm; V\Phi)  = \frac{1}{\mid \Omega \mid} \int_{\Omega} e^{-i\bM^{\bfm}\cdot \by}V(\by)\Phi(\by) d\by.
\end{equation}
Recall that $\bq \vec{\bk} \cdot \bx = (q_{1}\bk_{1} + q_{2} \bk_{2} )\cdot \bx$ and $R\bM^\bfr=\bM^{\mathcal{R}\bfr}$;
see \eqref{RbM}.  Thus,
\begin{align*}
c(\bfm; V\Phi)  &= \frac{1}{\mid \Omega \mid} \int_{\Omega} e^{-i\bM^{\bfm}\cdot \by}V(\by)\Phi(\by) d\by\\
&= \frac{1}{\mid \Omega \mid} \int_{\Omega} e^{-i\bM^{\bfm}\cdot \by} \left( \sum_{\bq \in \mathbb{Z}^{2}}V_{\bq}e^{i\bq\vec{\bk}\cdot \by} \right)  \left( \sum_{\bfr \in \mathcal{S}}c(\bfr; \Phi) \left[\sum_{j=0}^{3}\sigma^{4-j}e^{iR^{j}\bM^{\bfr}\cdot \by} \right]\right) d\by\\
&= \frac{1}{\mid \Omega \mid} \sum_{\bq \in \mathbb{Z}^{2}}\sum_{\bfr \in \mathcal{S}}
\int_{\Omega}   V_{\bq}  c(\bfr; \Phi) \left[\sum_{j=0}^{3}\sigma^{4-j}e^{i(\bM^{\mathcal{R}^j\bfr}-\bM^{\bfm} + \bq \vec{\bk})\cdot \by} \right] d\by\\
&= \frac{1}{\mid \Omega \mid} \sum_{\bq \in \mathbb{Z}^{2},  \bfr \in \mathcal{S}}V_{\bq} c(\bfr; \Phi)\int_{\Omega} \sum_{j=0}^{3}\left( \sigma^{4-j}e^{i(\bq -(\bfm-\mathcal{R}^{j}\bfr))\vec{\bk}\cdot \by} \right) d\by\\
&= \sum_{\bq \in \mathbb{Z}^{2},  \bfr \in \mathcal{S}}V_{\bq} c(\bfr; \Phi)\sum_{j=0}^{3}\left[ \sigma^{4-j}\ \delta\left(\bq -(\bfm-\mathcal{R}^{j}\bfr)\right) \right] \ =\  \sum_{  \bfr \in \mathcal{S}} \ \left[ \sum_{j=0}^{3}\ \sigma^{4-j}\  V_{\bfm-\mathcal{R}^{j}\bfr}\ \right]  c(\bfr; \Phi)\\
& = \sum_{\bfr \in \mathcal{S}}\ \left(\  V_{\bfm-\bfr}\ +\ \sigma^3V_{\bfm-\mathcal{R}\bfr}\ +\ 
 \sigma^2 V_{\bfm-\mathcal{R}^2\bfr}+ \sigma V_{\bfm-\mathcal{R}^3\bfr}\ \right)\ c(\bfr; \Phi)
\end{align*}

Thus,
\begin{align}
c(\bfm; V\Phi) &= \sum_{\bfr \in \mathcal{S}} \mathcal{K}_{\sigma}(\bfm, \bfr) c(\bfr; \Phi). \label{cVPhi}\\
  K_{\sigma}(\bfm, \bfr) &= \sum_{j=0}^{3}\sigma^{4-j} V_{\bfm-\mathcal{R}^{j}\bfr} \label{cK-def}\\
  &=\   V_{\bfm-\bfr}\ +\ \sigma^3\ V_{\bfm-\mathcal{R}\bfr}\ +\ \sigma^2\ V_{\bfm-\mathcal{R}^2\bfr}\ +\ 
   \sigma\ V_{\bfm-\mathcal{R}^3\bfr} \nn\\
   &=\  V_{m_1-r_1,m_2-r_2}\ +\ \sigma^3\ V_{m_1-r_2,m_2+r_1+1}\nn\\
  & \qquad\qquad +\ \sigma^2\ V_{m_1+r_1+1,m_2+r_2+1}\ +\ 
   \sigma\ V_{m_1+r_2+1,m_2-r_1}.\nn
\end{align}

We have the  following analogue of Proposition 6.1 of \cite{FW:12}:
\begin{prop}\label{equiv-evp}
Let $\sigma\in\{+i,-i,+1,-1\}$. The $L^2_{\bM,\sigma}-$spectral problem \eqref{eps-evp}  is equivalent to the following algebraic eigenvalue problem for $\{c(\bfm)\}_{\bfm\in\mathcal{S}}$ and $\mu$: 
\begin{equation}
\left(\ |\bM^\bfm|^2-\mu\ \right)\ c(\bfm)\ +\ \eps\ \sum_{\bfr\in\mathcal{S}}\ \mathcal{K}_\sigma(\bfm,\bfr)\ c(\bfr)\ =\ 0,
\qquad \bfm\in\mathcal{S}.
\label{cm-sig}
\end{equation}
\end{prop}
\nit For each $\sigma$, we next solve \eqref{cm-sig} for $\eps\mapsto \{c^{\eps}(\bfm)\}_{\bfm\in\mathcal{S}},\ \mu^{\eps}$, 
perturbatively in $\eps$.  

We first set $\eps=0$ in \eqref{cm-sig}. Then, \eqref{cm-sig} reduces to eigenvalue problem:
\begin{equation}
\left(\ |\bM+m_1\bk_1+m_2\bk_2|^2-\mu^{(0)}\ \right)\ c^0(\bfm)=0,\qquad \bfm=(m_1,m_2)\in\mathcal{S},
\label{0th-ord}
\end{equation} 
which has a simple eigenpair:
\begin{equation}
  \mu^{(0)}=\vert \bM \vert^{2}=2\pi^2=\mu^{(0)}_S,\qquad  c^{(0)}(\bfm) = \delta_{m_{1}+1, m_{2}},
  \label{0-order-sol}
  \end{equation}
  
In arriving at \eqref{0-order-sol}, recall from Theorem \ref{H0-mult4} that  $\mu_S^{(0)} = \abs{\bM}^{2}=2\pi^2$ is an eigenvalue of multiplicity four with four dimensional eigenspace spanned by the eigenvectors: $c(m_1,m_2)= \delta_{m_{1},m_{2}}$,\ $\delta_{m_{1},m_{2}+1}$,\ $\delta_{m_{1}+1,m_{2}+1}$ and $\delta_{m_{1}+1,m_{2}}$, corresponding to those $(m_1,m_2)$ such that $\bM^\bfm=\bM+m_1\bk_1+m_2\bk_2$ is a vertex of $\brill$ and to the orbit under $\mathcal{R}:\ \{(0,0),(0,-1),(-1,-1),(-1,0)\}\in\Z^2/\sim$\ . Recall from Remark \ref{e-class} that $(m_1,m_2)=(-1,0)$ is the representative from this equivalence class,
and hence the choice of  eigenstate in \eqref{0-order-sol}.

 Note also that the solution \eqref{0-order-sol} of the $\eps=0$ algebraic eigenvalue problem corresponds the simple  $L^{2}_{\bM, \sigma}-$eigenpair
   of $H^{(\eps=0)}$, $\mu^{(0)}=\vert \bM^{(0,-1)} \vert^{2}=\vert \bM \vert^{2}$, with corresponding eigenstate
\begin{align}
\Phi^{(\eps = 0)}(\bx) &= \sum_{j=0}^{3} \sigma^{4-j}e^{i(R^{j}\bM^{0,-1})\cdot \bx}
= \sigma e^{i\bM \cdot \bx} \left(1 +  \sigma e^{-i\bk_{1} \cdot \bx}+ \sigma^{2} e^{-i(\bk_{1} + \bk_{2}) \cdot \bx} + \sigma^{3} e^{-i \bk_{2} \cdot \bx} \right).
\end{align}

We next proceed to solve the system \eqref{cm-sig} for smooth curve of eigenpairs: $\eps\mapsto\ \mu^{\eps},\ \{c^{\eps}(\bfm)\}_{\bfm\in\mathcal{S}}$ for all $\eps$ sufficiently small. We write  $\{c(\bfm)\}_{\bfm\in\mathcal{S}}\in l^2(\mathcal{S})=
 (c_\parallel, c_\perp)\in \C\times l^{2}( \mathcal{S}^{\perp}),$ where 
\[ 
  c_{\parallel} \equiv c(0,-1),\ \text { and } \ c_\perp=\{c_{\perp}(\bfm)\}_{\bfm\in\mathcal{S}^\perp},\ \   \mathcal{S}^\perp\equiv\mathcal{S}\setminus\{(-1,0)\}. 
\]

Then \eqref{cm-sig} is equivalent to the following
 coupled system  for $c_{\parallel}$ and $c_\perp=\{c_{\perp}(\bfm)\}_{\bfm\in\mathcal{S}^\perp}$: 
\begin{align}\label{lsr-coupled1}
&\left[ \vert \bM \vert^{2} - \mu + \eps \mathcal{K}_{\sigma}(0, -1, 0, -1) \right] c_{\parallel} + \eps \sum_{\bfr \in \mathcal{S}^{\perp}} \mathcal{K}_{\sigma}(0, -1, \bfr) c_{\perp}(\bfr) = 0.\\
& \eps \mathcal{K}_{\sigma}(\bfm, 0, -1)c_{\parallel} + \left( \vert \bM^{\bfm} \vert^{2} - \mu \right) c_{\perp}(\bfm) + \eps  \sum_{\bfr \in \mathcal{S}^{\perp}} \mathcal{K}_{\sigma}(\bfm, \bfr) c_{\perp}(\bfr) = 0, \ \ \ \bfm \in \mathcal{S}^{\perp}.
\label{lsr-coupled2}\end{align}
For $\eps$ small, we shall solve \eqref{lsr-coupled1}-\eqref{lsr-coupled2}  in a neighborhood of the solution to the $\eps=0$  problem:  
$c_\parallel^{(0)}=1$, $\mu_S^{(0)}=|\bM|^2$, $c^{(0)}_\perp(\bfr)=0,\ \bfr\in\mathcal{S}^{\perp}$. We proceed by a Lyapunov-Schmidt reduction strategy in which we first solve \eqref{lsr-coupled2} for the mapping $c_\parallel\mapsto c_\perp[c_\parallel,\mu]$
 in a neighborhood of $\mu=\mu_S^{(0)}=|\bM|^2$, 
and then substitute this mapping into \eqref{lsr-coupled1} to obtain a closed equation for $c_\parallel$.
We provided a sketch of the argument. For the details, see  \cite{FW:12}.

Equation \eqref{lsr-coupled2} for $c_\perp$ may be expressed in the form:
\begin{equation}\label{c-perp-compact}
(I + \varepsilon \mathcal{T}_{\mathcal{K}_{\sigma}}(\mu) ) c_{\perp} = \eps \ c_{\parallel} \ F_{\sigma}(\mu) ,\end{equation}
with the definitions: $F_{\sigma}(\mu)=\{F_{\sigma}(\bfm, \mu)\}_{\bfm\in\mathcal{S}^\perp}$,\ $F_{\sigma}(\bfm, \mu)\ \equiv\ \frac{ -\mathcal{K}_{\sigma}(\bfm, 0, -1) }{\mid \bM^{\bfm} \mid^{2} - \mu}$,  and 
\begin{equation}\label{c-perp}
\left[(I + \varepsilon \mathcal{T}_{\mathcal{K}_{\sigma}}(\mu) ) c_{\perp}\right](\bfm)\
 \equiv\ \left[ \delta_{\bfm, \bfr} + \frac{\varepsilon}{\mid \bM^{\bfm} \mid^{2} - \mu} \sum_{\bfr \in \mathcal{S}^{\perp}} \mathcal{K}_{\sigma}(\bfm, \bfr) \right] c_{\perp}(\bfr)\ ,\ \bfm\in\mathcal{S}^\perp. 
\end{equation}
For all $\mu$ in a fixed neighborhood of $\mu^{(0)}=|\bM|^2$, we have for all $\bfm\in\mathcal{S}^\perp$ that $\Big| |\bM^\bfm|^2-\mu \Big|\ge\theta>0$ independent of $\eps$. It follows that for all $|\eps|<\eps^0$ sufficiently small, 
$(I + \varepsilon \mathcal{T}_{\mathcal{K}_{\sigma}}(\mu) )^{-1}$ is well-defined as a bounded operator on $l^2(\mathcal{S}^\perp)$. 
Solving  \eqref{c-perp-compact} for $c^{\eps}_{\perp}[c_\parallel,\mu]$ and substituting into \eqref{lsr-coupled1} we conclude:
\begin{prop}\label{eig-cond}
For all $|\eps|<\eps^0$, $\mu$ is an $L^2_{\bM,\sigma}$ eigenvalue if and only if $\mathcal{M}_{\sigma}(\mu, \varepsilon)=0$,
where 
\begin{equation}\label{bif-eqn}
\mathcal{M}_{\sigma}(\mu, \varepsilon) \equiv \mid \bM \mid ^{2} - \mu + \varepsilon \ \mathcal{K}_{\sigma}(0, -1, 0, -1) + \varepsilon^{2}\sum_{\bfr \in \mathcal{S}^{\perp}}  \mathcal{K}_{\sigma}(0, -1, \bfr) [ (I+  \varepsilon \ \mathcal{T}_{\mathcal{K}_{\sigma}}(\mu))^{-1}F_{\sigma}(\bfm, \mu)](\bfr)
\end{equation}
is analytic in a neighborhood about $(\varepsilon,\mu) = (0,\mu_S^{(0)})=(0,|\bM|^2)$.
\end{prop}
Since $\mathcal{M}_{\sigma}(\mu_S^{(0)},0)=0$ and $\D_\mu\mathcal{M}_{\sigma}(\mu_S^{(0)},0)=-1\ne0$, the implicit function theorem, implies that  there is a function $\varepsilon \rightarrow \mu^{\eps}$, defined and analytic for $|\eps|<\eps^1\le\eps^0$,  such that
$ \mathcal{M}_{\sigma}(\mu_S^{\eps}, \varepsilon)  = 0$.
Thus taking $c^{\eps}_{\parallel} \equiv 1$, the solution of the coupled system \eqref{lsr-coupled1}-\eqref{lsr-coupled2}, for  $\vert \eps \vert < \eps^{1}$,  is
\begin{align}
 \mu_S^{\eps} &= \mid \bM \mid^{2} + \varepsilon \mathcal{K}_{\sigma}(0, -1, 0, -1) + \mathcal{O}(\varepsilon^{2}) 
  \label{sig-eig}\\
c_{\parallel}^\varepsilon &= c(0,-1) \equiv 1\nn\\
c_{\perp}^\varepsilon &= \{c^\varepsilon(\bfm)\}_{\bfm \in \mathcal{S}^{\perp}} = \varepsilon (I+  \varepsilon \mathcal{T}_{\mathcal{K}_{\sigma}}(\mu))^{-1}F_{\sigma}(\bfm, \mu^{\eps}),  \; \bfm \in \mathcal{S}^{\perp},\label{cperp-eps}\\
& \text{where }F_{\sigma}(\bfm, \mu^{\eps}) = \frac{ -\mathcal{K}_{\sigma}(\bfm, 0, -1) }{\mid \bM^{\bfm} \mid^{2} - \mu^{\eps}}.
\nn\end{align}

From \eqref{sig-eig} we may now make explicit the splitting of the four-fold degenerate $L^2_\bM$ eigenvalue of 
 $H^{\eps}=-\Delta+\eps V$  for $\eps$ non-zero and small. By \eqref{cK-def} and the relations among the Fourier coefficients
  of $V$ displayed in \eqref{Vcycle1} we have
\begin{equation}
 \mathcal{K}_{\sigma}(0, -1, 0, -1)\ =\ V_{00}\ +\ \left(\ \sigma^3+\sigma\ \right)\ V_{01}\ +\ \sigma^2\ V_{11}
 \label{Ksig0}
 \end{equation}
Therefore, 
\begin{align}
\mathcal{K}_{\pm1}(0, -1, 0, -1) &= V_{0,0} \pm 2 V_{0,1} + V_{1,1} \label{Ksig_pm1}\\
\mathcal{K}_{+i}(0, -1, 0, -1)&= \mathcal{K}_{-i}(0, -1, 0, -1)\ =\  V_{0,0}  - V_{1,1}\label{Ksig-eye} 
\end{align}
This establishes that for typical choices of Fourier coefficients --- specifically 
 $V_{0,0} \pm 2 V_{0,1} + V_{1,1}\ne V_{0,0}  - V_{1,1}$ or equivalently $V_{1,1}\ne V_{0,1}$--- and for  $\eps$ small and non-zero,  the 
multiplicity four $L^2_\bM$ eigenvalue, $\mu_S^{(0)}=|\bM|^2$, splits, at order $\eps$, distinct $L^2_{\bM,+1}$ and $L^2_{\bM,-1}$ eigenvalues
and an $L^2_\bM-$double eigenvalue in the subspace $L^2_{\bM,+i}\oplus L^2_{\bM,-i}$. Note in fact that the latter is an exact (to all orders) double eigenvalue in $L^2_{\bM,+i}\oplus L^2_{\bM,-i}$. Indeed, consider the simple $L^2_{\bM,+i}$ eigenvalue,
whose existence is guaranteed by the above proof for $\sigma=+i$. Applying $\mathcal{P}\circ\mathcal{C}$ to the corresponding eigenfunction we obtain an eigenfunction in the space $L^2_{\bM,+i}$ with the identical eigenvalue. This must coincide with the simple eigenvalue constructed for $\sigma=-i$.
 Summarizing we have:
 

\begin{thm}\label{mu-epsilon-evals}
Consider $H^{\eps}=-\Delta+\eps V$, where $V$ is admissible (Definition \ref{def:sq-pot}) and $0<|\eps|<\eps_1$ is sufficiently small.
Assume  the non-degeneracy condition on distinguished Fourier coefficients:
\begin{equation}
V_{1,1}\ne \pm V_{0,1},\label{nodeF}
\end{equation}
 where $V_{m_1,m_2}$ denotes the $(m_1,m_2)$ Fourier coefficient of $V$.
 Then the $4-$dimensional eigenspace of $H^{0}$ corresponding to the eigenvalue $\mu_S^{(0)}=|\bM|^2$ perturbs to
a $2-$dimensional eigenspace with eigenvalue $\mu_S^{\eps}$ and additionally  two $1-$dimensional eigenspaces with eigenvalues
 $\mu^{\eps}_{(+1)}$ and $\mu^{\eps}_{(-1)}$ as follows:
\begin{enumerate}
\item  $\mu_S^{\eps}$ is of geometric multiplicity 2 eigenvalue, with a $2-$dimensional eigenspace
 $\mathbb{X}_{i}\subset  L_{\bM,i}^{2}$ and $ \mathbb{X}_{-i} \subset L_{\bM,-i}^{2}$, and has the expansion
\begin{equation}\label{muK-eval-expressions-1}
\begin{split}
\mu_S^{\eps}
&=\abs{\bM}^{2} + \eps ( V_{0,0}  - V_{1,1}  ) + \mathcal{O}(\eps^{2}).
\end{split}
\end{equation}
Associated with it are the eigenstates  $\Phi_{1}^{\eps} \in L^{2}_{\bM, i}$ and $\Phi_{2}^{\eps} \in L^{2}_{\bM, -i}$, related by the symmetry:
\[
\Phi_{2}(\bx) = \left(\mathcal{P}\circ\mathcal{C}\right)[\Phi_1](\bx)=\conjugatet{\Phi_{1}(-\bx)},
\]
which we also denote $\Phi_{(+i)}^{\eps}$ and $\Phi_{(-i)}^{\eps}$, respectively. Their
 Fourier expansions are:
\begin{align}
\Phi_{(+i)}^\eps = \Phi_{1}^{\eps}(\bx) \ &= \ \sum_{\bfm\in \mathcal{S}} c^{\eps}_{(+i)}(\bfm) \left( e^{i\bM^{\bfm}\cdot \bx} - i e^{iR\bM^{\bfm}\cdot \bx} - e^{iR^{2}\bM^{\bfm}\cdot \bx} + i e^{iR^{3}\bM^{\bfm}\cdot \bx} \right) , \text{ and }\\
\Phi_{(-i)}^\eps =  \Phi_{2}^{\eps}(\bx) \ &= \ \sum_{\bfm\in \mathcal{S}} \conjugatet{c^{\eps}_{(+i)}(\bfm)} \left( e^{i\bM^{\bfm}\cdot \bx} + i e^{iR\bM^{\bfm}\cdot \bx} - e^{iR^{2}\bM^{\bfm}\cdot \bx} - i e^{iR^{3}\bM^{\bfm}\cdot \bx} \right) .
\end{align}
\item 
The distinct eigenvalues $\mu_{(+1)}^{\eps}$ and  $\mu_{(-1)}^{\eps}$ is each of  geometric multiplicity $1$, with corresponding $1-$dimensional eigenspaces $\mathbb{X}_{\pm1}  \subset L_{\bM,\pm1}^{2}$, and they are given by :
\begin{equation}\label{muK-eval-expressions-2}
\begin{split}
\mu_{(\pm1)}^{\eps}
&= \abs{\bM}^{2} + \eps ( V_{0,0} \pm 2 V_{0,1} + V_{1,1} ) + \mathcal{O}(\eps^{2}),
\end{split}
\end{equation}
with associated eigenstates $\Phi_{(\pm 1)}^{\eps}$:
\begin{equation}
\Phi_{(\pm1)}^{\eps}(\bx) \ = \ \sum_{\bfm\in \mathcal{S}} c_{(\pm 1)}^{\eps}(\bfm) \left( e^{i\bM^{\bfm}\cdot \bx} + e^{iR\bM^{\bfm}\cdot \bx} + e^{iR^{2}\bM^{\bfm}\cdot \bx} +  e^{iR^{3}\bM^{\bfm}\cdot \bx} \right) .
\end{equation}
\end{enumerate}
\end{thm}

Theorems \ref{mu-epsilon-evals} and  \ref{quad-disp} imply 
\begin{cor}\label{small-eps-disp}  Consider the setup of Theorem \ref{mu-epsilon-evals}. 
There exists $\eps_1>0$ and $\kappa_1(\eps)>0$, which tends to zero as $\eps\to0$, such that the following holds. 
Fix $\eps\in(-\eps_1,\eps_1)\setminus\{0\}$. Then,
\begin{enumerate}
\item  for all $|\kappa|=\sqrt{\kappa_1^2+\kappa_2^2}<\kappa_1(\eps)$, the two dispersion surfaces which touch at  $\bM$ (and therefore the vertices of $\brill$)  are locally described by:
\begin{equation}
\mu_{\pm}^\eps(\bM + \kappa) - \mu_S^\eps =
(1-\alpha^\eps)|\kappa|^2+\mathscr{Q}^\eps_6(\kappa)\ \pm \sqrt{\Big|\ \gamma^\eps(\kappa_{1}^{2} - \kappa_{2}^{2})+ 2\beta^\eps \kappa_{1}\kappa_{2}\ \Big|^2\ +\ \mathscr{Q}^\eps_8(\kappa)}
\label{mu_eps-diff}\end{equation}

If  $|V_{11}| \ne  |V_{01}|$, then coefficients $\alpha^\eps$, $\beta^\eps$ and $\gamma^\eps$ are expressions which depend on 
$\{\Phi_{(+i)}^\eps,\Phi_{(-i)}^\eps\}$ and have the following expansions for $\eps$ positive and small:
\begin{align}
\alpha^\eps &=  \frac{32\pi^2}{\eps} \left(\frac{V_{11}}{V_{11}^2 - V_{01}^2}\right) + \mathcal{O}(1);\label{alph}\\
\beta^\eps &= \frac{32\pi^2}{\eps}\left( \frac{V_{11}}{V_{11}^2 - V_{01}^2}\right)+ \mathcal{O}(1);\label{beta}\\
\gamma^\eps &= -\frac{32\pi^2}{\eps} \ i  \left(\frac{V_{01}}{V_{11}^2 - V_{01}^2}\right)+ \mathcal{O}(1) . \label{gamma}
\end{align}
The functions $\mathscr{Q}^\eps_6(\kappa)=\mathcal{O}(|\kappa|^6)$ and $\mathscr{Q}^\eps_8(\kappa)=\mathcal{O}(|\kappa|^8)$ are analytic in a complex neighorhood of the origin in $\C^2$;
see Theorem \ref{quad-disp} for more discussion. 
\item Let $V$ be admissible and assume the following (generically satisfied) conditions on Fourier coefficients:
\[ V_{11} \neq \pm V_{01},\quad  V_{11}\ne0\ \ {\rm and}\ V_{01}\ne0.\]
Then,  for all $\eps\in (-\eps_1,\eps_1)\setminus\{0\}$, the coefficients $\alpha^\eps, \beta^\eps $ and $\gamma^\eps $ are all non-zero.
\item In the case where $V$ is admissible and also reflection invariant, by Corollary \ref{rho-inv}, we have $\gamma^\eps=0$ for any $\eps$, and in particular $V_{01}=0$. If $V_{11}\ne0$, then the coefficients 
 for all $\eps\in (-\eps_1,\eps_1)\setminus\{0\}$, the coefficients $\alpha^\eps$ and $ \beta^\eps $ are both non-zero.

\end{enumerate}
\end{cor}

\nit The proof of   Corollary \ref{small-eps-disp} is completed in Appendix \ref{mu-eps-coeff} with the derivation of expansions
\eqref{alph}, \eqref{beta} and \eqref{gamma}.

\section{ $H^\eps=-\Delta+\eps V$ for $V$ admissible and $\eps$ generic}\label{eps-generic}

Theorem \ref{mu-epsilon-evals} considers degeneracies among dispersion surfaces for all $\eps$ non-zero, real and small. In this section we extend this result to generic real values of $\eps$, with no constraint on its size. Thus, generic 
large ({\it high contrast}) potentials are included.

\begin{thm}\label{generic-eps}
Let $V$ denote an admissible potential (Definition \ref{def:sq-pot}). Let $\eps_1$ be as in Theorem  \ref{mu-epsilon-evals} and Corollary \ref{small-eps-disp}. 
 Consider either of the  two scenarios:
\begin{itemize}
\item[(I)] $V$ admissible with $V_{11} \neq \pm V_{01},\quad  V_{11}\ne0\ \ {\rm and}\ V_{01}\ne0$, or
\item[(II)] $V$ admissible and reflection invariant with $V_{11}\ne0$.
\end{itemize}
Then, there exists a discrete set $\widetilde{\mathcal C}\subset\R\setminus(0,\eps_1)$ such that if $\eps\notin\widetilde{\mathcal C}$,
 the conditions of Theorem \ref{quad-disp} are satisfied and two dispersion surfaces touch at the vertices of $\brill$. 
Moreover,  in scenario (I),  $\alpha^\eps, \beta^\eps $ and $\gamma^\eps $ are all non-zero for all $\eps\notin\widetilde{\mathcal C}$, and $\alpha^\eps$ and $ \beta^\eps $ are both non-zero for all $\eps\notin\widetilde{\mathcal C}$.
\end{thm} 

\begin{remark}\label{which?}
For $\eps\in(-\eps_1,\eps_1)\setminus \{0\}$, Theorem  \ref{mu-epsilon-evals} ensures that a pairs of dispersion surfaces, among the first four, touch at band degeneracies for quasi-momenta at the vertices of $\brill$; for the specific scenarios see Section \ref{numerical}. For general $\eps\notin\widetilde{\mathcal C}$ we make no assertions about which of the infinitely many dispersion surfaces touch at high-symmetry quasi-momenta. 
 But in analogy with the honeycomb setting studied in \cite{FLW-CPAM:17}, we expect for the case of a potential which is a superposition of potential wells, that in the strong binding regime there will exist quadratic degeneracies at the intersection of the first two dispersion surfaces. \end{remark}

We discuss the strategy for the proof of Theorem \ref{generic-eps}, but do not present all details. Arguments 
of this type, rooted in complex analysis, were developed in \cite{FW:12} and Appendix D of \cite{FLW-MAMS:17}.
The strategy is based on a continuation argument in the parameter $\eps$, starting with $\eps$ varying in the open interval $(0,\eps_1)$; analogous arguments apply to negative values of $\eps$. Eigenvalues, $\mu$, of the operator  $H^\eps=-\Delta +\eps V$,  in the spaces $L^2_{\bM,\sigma}$, for $\sigma=\pm1,\pm i$,  are realized as zeros of a modified Fredholm determinant $\mathcal{E}_\sigma(\mu,\eps)$. The mapping $(\mu,\eps)\mapsto \mathcal{E}_\sigma(\mu,\eps)$ is analytic. For $\eps$ real, $\mu$ is an $L^2_{\bM,\sigma}$ eigenvalue of geometric multiplicity $m$ if and only if $\mu$ is a zero of $\mathcal{E}_\sigma(\mu,\eps)$ of multiplicity $m$. The strategy of \cite{FW:12}
 (see also Appendix D of \cite{FLW-MAMS:17}) can be used to establish that there is a discrete set 
 $\widetilde{\mathcal C}\subset \R\setminus (0,\eps_1)$ such that for all $\eps\notin\widetilde{\mathcal C}$, there exists 
  $\mu_S^\eps\in\R$ such that 
  \begin{enumerate}
  \item $\mu=\mu^\eps_S\in\R$ is a simple zero of $\mathcal{E}_{+i}(\mu,\eps)$ and of $\mathcal{E}_{-i}(\mu,\eps)$.
  \item $\mathcal{E}_{1}(\mu^\eps_S,\eps)\ne0$ and $\mathcal{E}_{-1}(\mu^\eps_S,\eps)\ne0$.
  \item For scenario $(I)$, $\alpha^\eps, \beta^\eps$ and $\gamma^\eps$ are all non-zero  and $\alpha^\eps$  and for scenario $(II)$,  $\alpha^\eps$ and  $ \beta^\eps $ are non-zero.
\end{enumerate}
Therefore, by Theorem \ref{quad-disp}, for all $\eps\notin\widetilde{\mathcal C}$ there exist quadratic degeneracies at the  quasi-momentum / energy pairs $(\mu_S^\eps,\bM_\star)$, where $\bM_\star$ varies over the four vertices of $\brill$. These are locally described  by 
\eqref{quad-disp} of Theorem \ref{quad-disp}; see also \eqref{mu-rho} of Corollary \ref{rho-inv}.

\section{Computational Experiments}\label{numerical}

In this section, we describe  numerical computations of the spectrum of periodic Schr\"odinger operators with admissible potentials in the sense of Definition \ref{def:sq-pot}. We discuss these numerical results in the context our analytical results. First, in Section \ref{s:numer1}, we consider examples of 
admissible potentials which exhibit quadratic  intersections of dispersion surfaces at $\bM$; see Theorems \ref{quad-disp}, \ref{mu-epsilon-evals} and \ref{generic-eps}.  In order to observe clear numerical separation of the bands, we work here with $\eps$ generically not so small, meaning many of our results fit more into Theorem \ref{generic-eps} but display many of the effects described in the small $\eps$ setting.  

 In Section \ref{s:numer2}, we consider large amplitude  potentials and revisit the question posed in Section \ref{LiebLattice} regarding Lieb lattice potentials. 

\medskip

To numerically approximate the dispersion surfaces for the Schr\"odinger operator, we use the periodic formulation of the self-adjoint eigenvalue problem \eqref{per-evp}. For the numerical experiments of Section \ref{s:numer1} we discretized the fundamental period cell, $\Omega$, and used a finite difference approximation to $H_V(\bk)$ for a fixed $\bk\in\brill$. 
In the numerical experiments of Section \ref{s:numer2}, we discretized the mapping: $f\mapsto\  [-(\partial_x + i k_x)^2 -(\partial_y + i k_y)^2]f$ in Fourier space, and the operator $f\mapsto V f $ in physical space. 
For both approaches, we used the MATLAB function  \verb+eigs+ using the \verb+'sr'+ flag. 
For each fixed $\bk$ varying over a discretization of an appropriate subset of $\brill$, {\it i.e.} the irreducible Brillouin zone or the circuit ${\bf \Gamma} \to {\bf X} \to {\bf M} \to {\bf \Gamma}$ outlined in Figure \ref{fig:squareBZ}, we compute the five smallest eigenvalues.

\subsection{Computations of quadratic degeneracy of dispersion surfaces near $\bM$ and comparison with 
Theorem \ref{mu-epsilon-evals}} \label{s:numer1}

We consider a class of periodic potentials, which are the $\Z^2-$ periodic extension of the function 
\begin{equation} \label{e:PotForm}
V(\bx) = \sum_i  s_i f(|\bx-\bx_i|),  \qquad \qquad \bx \in\Omega=[0,1]^2. 
\end{equation}
Here $f(x) = \frac{1}{2}\left(1+ \cos(\pi x/r)\right) \chi_{\{x<r\}}$ is a compactly supported, $C^1$ function and $r=0.2$. The points $\{\bx_i\}$ are lattice points within the primitive cell, $[0,1]^2$. The binary variables, $s_i \in \{+1,-1\}$, determine the sign of the potential at the lattice points. We choose $\{\bx_i\}$ and $\{s_i\}$ so that $V(\bx)$ is an admissible potential; see Definition \ref{def:sq-pot}.  By varying $\{\bx_i\}$ and $\{s_i\}$, we show that the combinations of Fourier  coefficients: $V_{00}, V_{01}$ and $V_{11}$ ($V_{m,n} =  (2\pi)^{-2}\int_0^1 \int_0^1 \  V(x,y) \ e^{- 2 \pi i (mx + ny) } \ dx dy$), 
appearing in Theorem \ref{mu-epsilon-evals} can be varied in order to exhibit different local behavior of 
 the first four dispersion surfaces near $\bM$. 
We remark that, while Theorem \ref{mu-epsilon-evals} describes the  dispersion surfaces for the operator $H_V = - \Delta + \eps V$ near the point $\bk = \bM$  for sufficiently small and non-zero and real $\eps$, our computations are performed for $\eps = O(1)$. 

In each of Figures \ref{figure11} -- \ref{figure10} we plot, over one period cell, an admissible potential of the form in \eqref{e:PotForm} and the first five dispersion curves for $H_V = -\Delta + \eps V$, and 
  compare the numerically computed results with the assertions of  Theorem \ref{mu-epsilon-evals}. For each potential, the lattice sites, $\{\bx_i\}$, and signs, $\{s_i\}$, are easily inferred from the plot of the potential. They are also summarized in the following table: 
\begin{center}
\begin{tabular}{|c |  |l|  l |l | }
\hline
Figure & $\eps$ & $\bx_i$ & $s_i$ \\
\hline
\ref{figure11} & $2$ & $ [(0,0), \ (1/2,0), \ (0, 1/2)]$ & $(1,-1,-1)$ \\
\ref{figure13} & 2 & [(0,0), \ (1/2,0), \ (0, 1/2)] & $(-1,1,1)$\\
\ref{figure12} & $2$ & $[(1/2,1/4), \ (1/4,1/2), \ (3/4, 1/2), \ (1/2, 3/4)]$ & $(1,1,1,1)$ \\
\ref{figure9} & $4$ & $ [(0,0), \ (1/2,0), \ (0, 1/2)] $ & $(-1, -1, -1)$  \\
\ref{figure10} & $2$ & $ [(1/2,1/2)] $  & $(-1)$ \\\hline
\end{tabular}
\end{center} 
The potentials in Figures \ref{figure11}, \ref{figure13}, and \ref{figure9} are Lieb lattice potentials (see Example \ref{lieb-pot}) while the potential in Figure \ref{figure12} is a square lattice potential (see Example \ref{square-pot}); see Figure \ref{fig:Lattices}. 
The dispersion curves which pass through $\mu^\eps_i, \mu^\eps_{-i}, \mu^\eps_1$ and $ \mu^\eps_{-1}$ are plotted for quasi-momentum $\bK$ along the cyclic path ${\bf \Gamma} \to {\bf X} \to {\bf M} \to {\bf \Gamma}$ in $\brill$; see Figure \ref{fig:squareBZ}. 
 In all dispersion plots, we observe  saddle-like  touching of the dispersion surfaces as described in Theorem \ref{quad-disp}. 

\medskip

We note that the ordering of the dispersion surfaces, specified by certain Fourier coefficients of the potential, in Theorem \ref{mu-epsilon-evals} (which applies to $\eps$ sufficiently small) persists for larger (order 1) values of $\eps$. These dispersion surface orderings are summarized in the following table. 
\begin{center} 
\begin{tabular}{|c| l | c|}\hline
& Fourier Coefficient Condition & Dispersion Surface Ordering \\
\hline
(1a) & $V_{1,1} > 0, V_{0,1} < 0$ and $|V_{0,1}| < V_{1,1}$ & $\mu_{+ i}^{\eps} = \mu_{- i}^{\eps}< \mu_{+1}^{\eps} <  \mu_{-1}^{\eps}$ \\
(1b) & $0 < V_{0,1} < V_{1,1}$ & $\mu_{+ i}^{\eps}= \mu_{- i}^{\eps} < \mu_{-1}^{\eps}  < \mu_{+1}^{\eps}$ \\
\hline
(2a) &  $V_{1,1}, V_{0,1} < 0 $ and $|V_{0,1}| > | V_{1,1} |$ & $ \mu_{+1}^{\eps}< \mu_{+ i}^{\eps} = \mu_{- i}^{\eps} <  \mu_{-1}^{\eps}$ \\
(2b) &  $V_{1,1} < 0$ and $V_{0,1} > |V_{1,1}|$ & $ \mu_{-1}^{\eps} < \mu_{+ i}^{\eps}= \mu_{- i}^{\eps}< \mu_{+1}^{\eps} $ \\
\hline
(3a) & $V_{1,1},V_{0,1} < 0 $ and $  |V_{0,1} | < |V_{1,1} |$ & $ \mu_{+1}^{\eps} < \mu_{-1}^{\eps} <\mu_{+ i}^{\eps} = \mu_{- i}^{\eps}$ \\
(3b) &$ V_{1,1} < 0,  \  V_{0,1}>0 , $ and $| V_{0,1} | < | V_{1,1} | $ &  $  \mu_{-1}^{\eps}< \mu_{+1}^{\eps}  < \mu_{+ i}^{\eps}= \mu_{- i}^{\eps}$\\ \hline
\end{tabular} 
\end{center} 
Here, we've enumerated the various cases so that, {\it e.g.}, the multiplicity occurs in the first eigenvalue for cases (1a) and (1b). For larger values of $\eps$ this ordering may be violated. However, by Theorem \ref{generic-eps}, quadratic degeneracies in the band structure of $-\Delta+\eps V$ will still occur for all but a discrete set of $\eps-$ values.

The potential in Figure \ref{figure11} satisfies condition (1b) and we observe that the first two surfaces intersect at $\bM$, with the others separated and laying above.  
The potential in Figure \ref{figure13} satisfies condition (3a) and we observe that the third and fourth surfaces intersect at $\bM$, with the others separated and laying below.  

In Figure \ref{figure12} the coefficients satisfy (2a), except that $V_{1,1} = 0$. Here, the first and fourth bands are strongly separated, but the second and third bands not only intersect at $\bM$, but, interestingly, remain touching on the $\bM  \to {\bf \Gamma}$ interval. 

In Figure \ref{figure9}, the potential consists of potential wells centered on the Lieb lattice sites and satisfies  condition (2b). Consequently, we observe an intersection at $\bM$ between the second and third surfaces. The coefficients satisfy $V_{0,1} \approx - V_{1,1} $, so the first dispersion surface is quite close to the second and third surfaces at $\bM$, but does not touch. The first three bands are separating from the fourth and we observe the emergence of the tight-binding dispersion relation, as shown in Figure \ref{fig:lieb-tb}. This will  be further investigated in Section \ref{s:numer2}; see Figure \ref{TB1}.

Finally, in Figure \ref{figure10}, we consider a potential that consists of potential wells centered on the square lattice sites and satisfies condition (2a). We see that the second and third band remaining touching with the fourth band close but laying above. The first band is well separated and lays below. Interestingly, the linear prediction is that 
$\mu_{+1}^{\eps} = 2\pi^2  +  0.0005\, \eps + o(\eps)$, but for this value of $\eps$, we have $\mu_{+1}^{\eps} < 2 \pi^2$ so we are already in a higher-order regime. 
This will  be further investigated in Section \ref{s:numer2}; see Figure \ref{TB2}.

\begin{figure}[t!]
\begin{center}
\caption{Eigenvalue ordering, up to $\mathcal{O(\eps)}$:  $ \mu_{+ i}^{\eps}  = \mu_{- i}^{\eps}  < \mu_{-1}^{\eps}  < \mu_{+1}^{\eps} $.}
\label{figure11}
   \subfloat[Potential]{\includegraphics[width=.45\textwidth]{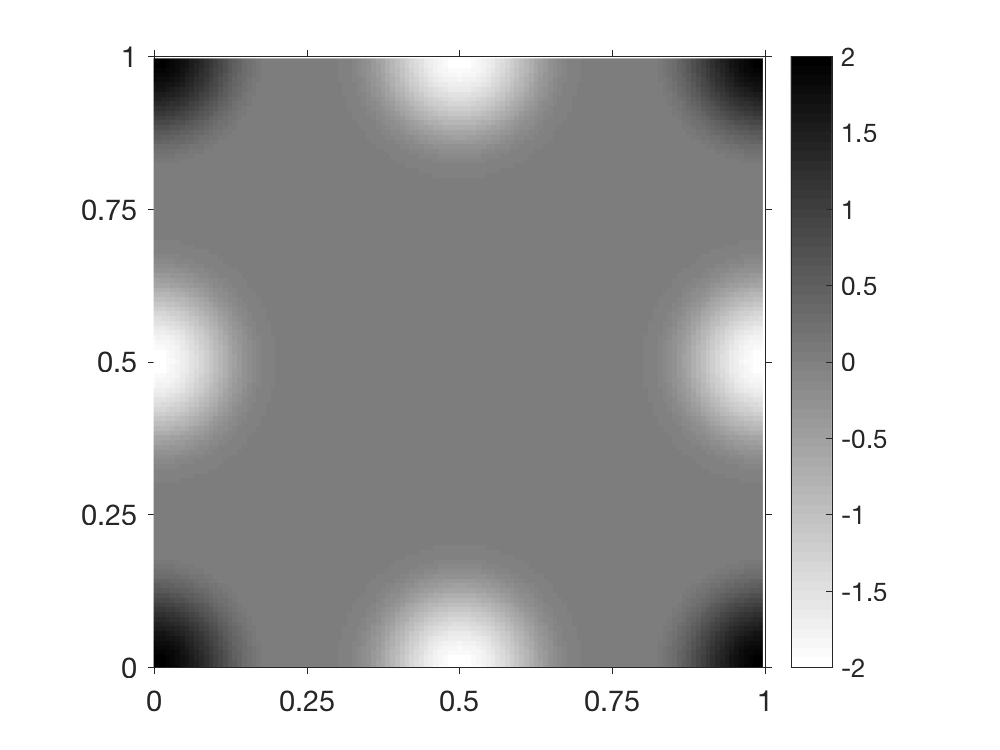}} 
\subfloat[Dispersion curves for $H_V$]{\includegraphics[width=.45\textwidth]{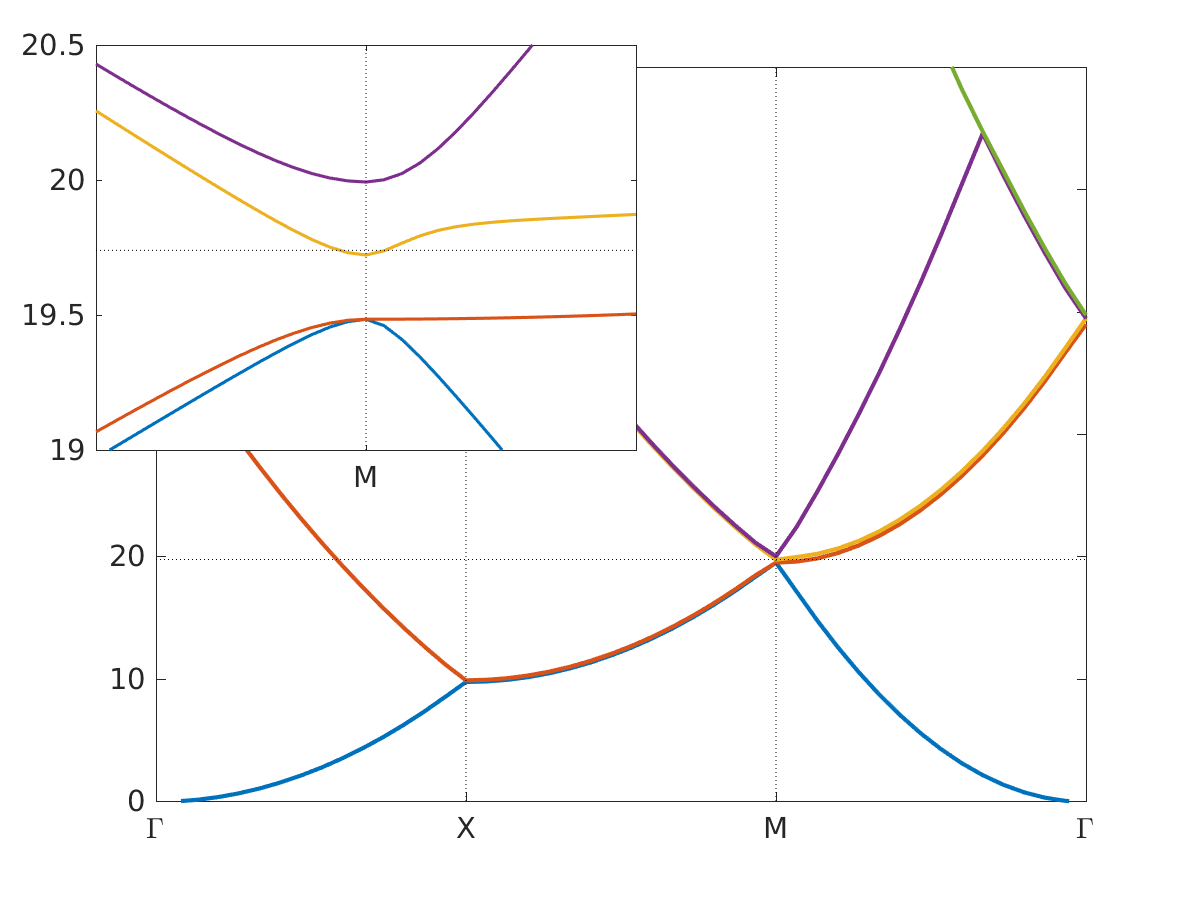}} 
\end{center}
\end{figure}

\begin{figure}[h!]
\begin{center}
\caption{Eigenvalue ordering, up to $\mathcal{O(\eps)}$:  $ \mu_{+1}^{\eps} < \mu_{-1}^{\eps} < \mu_{+ i}^{\eps}  = \mu_{- i}^{\eps}$.}
\label{figure13}
   \subfloat[Potential]{\includegraphics[width=.45\textwidth]{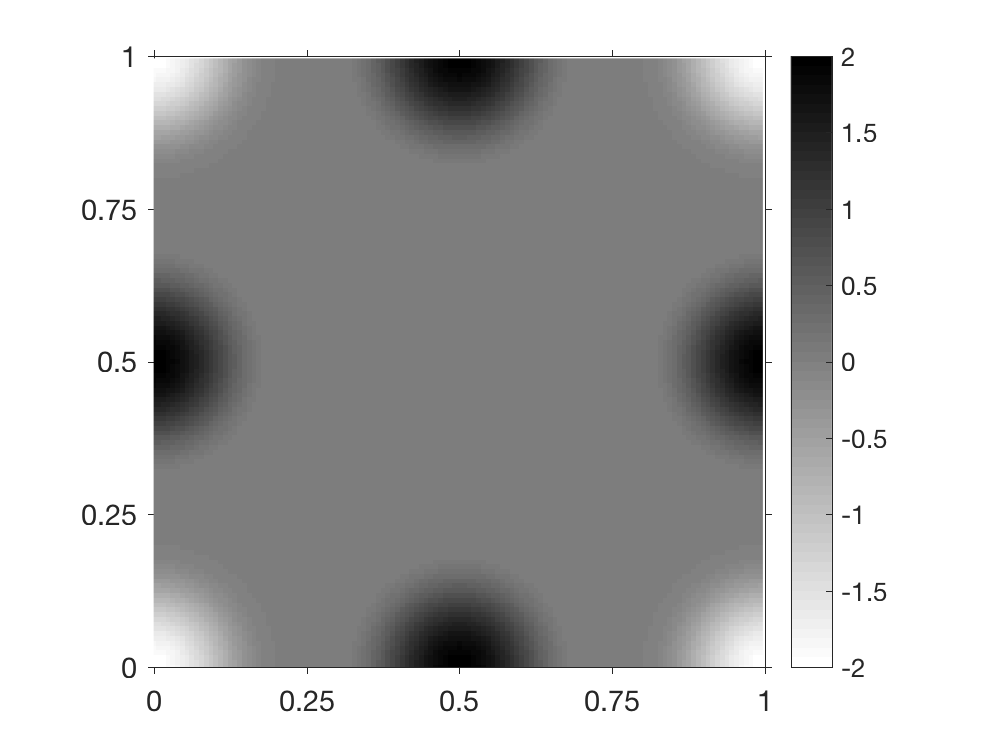}}
\subfloat[Dispersion curves for $H_V$]{\includegraphics[width=.45\textwidth]{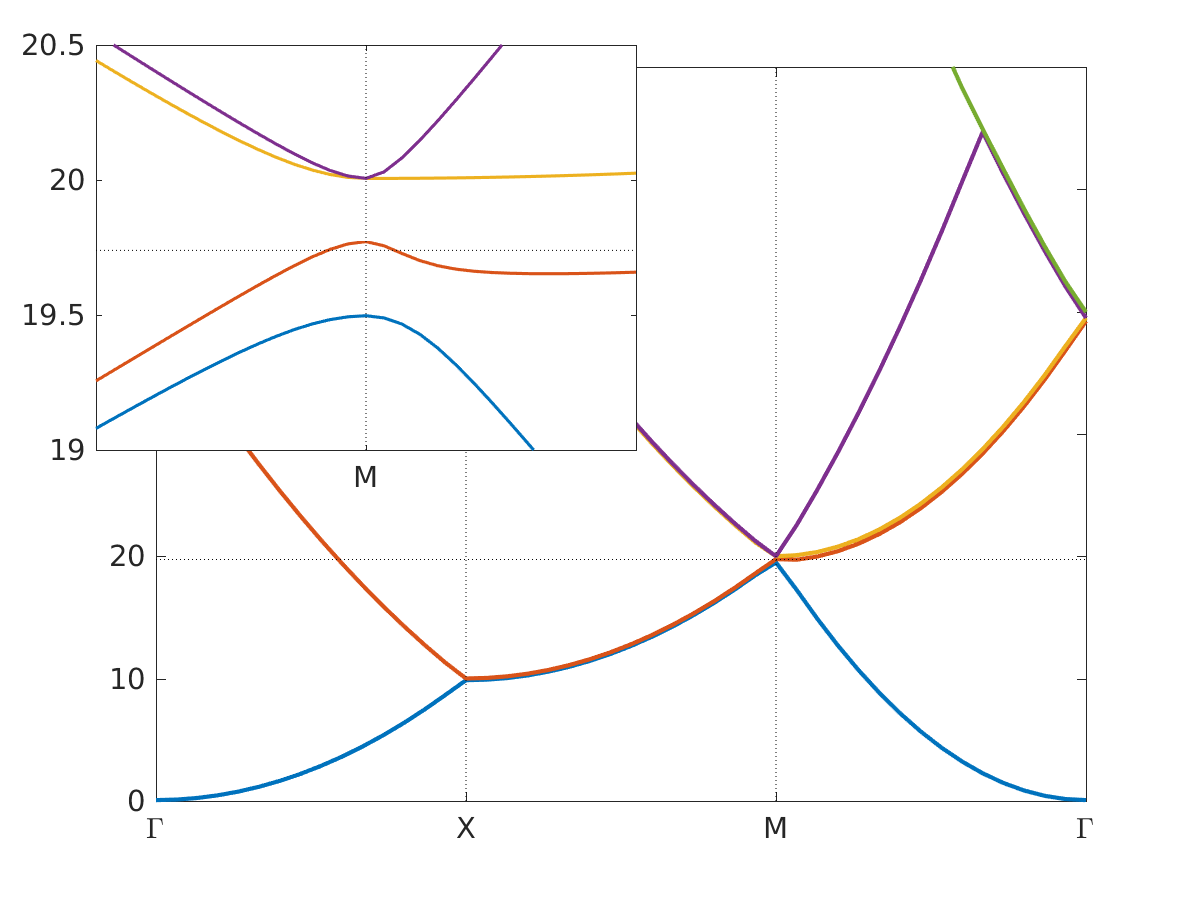}} 
\end{center}
\end{figure}

\begin{figure}[h!]
\begin{center}
\caption{Eigenvalue ordering, up to $\mathcal{O(\eps)}$:  $ \mu_{+1}^{\eps} < \mu_{+ i}^{\eps}  = \mu_{- i}^{\eps}  < \mu_{-1}^{\eps}$.}
\label{figure12}
   \subfloat[Potential]{\includegraphics[width=.45\textwidth]{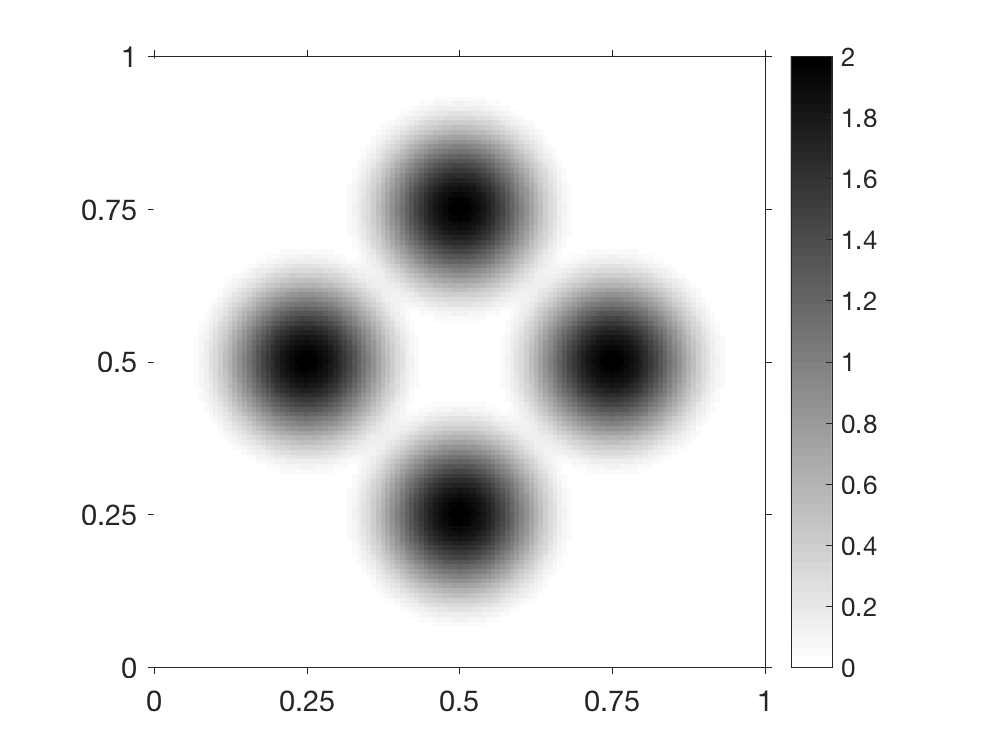}}
\subfloat[Dispersion curves for $H_V$]{\includegraphics[width=.45\textwidth]{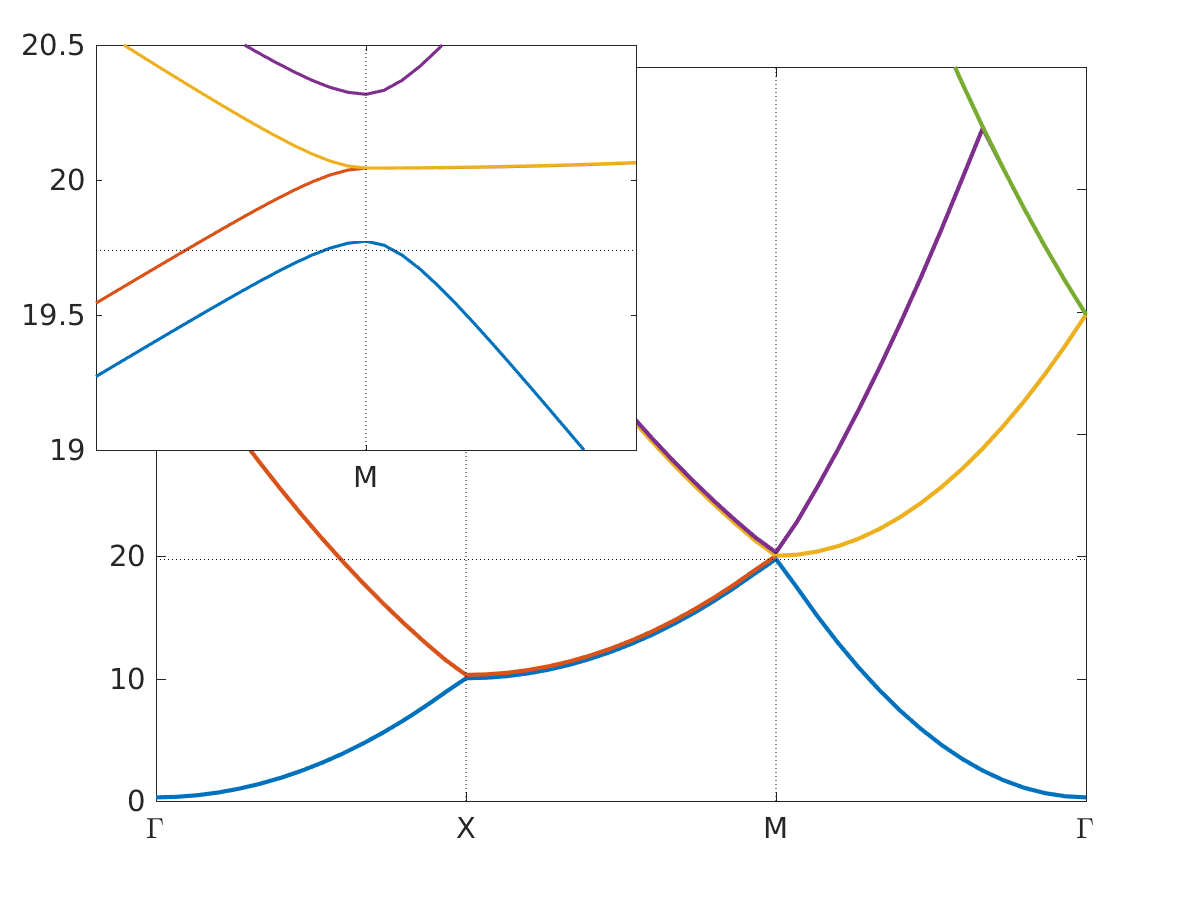}} 
\end{center}
\end{figure}

\begin{figure}[h!]
\caption{Eigenvalue ordering, up to $\mathcal{O(\eps)}$: $\mu_{-1}^{\eps} < \mu_{+ i}^{\eps} = \mu_{- i}^{\eps} < \mu_{1}^{\eps}$. }
\label{figure9}
\begin{center}
   \subfloat[Potential]{\includegraphics[width=.45\textwidth]{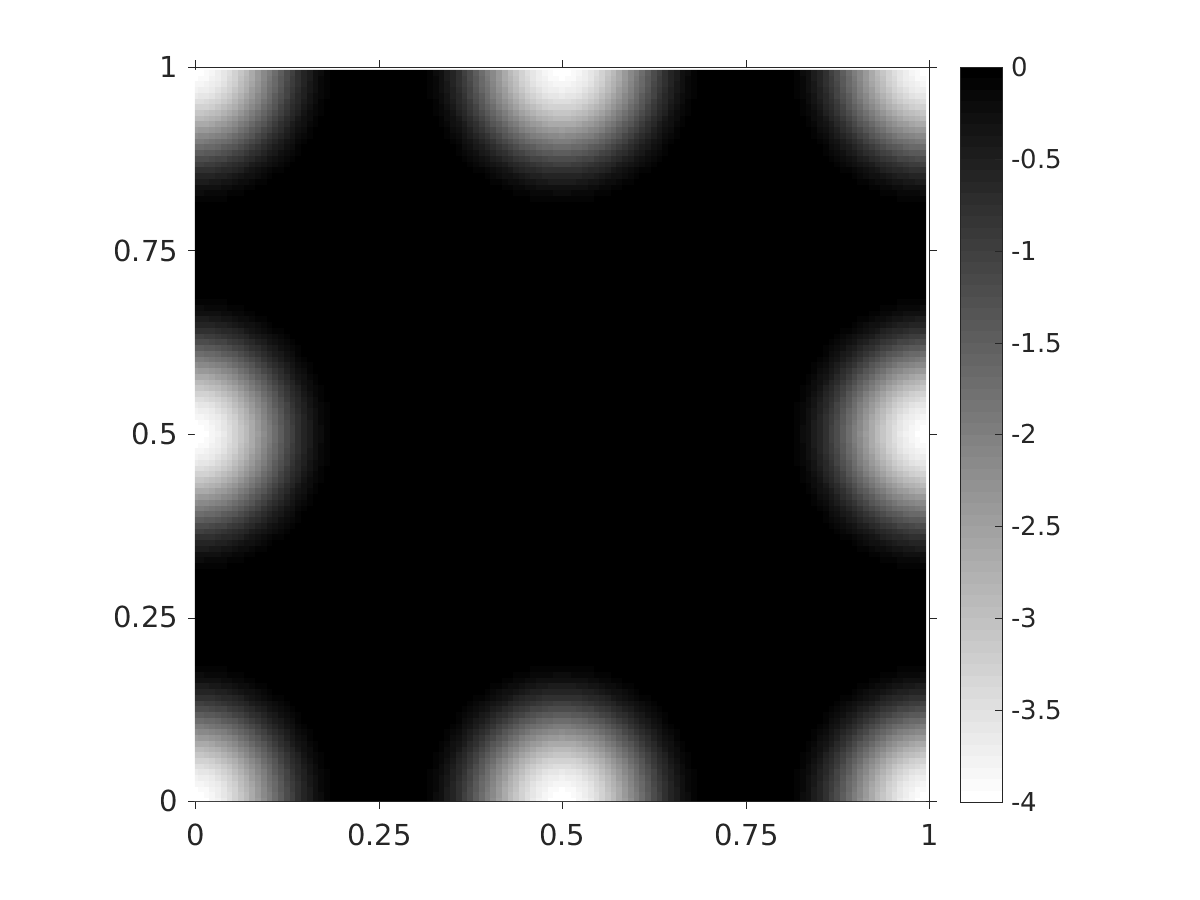}}
\subfloat[Dispersion curves for $H_V$]{\includegraphics[width=.45\textwidth]{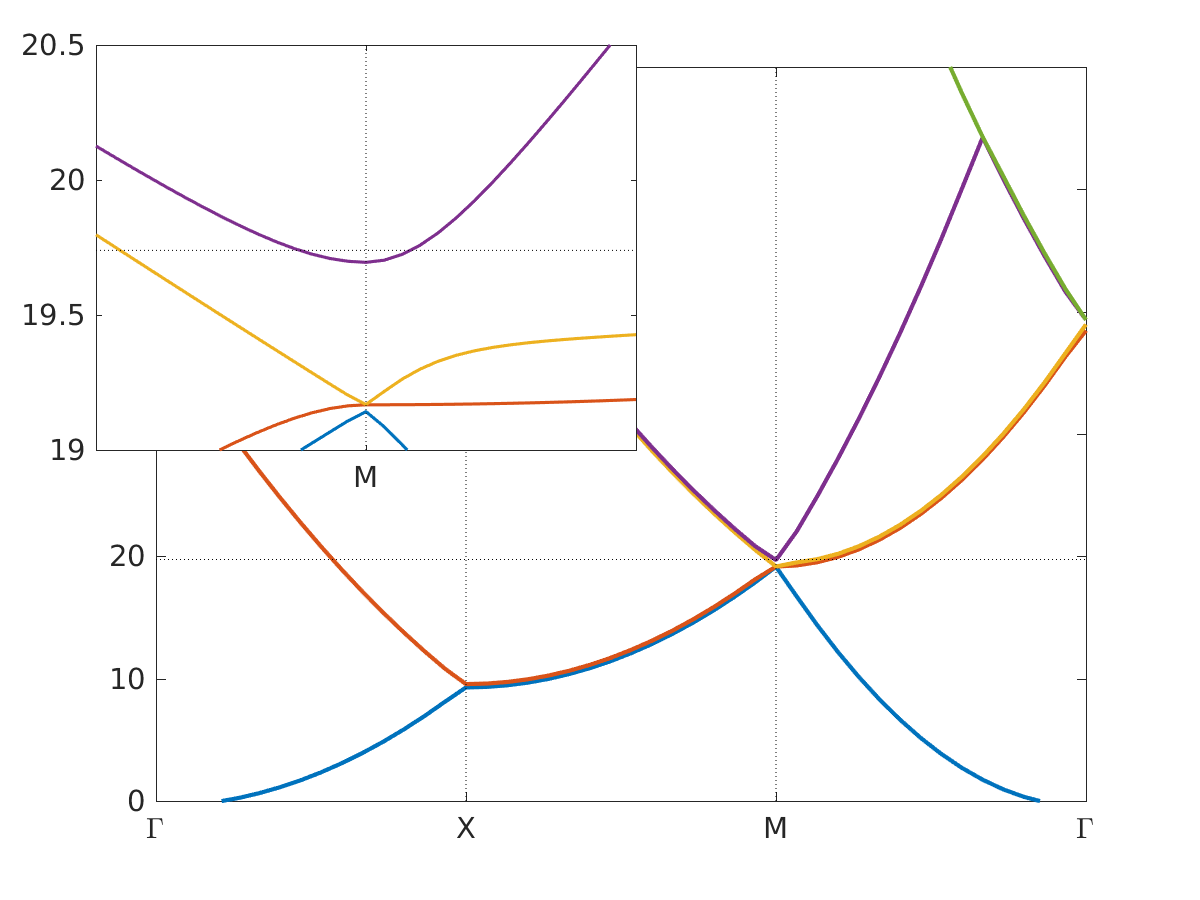}}
\end{center}
\end{figure}

\begin{figure}[h!]
\caption{Eigenvalue ordering, up to $\mathcal{O(\eps)}$:  $\mu_{+1}^{\eps}  < \mu_{+ i}^{\eps}  = \mu_{- i}^{\eps} <  \mu_{-1}^{\eps}  $. }
\label{figure10}
\begin{center}
   \subfloat[Potential]{\includegraphics[width=.45\textwidth]{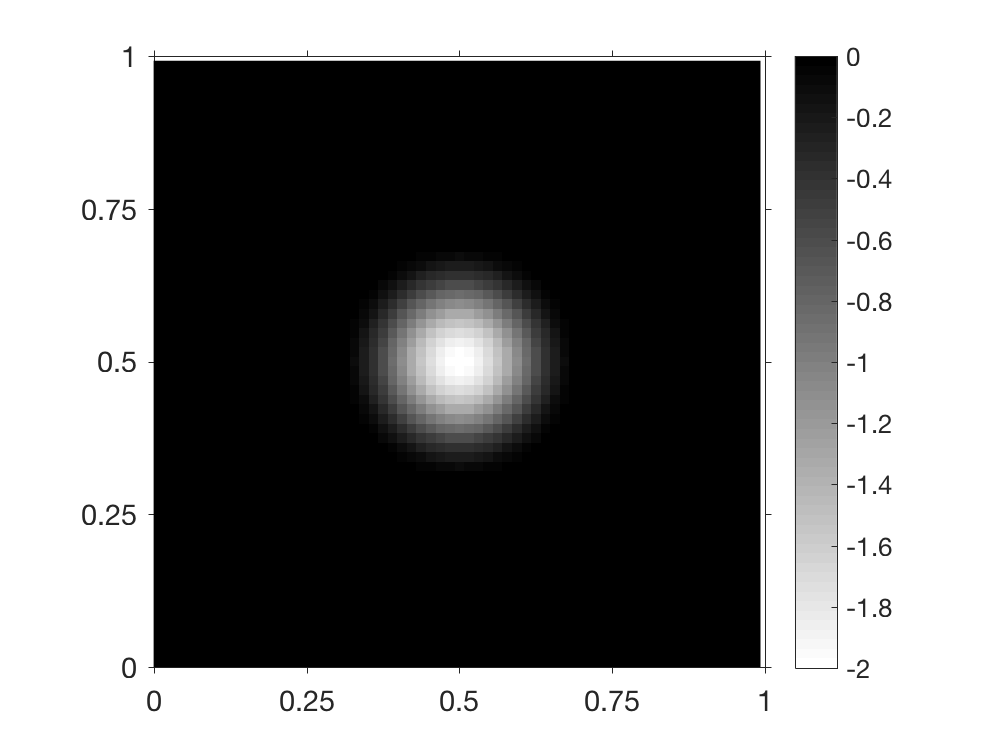}}
\subfloat[Dispersion curves for $H_V$]{\includegraphics[width=.45\textwidth]{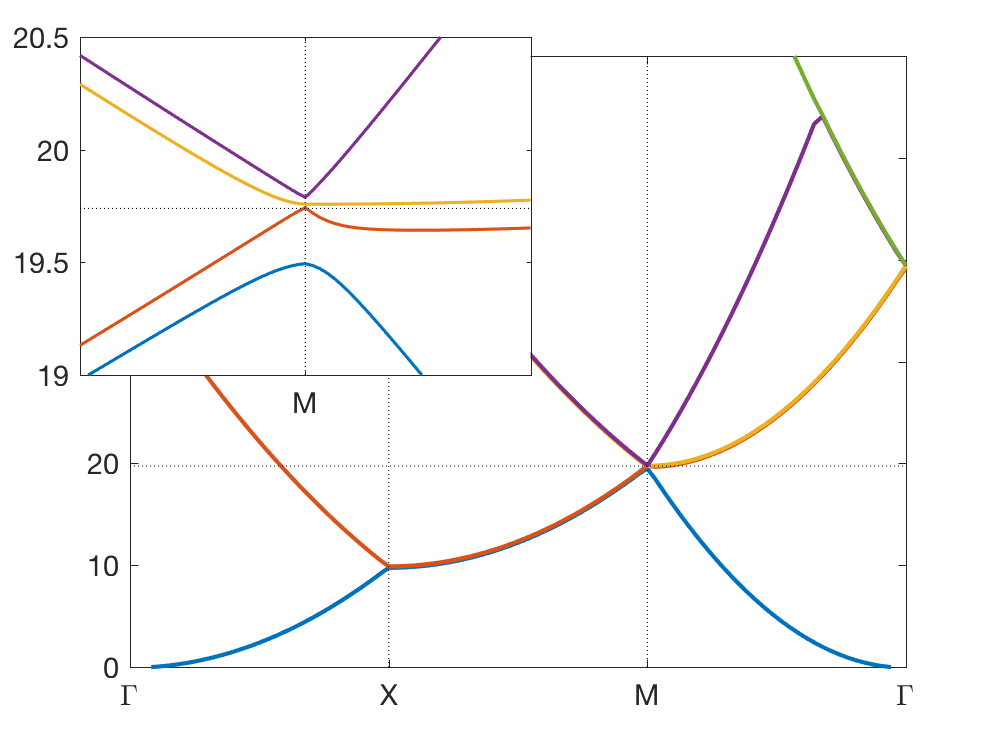}} 
\end{center}
\end{figure}


\subsection{Periodic arrays of deep potential wells; the strong binding regime} \label{s:numer2}
We will consider here potentials which are a sum over translates of a fixed potential  $\tilde V$, which is localized within a unit cell:
\begin{equation*}
V(\bx) = \sum_{\bfm \in \Z^2} \tilde V (\bx- \bfm) .
\end{equation*}
The potential $\tilde V$ is taken to be a finite sum of well-localized identical negative Gaussians (potential wells) with centers located  within the primitive cell $[0,1]^2$,
 {\it i.e.} an arrangement of ``atoms'' within the unit cell.   
The regime where the depth of the atomic potentials is large is  the regime of {\it strong binding}. The spectral properties associated with the ``low-lying'' bands are  expected to be well-approximated, after appropriate rescaling, by a {\it tight-binding} limiting model; see, for example, 
\cite{Ashcroft-Mermin:76}. A rigorous analysis of the low-lying dispersion surfaces and their approximation by those of the tight-binding model was carried out for honeycomb structures in \cite{FLW-CPAM:17}.

Figures \ref{TB1} -- \ref{TB3} of this section display dispersion surfaces for $-\Delta+V$, for different choices of $V$, each  for increasing atomic well-depths. We now discuss these examples.

\begin{example}[Superposition over the Lieb lattice of Gaussian wells]\medskip

{\ }

\nit To address the structural stability Question \ref{tbquest} in the Introduction, we consider a periodic potential, whose restriction to the primitive cell is:
\begin{equation*}
\tilde V_L (\bx) = -V_0 (e^{-|\bx|^2/\sigma} + e^{-|\bx - (1/2, 0)|^2/\sigma} + e^{-|\bx - (0, 1/2)|^2/\sigma} )\ ,\ V_0>0.
\end{equation*}
As a typical value of $\sigma$ we take $\sigma = .001$  and vary the depth of the atomic wells by increasing $V_0$.  In Figure \ref{TB1}, we observe that as $V_0$ is increased the first $3$ (the low-lying) dispersion surfaces come together in a manner approaching that of the tight-binding limit shown in Figure \ref{fig:lieb-tb}; see Appendix \ref{tb-lieb} for a discussion of the tight-binding model. Also related is the discussion around Figure \ref{figure9}. 

 However, as indicated in Theorem \ref{quad-disp} for finite $V_0$, there are only two surfaces touching at $\bM = (\pi,\pi)$ and the $2$nd band has hyperbolic character.  The inset in Figure \ref{figure9} displays this character.

\end{example}
 
 \begin{example}[Superposition of $\Z^2-$translates of Gaussian wells] 
We consider the potential, whose restriction to the primitive cell is given by:
\begin{equation*}
\tilde V(\bx) = -V_0 e^{-|\bx|^2/\sigma} .
\end{equation*}
For $V_0$ positive and large, we expect that the lowest dispersion surfaces will be governed by a tight-binding model with a single band, namely the discrete Laplacian on a square lattice. Indeed, 
Figure \ref{TB2} shows that the first (lowest) dispersion surface separates from the other (higher) dispersion surfaces
and takes on a quadratic character in a neighborhood of $\bk=0$.  The second and third surfaces intersect for all $V_0$ as is consistent with the curves in Figure \ref{figure10}.
\end{example}

 \begin{example}[A potential which is not invariant under reflection about the line $x_1=x_2$]  
 In Figure \ref{TB3} we demonstrate how the absence of reflection symmetry about the line $x_1=x_2$ in the physical domain manifests itself in less symmetry in the dispersion surfaces. In this case, Theorem \ref{quad-disp} anticipates
lesser symmetry, manifested in the non-zero cross-term, $\beta$, in then normal form \eqref{mu-pert}.  Hence, we take a simple potential of the form
\begin{equation*}
\tilde V(\bx) = -V_0(  \cos(2 \pi (x_1 + 2 x_2)) + \cos (2 \pi (2 x_1 - x_2))  ).
\end{equation*}
Thus, in Figure \ref{TB3}, we observe a lack of reflection symmetry about the line $\kappa_1 = \kappa_2$ even as we take $V_0 \to \infty$ to account for the fact that the potential studies here is in reality periodic on a sub-cell of the primitive domain on which we work here.
\end{example}

\begin{figure}[h!]
\begin{center}
   \subfloat[$V_0 = 10.0$]{\includegraphics[width=5.5cm]{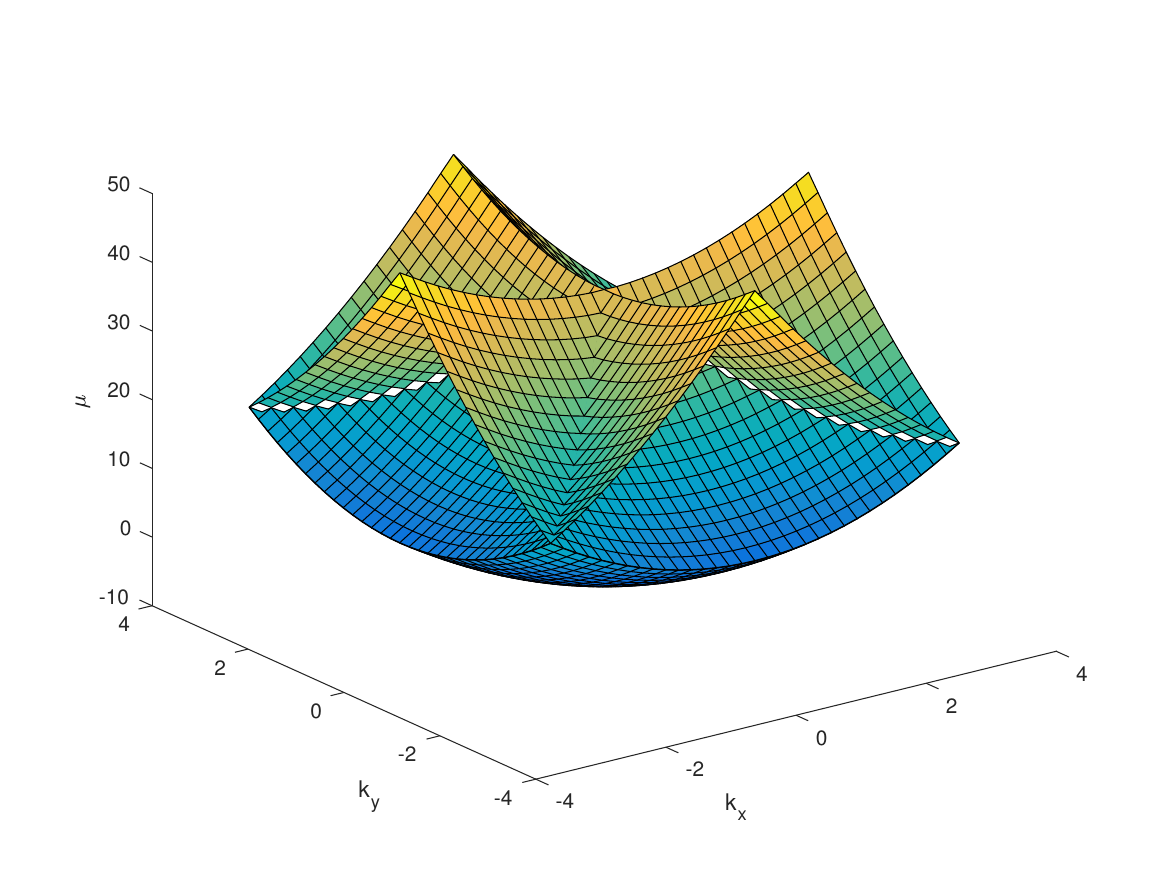} }
   \subfloat[$V_0 = 500.0$]{\includegraphics[width=5.5cm]{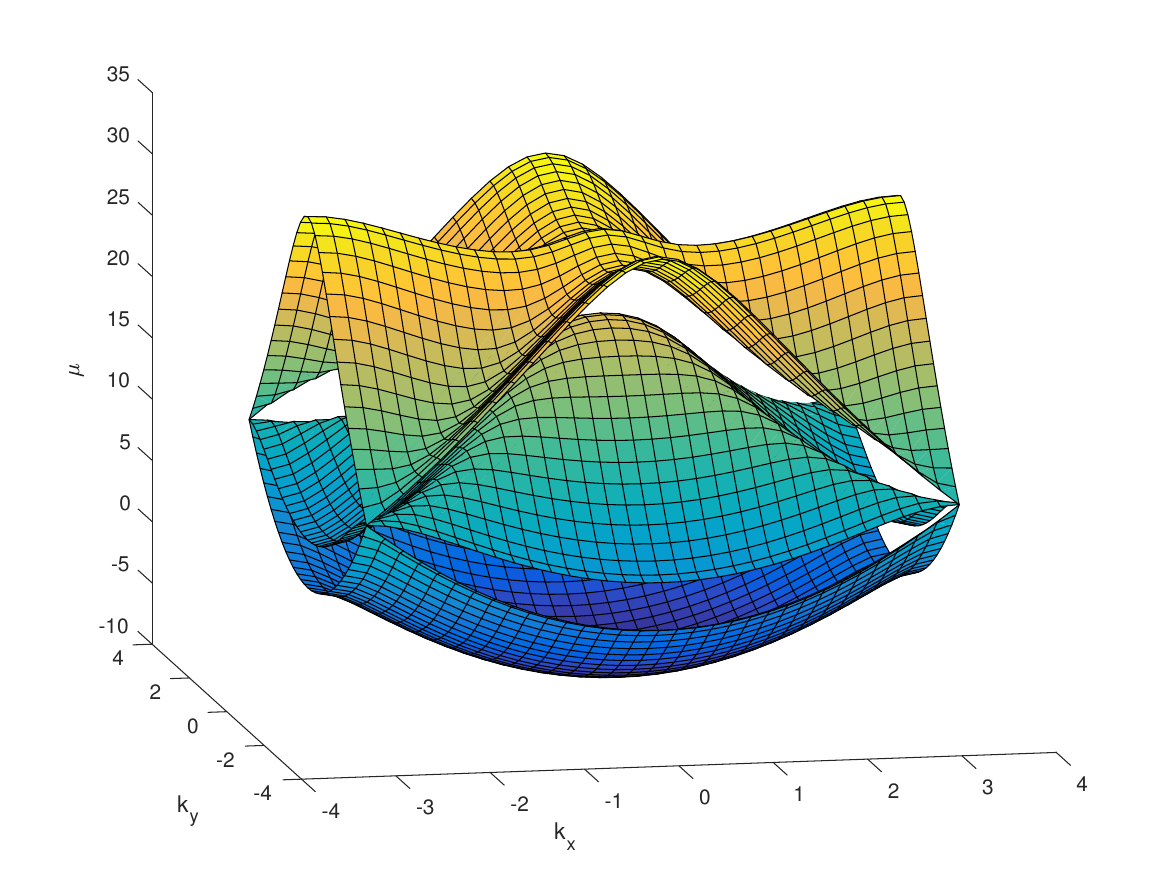} }\\
   \subfloat[$V_0 = 1000.0$]{\includegraphics[width=5.5cm]{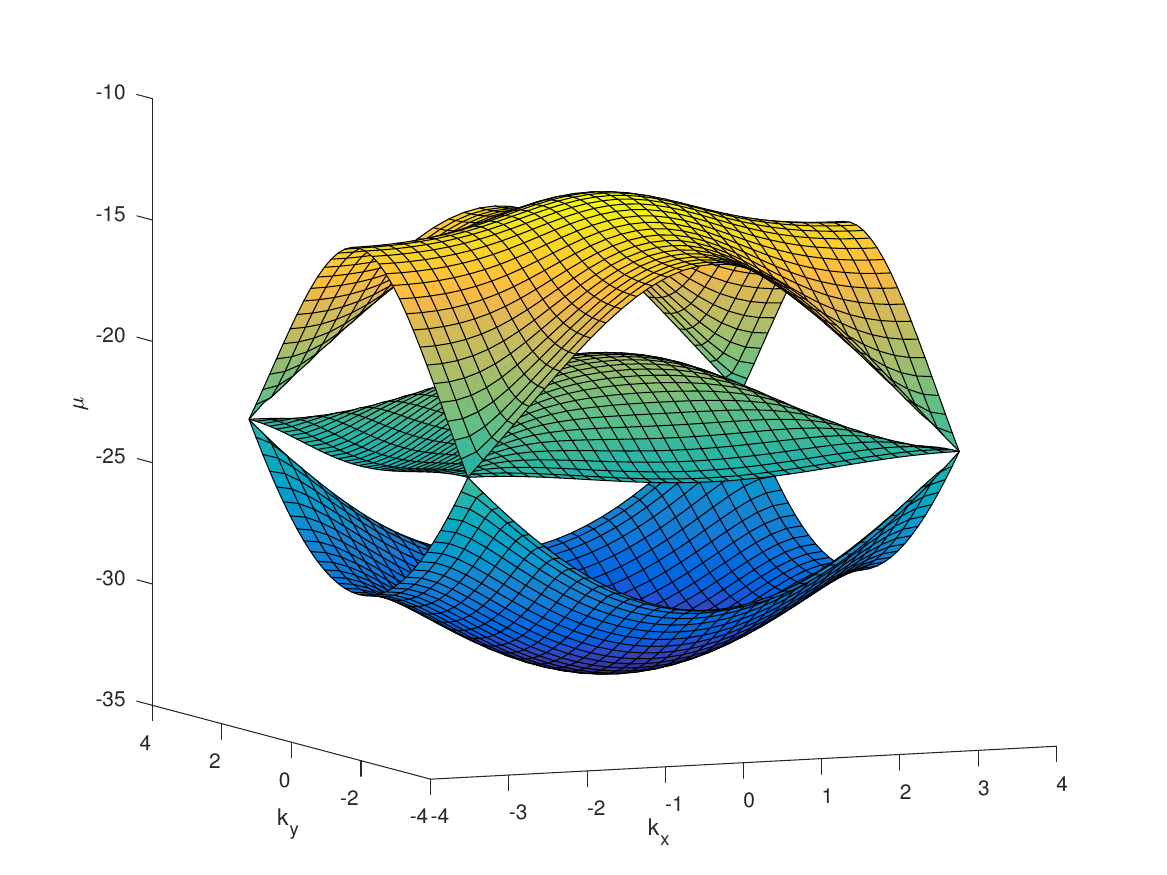}} 
   \subfloat[$V_0 = 2000.0$]{\includegraphics[width=5.5cm]{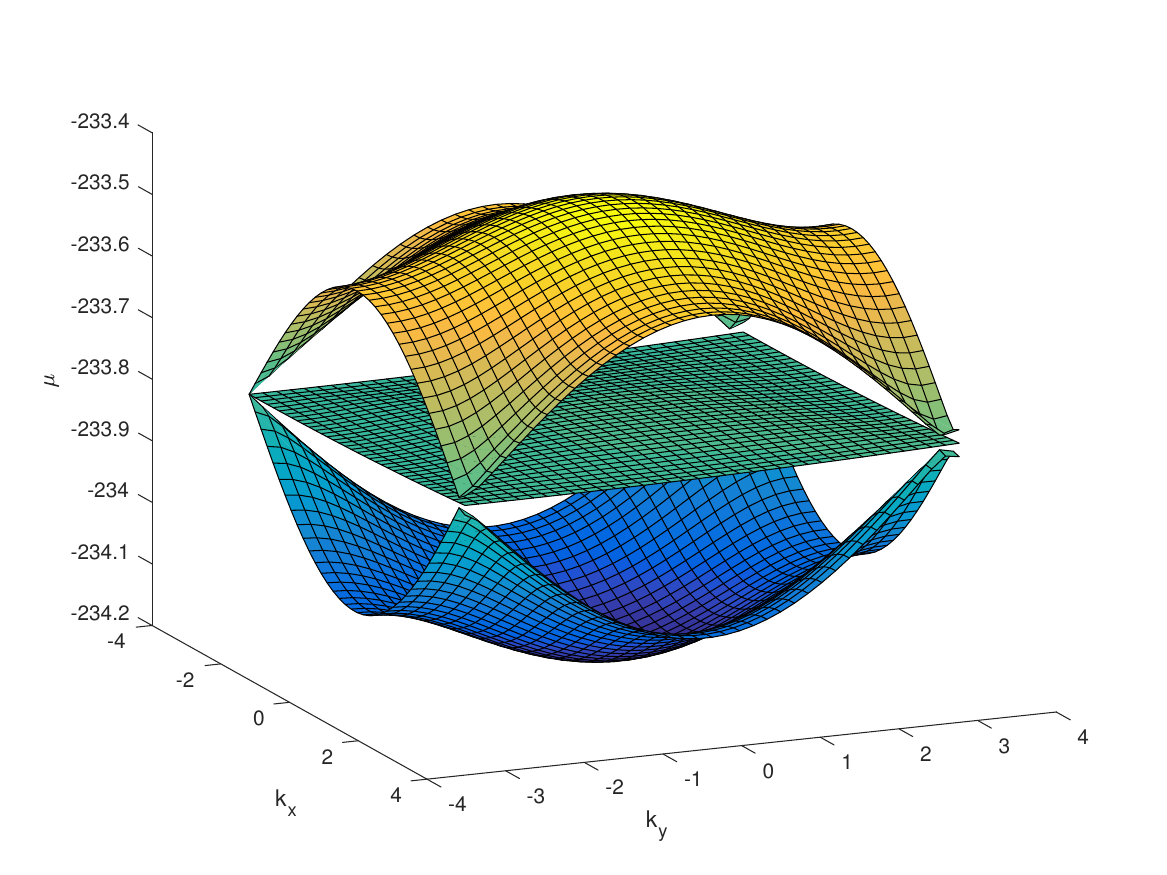}} 
\end{center}
\caption{A plot of the dispersion surfaces for the Gaussian Lieb Lattice potential with depth $V_0 = 10$  {\bf (top left)}, $500$ {\bf (top right)}, $1000$ {\bf (bottom left)} and $2000$ {\bf (bottom right)}.}
\label{TB1}
\end{figure}

\begin{figure}[h!]
\begin{center}
   \subfloat[$V_0 = 10.0$]{\includegraphics[width=.45\textwidth]{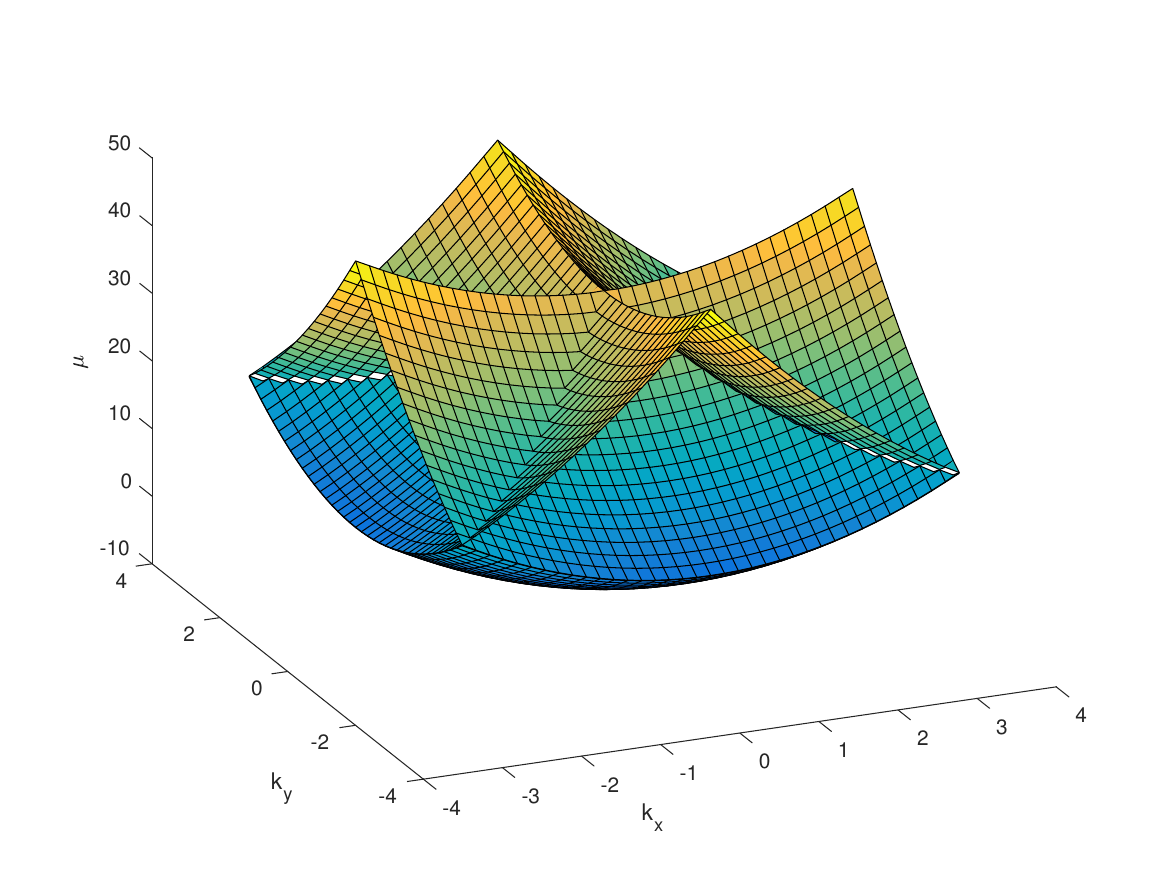}}
\subfloat[$V_0 = 500.0$]{\includegraphics[width=.45\textwidth]{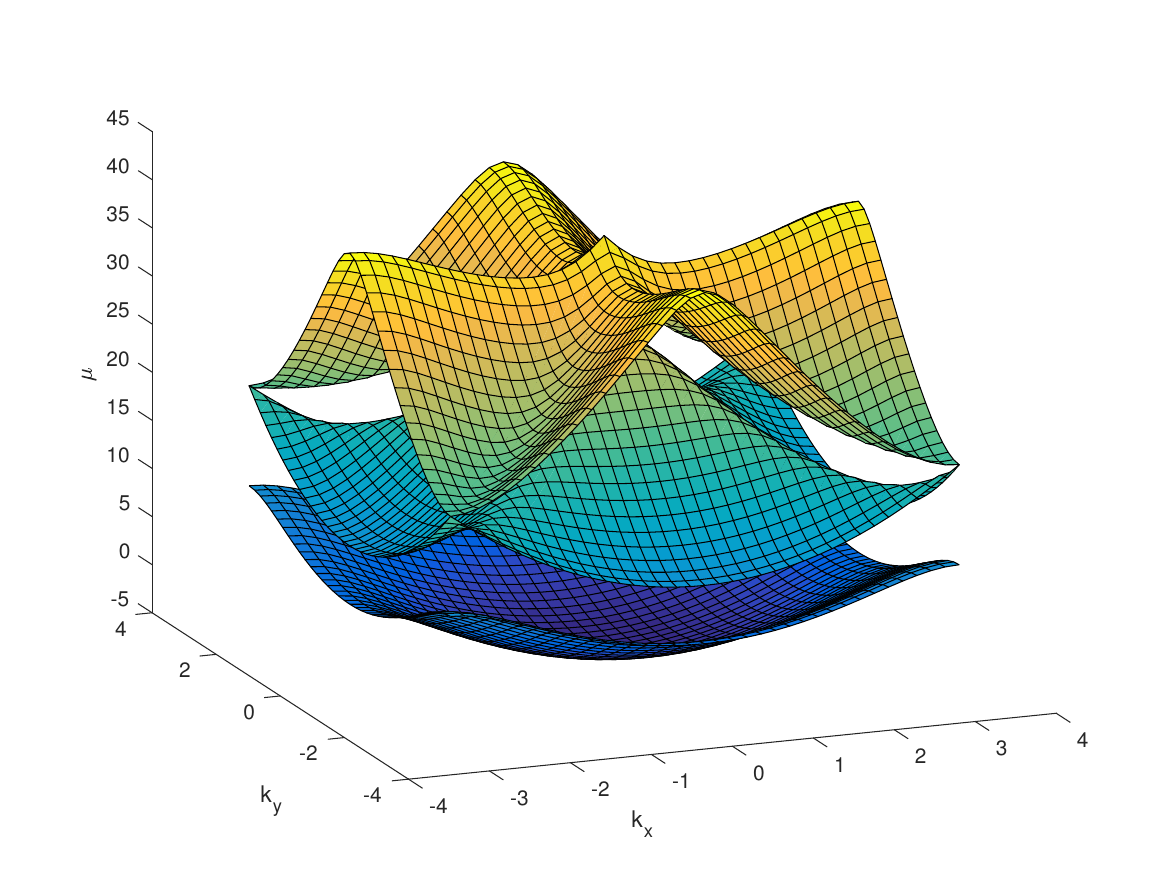}} \\
\subfloat[$V_0 = 1000.0$]{\includegraphics[width=.45\textwidth]{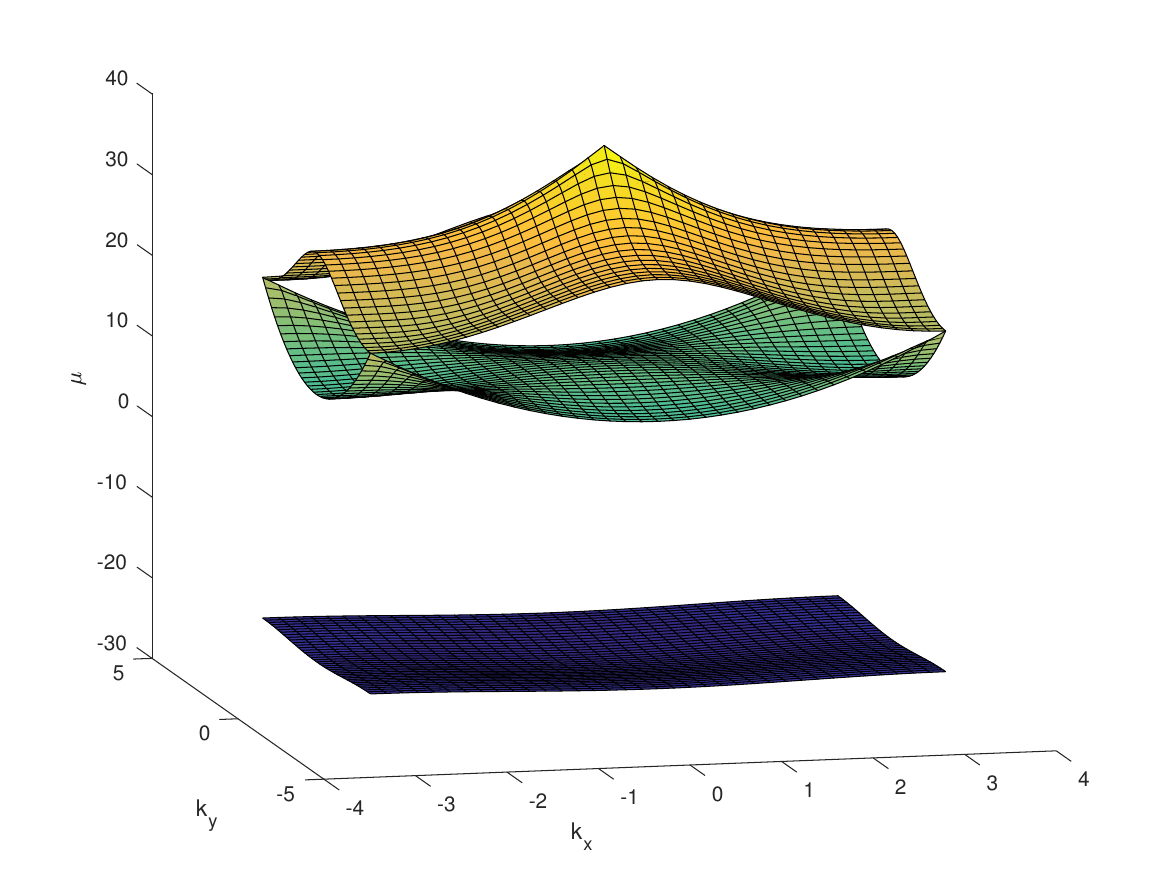}} 
\subfloat[A zoom in of the bottom surface for $V_0 = 1000.0$]{\includegraphics[width=.45\textwidth]{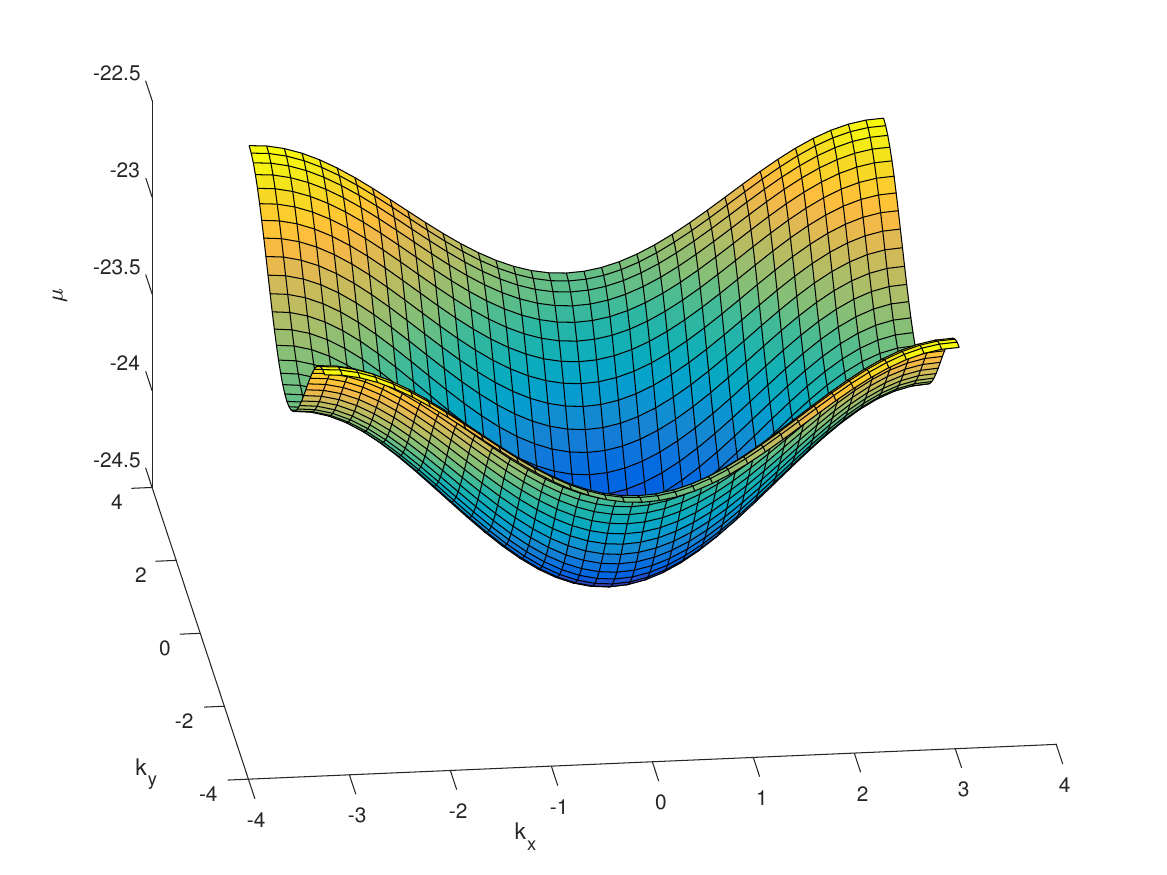}}
\end{center}
\caption{A plot of the dispersion surfaces for the Gaussian Square Lattice potential with depth $V_0 = 10$ {\bf (top left)}, $500$ {\bf (top right)}, $1000$ {\bf (bottom left)} and a close-up of the bottom surface at depth $V_0 = 1000$ {\bf (bottom right)}. }
\label{TB2}
\end{figure}

\begin{figure}[h!]
\begin{center}
 \subfloat[$V_0 = 10.0$]{\includegraphics[width=5.5cm]{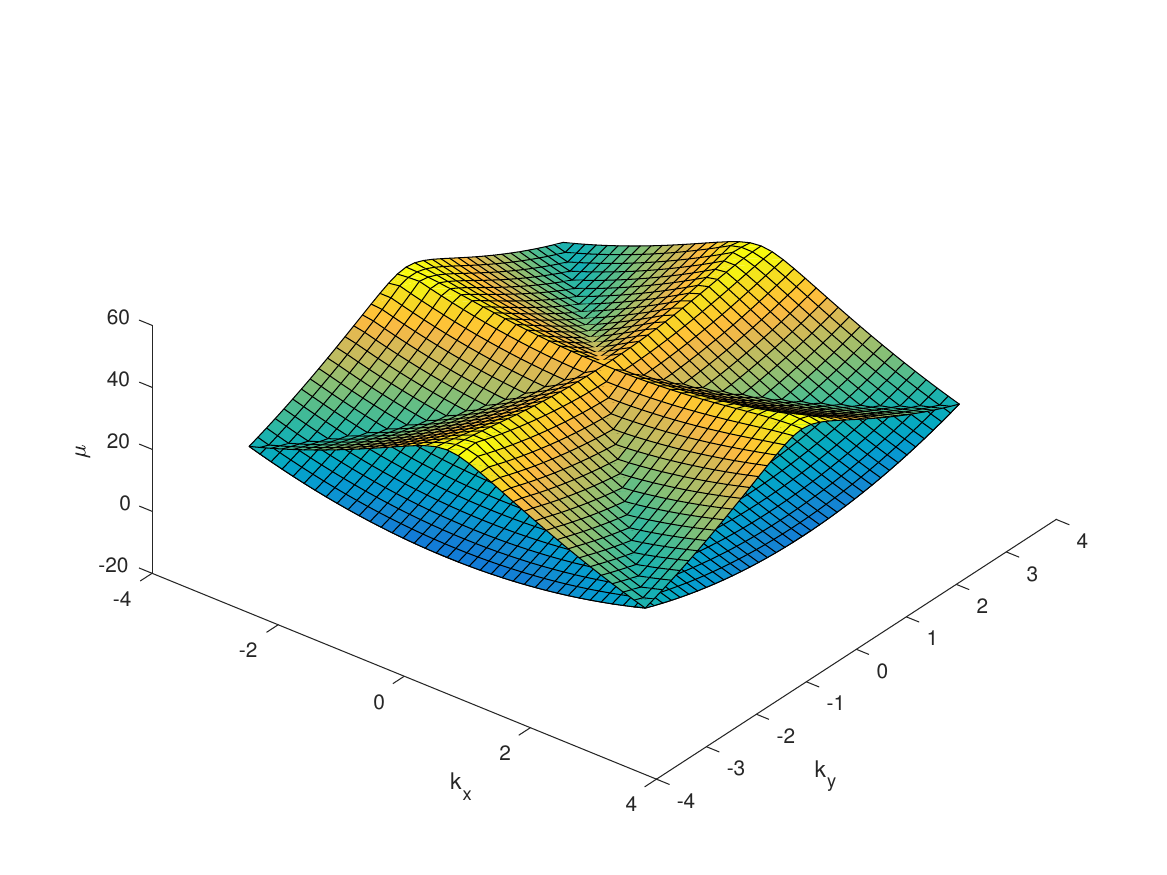} }
  \subfloat[$V_0 = 100.0$]{\includegraphics[width=5.5cm]{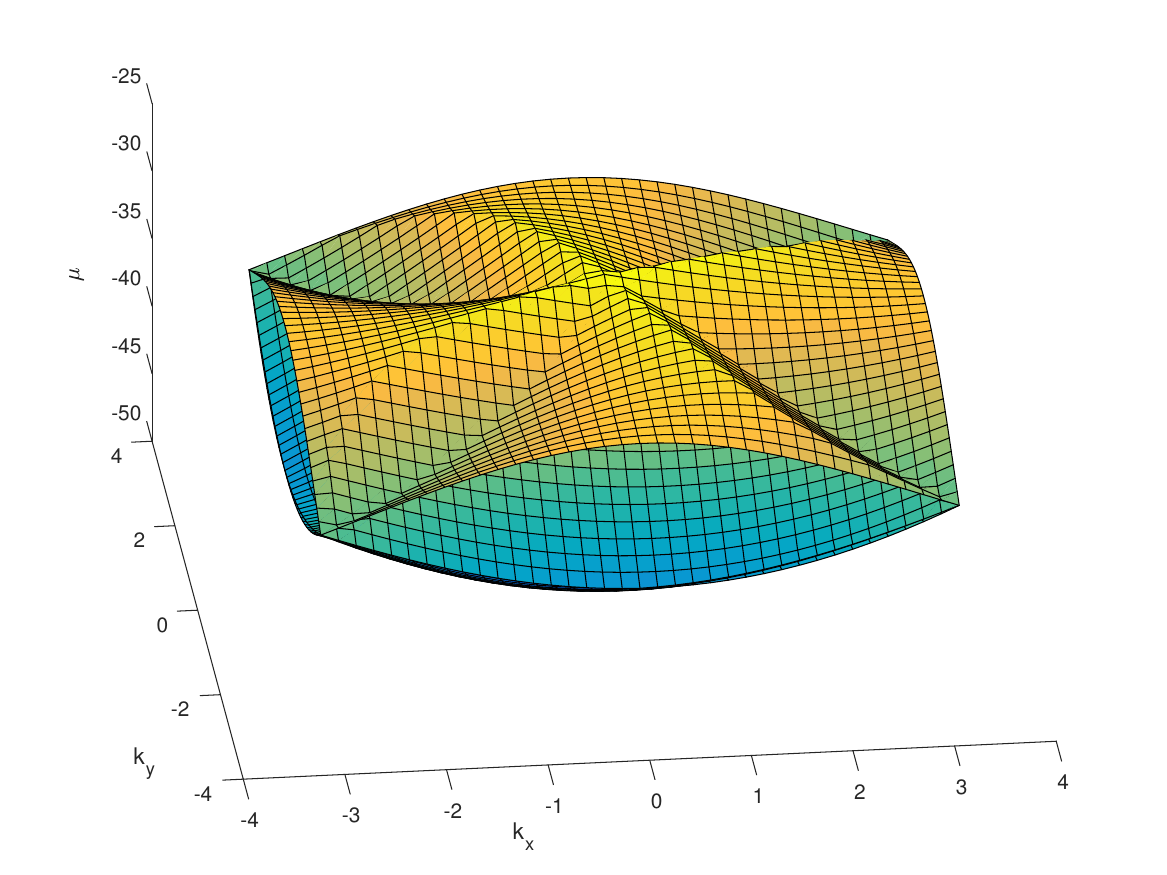} } \\
   \subfloat[$V_0 = 500.0$]{\includegraphics[width=5.5cm]{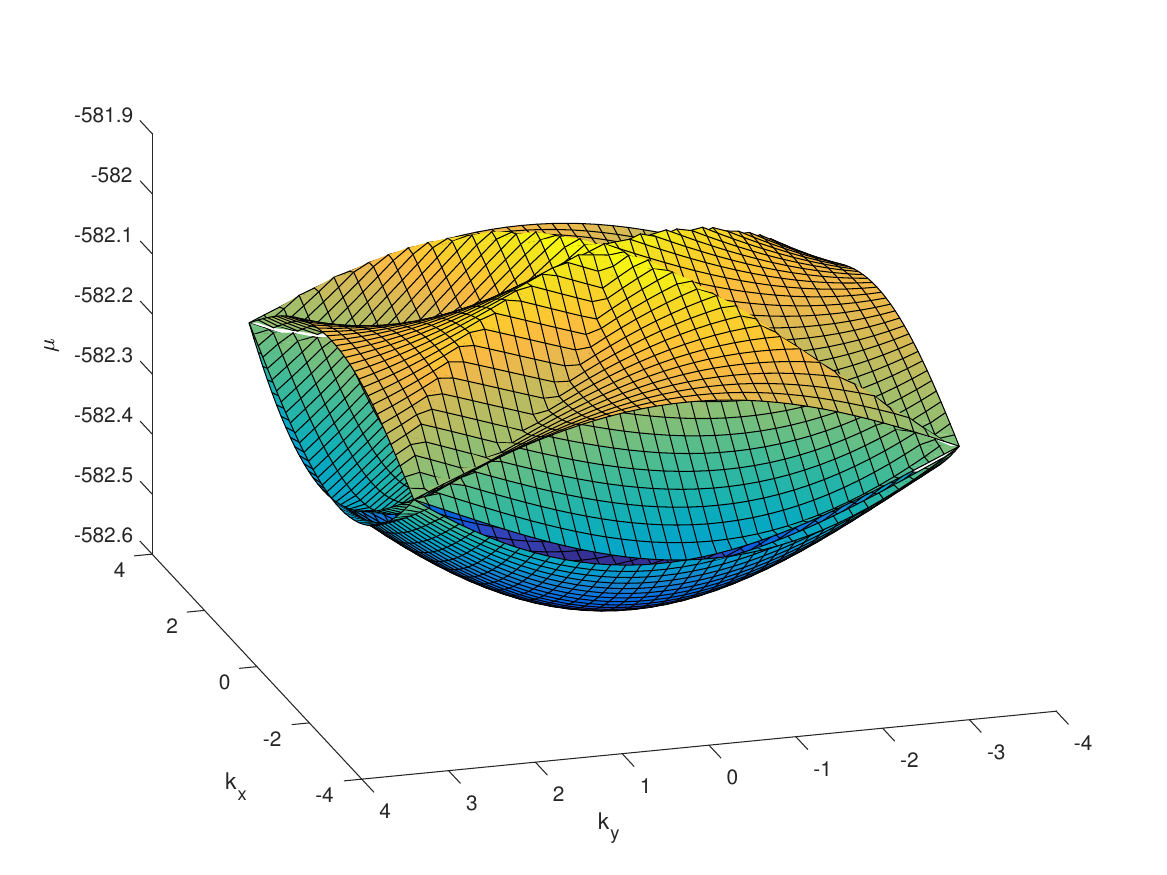} }
 \subfloat[$V_0 = 100.0$, $X \to \rho(X)$ cross-sections]{\includegraphics[width=5.5cm]{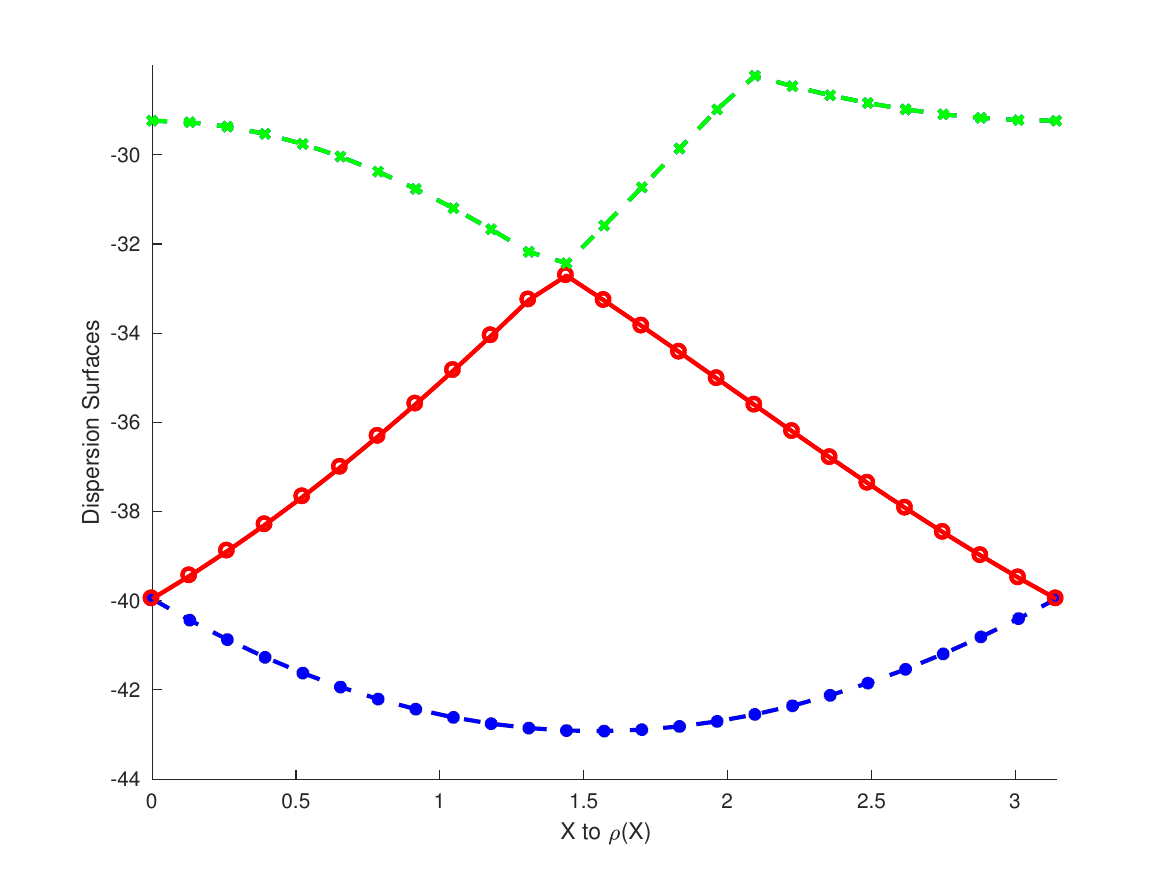} } 
\end{center}
\caption{A plot of the dispersion surfaces for the Non-Reflection Symmetric potential with depth $V_0 = 10$ {\bf (top left)} and $100$ {\bf (top right)}, $V_0 = 500$ {\bf (bottom left)} showing the lack of reflection symmetry in the surfaces.  To highlight the asymmetry, we have included a plot of cross-sections of the first three dispersion surfaces with $V_0 = 100$ along the line from $X = (0,\pi)$ to $\rho(X) = (\pi,0)$. }
\label{TB3}
\end{figure}

\appendix

\section{Tight-binding model for the Lieb lattice}\label{tb-lieb}

Denote by $\Psi^{(m,n)}_{A}, \  \Psi^{(m,n)}_{B}, \ \Psi^{(m,n)}_{C}$  the  amplitudes associated with the sites of the three sublattices comprising the Lieb lattice; see Figure \ref{fig:Lattices}(right).
We consider the nearest neighbor tight-binding model for the Lieb lattice is given by $(H^{TB} \Psi)_{mn} = E \Psi_{mn}$ for $m,n\in\Z$: 
\begin{equation} \label{eq:TBH}
\begin{pmatrix} 
\Psi^{(m,n)}_{B} + \Psi^{(m,n+1)}_{B} \\
\Psi^{(m,n)}_{A} +  \Psi^{(m,n)}_{C}  + \Psi^{(m-1,n)}_{C}+ \Psi^{(m,n-1)}_{A} \\
 \Psi^{(m,n)}_{B} + \Psi^{(m+1,n)}_{B} 
\end{pmatrix}
=E
\begin{pmatrix}
\Psi^{(m,n)}_{A} \\
\Psi^{(m,n)}_{B} \\
\Psi^{(m,n)}_{C}
\end{pmatrix},\quad m,n\in\Z.
\end{equation}
Quasi-periodic (plane-wave) solutions with Bloch momentum $\bk = (k_1,k_2) \in \brill = [-\pi, \pi]^2$ are of the form:
$ \Psi^{(m,n)} = e^{i(m k_1 + n k_2 ) }  \psi$, 
where $\psi \in \mathbb C^3$ is independent of $(m,n)$. Substituting into  \eqref{eq:TBH} we obtain the algebraic eigenvalue problem:
$ \left(\ A(\bk)\ -\  E(\bk) \ I_{_{3\times3}}\ \right) \psi \ =\ 0$, where
\[
A(\bk)  =
\begin{pmatrix}
0 &  1+e^{i k_2}  & 0\\
 1+e^{-i k_2 } & 0 &  1+e^{-i k_1}  \\
0 &  1+e^{i  k_1} & 0\\
\end{pmatrix}. 
\]
The three bands of the tight-binding model are given by the three solution branches of the algebraic equation: $\textrm{det}(A(\bk) - E(\bk) \textrm{I}) = 0$, which are given explicitly:
$$
E_0(\bk) = 0, \ E_\pm(\bk)=\pm \sqrt{  4 + 2 \cos k_1 + 2 \cos k_2 }.  
\qquad \bk \in \brill,
$$ 
These three dispersion surfaces are plotted in Figure \ref{fig:lieb-tb}.

\section{Dynamics of wave-packets spectrally localized near $\bM$ }\label{sec:wave-pkts}

%
 In this section, we study the Schr\"{o}dinger evolution:  $i \partial_{t} \psi =  H_V\psi$  for wave-packet initial conditions which are spectrally located near the $\bM$-point. In particular, denote by  $\mu_S$ a multiplicity two $L^2_\bM$ eigenvalue
  with corresponding eigenspace given by ${\rm span}\{\Phi_1,\Phi_2\}$; see Theorem \ref{quad-disp}. We consider initial conditions of the form $\psi(\bx,t=0)= C_{10}(\delta\bx)\Phi_1(\bx)+C_{20}(\delta\bx)\Phi_2(\bx)$,  where $0<\delta \ll1$ and $C_{10}$ and $C_{20}$ are smooth and localized functions on $\R^2$. The Floquet-Bloch components of such initial conditions are concentrated near $\bM$.

Here we derive a formal multi-scale solution. A proof of the validity of this expansion can be derived along the lines of \cite{FW-CMP:14} in the context of honeycomb structures; see also \cite{allaire}.

We seek a solution depending on multiple spatial and temporal scales:
  $\psi=\psi^\delta(\bx,t;\vec\bX,\vec T)$, where $\vec\bX=(\bX_1,\bX_2,\dots)=(\delta\bx_1,\delta^2\bx_2,\dots)$ and $\vec{T}=(T_1,T_2,\dots)=(\delta t_1,\delta^2 t_2,\dots)$. In terms of this extended set of variables, the time-dependent Schr\"{o}dinger equation becomes:
  \begin{equation}\label{ms-schroed}
  i( \partial_{t} + \delta  \partial_{T_{1}} + \delta^{2}  \partial_{T_{2}}  + \ldots ) \psi^{\delta} = - \left( \nabla_{\bx} + \delta \nabla_{\bX_{1}} + \delta^{2}  \nabla_{\bX_{2}} \right)^{2} \psi^{\delta} + V(\bx) \psi^{\delta} + \ldots.
  \end{equation}
We seek a solution which is time-harmonic with respect to the fast time variable, $t$: 
\[ \psi^{\delta} = e^{-i\mu_S t}\ \sum_{j\ge0}\ \delta^j\ \psi_{j}(\bx;\vec\bX,\vec T).\]
  Substitution into \eqref{ms-schroed} and equating terms of order $\mathcal{O}(\delta^j),\ j\ge0$ we obtain the
 following hierarchy of equations at each order in the small parameter $\delta$. The first several are:
  \begin{align}
&\mathcal{O}(\delta^{0} ):&( \mu_S -H_V)  \psi_{0} &= 0 \ ; \label{order-epsi0}\\
&\mathcal{O}(\, \delta \,\, ):& (\mu_S - H_V) \psi_{1} &= -\left( i \partial_{T_{1}} + 2 \  \nabla_{\bx} \cdot \nabla_{\bX_{1}}  \right) \psi_{0} \ ; \label{order-epsi1}\\
&\mathcal{O}( \delta^{2} ):& (\mu_S - H_V) \psi_{2} &= -( i \partial_{T_{2}}  + 2 \ \nabla_{\bx} \cdot \nabla_{\bX_{2}} + \Delta_{\bX_{1}})\psi_{0}
    	-(  i \partial_{T_{1}} + 2 \nabla_{\bx} \cdot \nabla_{\bX_{1}}) \psi_{1} \ .  \label{order-epsi2}
  \end{align} 
  We regard each equation as a linear elliptic equation (non-homogeneous for $j\ge1$) in the variable $\bx\in\R^2$, depending on 
 slow variables $(\vec\bX,\vec T)$, treated as frozen parameters . The precise dependence of $\psi_{j}$ on these parameters is determined through solvability conditions for the 
 above hierarchy. In fact, to the order in $\delta$ that we solve, $\mathcal{O}(\delta^2)$, we find it necessary to modulate only the 
 scales $\bX_1, \bX_2, T_1$ and $T_2$, and so henceforth we assume $\vec\bX=(\bX_1, \bX_2)$ and  $\vec T=(T_1,T_2)$.
  
Consider the $\mathcal{O}(\delta^0)$ equation \eqref{order-epsi0}. By assumption, the general solution is
   \begin{equation}
  \psi_0(\bx;\vec\bX, \vec T)\ =\  C_{1}(\vec\bX, \vec T) \Phi_{1}(\bx) +C_{2}(\vec\bX, \vec T) \Phi_{2}(\bx)\ ,
\label{O0sol}  \end{equation}
  where $C_{1}(\vec\bX, \vec T)$ and $C_{2}(\vec\bX, \vec T)$  are to be determined. 
The expression \eqref{O0sol} satisfies the $\bM-$pseudo-periodic boundary condition $\psi_0(\bx+\bv;\vec\bX, \vec T)=e^{i\bM\cdot\bv} \psi_0(\bx;\vec\bX, \vec T)$ for all $\bv\in\Z^2$ and all $\bx\in\R^2$ and we shall impose this same pseudo-periodicity at all subsequent orders:  \[ \psi_j(\bx+\bv;\vec\bX, \vec T)=e^{i\bM\cdot\bv} \psi_0(\bx;\vec\bX, \vec T)\ ,\ j\ge1.\]

  Continuing on to $\mathcal{O}(\delta^2)$, we have from  \eqref{order-epsi1} that
    \begin{align}
  (\mu_S - H) \psi_{1} &=  \sum_{q=1}^2\ \left[\  -i\D_{T_1}C_q\ \Phi_q(\bx)\ +\ 2\nabla_{\bX_1}C_q\cdot \nabla_\bx\Phi_q(\bx)\ \right]
  \label{eqn:psi1}\\
  \psi_{1} (\bx+\bv;\vec \bX,\vec T) &= e^{i\bM \cdot \bv}\psi_1(\bx; \vec{\bX}, \vec{T})\nn
  \end{align}

Solvability of \eqref{eqn:psi1} requires that orthogonality of the right hand side of \eqref{eqn:psi1} to the span of $\{\Phi_1,\Phi_2\}$ .
Using the orthonormality relations: $\langle \Phi_{p}, \Phi_{q} \rangle = \delta_{p,q}, \ p,q=1,2$
and Proposition \ref{f-gradf}: $\langle \Phi_{p}, \nabla_{\bx} \Phi_{q} \rangle = {\bf 0}, \ p,q=1,2$,  we obtain that $\partial_{T_{1}} C_p = 0, \ p=1,2$. 
Therefore, in terms of the resolvent operator $\mathscr{R}(\mu_S) = (H - \mu_S I)^{-1}$, we have  
\begin{equation}\label{psi1-exp}
\psi_1(\bx;\bX_1,\bX_2,T_2)\ =\  2 \ \mathscr{R}(\mu_S) \sum_{q=1}^{2}\ \nabla_{\bX_1}C_q(\bX_1,\bX_2,T_2) \cdot \nabla_{\bx}\Phi_q(\bx)\ .
\end{equation}

We proceed finally to the equation and boundary conditions for $\psi_2$ at \eqref{order-epsi2}:
  \begin{align}
(\mu_S - H) \psi_{2} &= -( i \partial_{T_{2}}  + 2 \nabla_{\bx} \cdot \nabla_{\bX_{2}} + \Delta_{\bX_{1}})\psi_{0}
    	- 2 \nabla_{\bx} \cdot \nabla_{\bX_{1}}\ \psi_{1}\ \label{O2eqn}\\
   \psi_{2} (\bx+\bv; \vec{\bX}, \vec{T}) &= e^{i\bM \cdot \bv}\psi_2(\bx; \vec{\bX}, \vec{T}),\nn
  \end{align}
  where the expressions for $\psi_0$ and $\psi_1$ are displayed in \eqref{O0sol} and \eqref{psi1-exp}.

  \begin{prop}\label{env-eqns} 
   A sufficient condition for the solvability of \eqref{O2eqn} is that the pair of amplitudes $C_1=C_1(\bX_1,T_2)$ and $C_2=C_2(\bX_1,T_2)$ 
satisfy the coupled system of constant coefficient {\it homogenized} Schr\"{o}dinger equations:
\begin{equation}\label{env-eqn}
 i\partial_{T_{2}} C_{p} = -\Delta_{\bX_{1}}C_{p} + 4 \ 
 \sum_{q=1}^2\sum_{r,s=1}^2  a_{r,s}^{p,q} \ 
 \frac{\D^2 C_q}{ \D{X_1}_r \D{X_1}_s}  ,\quad p=1,2.
  \end{equation}
Here,  $ a_{s,r}^{p,q}$ are the matrix elements of $(A^{p,q})_{sr}$ displayed in \eqref{A-def}: $
 a_{r,s}^{p,q} = \langle \partial_{r} \Phi_{p}, \mathscr{R}(\mu_S) \partial_{s} \Phi_{q}\rangle
$ 
which were simplified using symmetry arguments in Section \ref{pf-quad-disp}.
\end{prop}
As a consequence of Proposition \ref{env-eqns}, we have item \ref{wave-pack-summ} in our summary of results, Section \ref{sec:summ}. Namely, solutions of the time-dependent Schr\"odinger equation  with initial conditions of the form:
\[ \psi(\bx,0)=\psi^\delta_0(\bx)\ =\ C_1(\bX)\ \Phi_1(\bx)\ +\ C_2(\bX)\ \Phi_2(\bx),\]
evolve on large, finite time scales according to:
\[
 i\frac{\D}{\D  T} C_{p} =  
 -\sum_{q=1}^2\sum_{r,s=1}^2  
 \frac{\D}{\D X_r}\ \Upsilon^{p,q}_{r,s}\ \frac{\D}{\D X_s}\ C_q  ,\quad p=1,2,
\]
 where $T = T_2 = \delta^2 t$ and $ \Upsilon^{p,q}_{r,s} $ depend on $(A^{p,q})_{sr}$  as stated above.

\begin{proof}
Taking the inner product of the right hand side of \eqref{O2eqn} with $\Phi_p(\bx),\ p=1,2$, substituting in the expressions for $\psi_0$ and $\psi_1$, recalling that $\D_{T_1}C_p = 0$ and applying Proposition \ref{f-gradf}, we have:
\begin{align}
(i \partial_{T_{2}} + \Delta_{\bX_1} ) C_p  -\  4 \ \sum_{q=1}^{2} \langle \ \nabla_{\bx}  \Phi_p \cdot  \nabla_{\bX_1} , \  \mathscr{R}(\mu_S) \  \nabla_{\bx}\Phi_q \  \cdot \nabla_{\bX_1} C_q \rangle\ =\ 0,\ \ p=1,2.
\end{align}
Expansion of the dot products yields \eqref{env-eqn}. \end{proof}

\begin{remark} The dispersion relation for system \eqref{env-eqn} is 
\[
\det\tM_{\rm{app}}(\nu,\kappa)=
\nu^2-|q(\kappa)|^2\ =\ \nu^2\ -\ |\gamma(\kappa_1^2-\kappa_2^2)\ +\ 2\beta\kappa_1\kappa_2|^2\ =\ 0,
\]
for $ \kappa=(\kappa_1,\kappa_2)\in\R^2$
yielding two branches given by the leading order expressions in \eqref{nf}.
\end{remark}

 \section{Calculations for $-\Delta+\eps V$, $V$ admissible and $\eps$ small}\label{mu-eps-coeff}

 Theorem \ref{quad-disp} and Corollary \ref{rho-inv}  give a precise description of touching dispersion surfaces at the vertices of $\brill$.
  The local expansions of these dispersion surfaces are given in terms of coefficients  $\alpha$, $\beta$ and $\gamma$. In this section we consider $H^\eps=-\Delta+\eps V$, where $\eps$ is real, non-zero and sufficiently small, and we obtain the leading order expressions of $\alpha^\eps$, $\beta^\eps$ and $\gamma^\eps$ for $|\eps|>0$ small, required to complete the  proof of Theorem \ref{small-eps-disp}.
 %
%
%
%
%
%
%
%
The expressions for  $\alpha^\eps$, $\beta^\eps$ and $\gamma^\eps$ are displayed in  \eqref{albega}  and are given by:
\begin{equation}
a_{\ell,m}^{(r),(s)}=\langle \D_{x_\ell} \Phi_{(r)}^\eps, \mathscr{R}^\eps(\mu_S^\eps) \D_{x_m}\Phi_{(s)}^\eps \rangle, \ \ r,s \in \{ +i, -i\},\quad \ell, m\in \{1,2\}.
\label{a-app}\end{equation}
In \eqref{a-app} we have used  the notation $\Phi_1=\Phi_{(+i)}$ and $\Phi_2=\Phi_{(-i)}$ of Theorem \ref{mu-epsilon-evals}.
We now seek the leading order behavior of $a_{\ell, m}^{(r) ,(s)}$ for $\eps$ small and non-zero.

Recall by Theorem \ref{mu-epsilon-evals} there are four $L^2_{\bM}$ eigenvalues counting multiplicity on an $\mathcal{O}(\eps)$ neighborhood of $\mu^\eps_S$. Two correspond to multiplicity one $L^2_{\bM, +i}$ and $L^2_{\bM, -i}$ eigenstates $\Phi^\eps_{(+i)}$ and $\Phi^\eps_{(-i)}$ with corresponding eigenvalue which we denote $\mu_S^\eps$. The other two eigenpairs consist of two distinct $L^2_{\bM, +1}$ and $L^2_{\bM, -1}$ eigenstates $\Phi^\eps_{(+1)}$ and $\Phi^\eps_{(-1)}$ with corresponding eigenvalues we denote $\mu_{(+1)}^\eps = \mu_S + \mathcal{O}(\eps)$ and $\mu_{(-1)}^\eps = \mu_S + \mathcal{O}(\eps),$ respectively; see equations \eqref{muK-eval-expressions-1} and \eqref{muK-eval-expressions-2}. All other eigenvalues of $H^\eps$ are at a distance of order 1 from $\mu_S^\eps$ for $\eps$ small, and hence we express $a_{\ell, m}^{(r),(s)}$ as dominant parts coming from these ``nearby'' eigenvalues, with a remainder which is smaller for $\eps > 0$ and small.

 We have 
\begin{equation}
\begin{split}
a^{ (r), (s) }_{\ell, m} &=  \langle \D_{x_\ell} \Phi_{(r)}^\eps, \mathscr{R}(\mu_S^\eps) \D_{x_m}\Phi_{(s)}^\eps \rangle 
\\&= \sum_{q \in \{+1, -1\}} \frac{ \overline{ \langle \Phi^\eps_{(q)}, \D_{x_\ell}\Phi_{(r)}^\eps \rangle}\langle \Phi^\eps_{(q)}, \D_{x_m}\Phi_{(s)}^\eps \rangle}{\mu_{(q)}^\eps - \mu_S^\eps}
 + \sum_{b \geq 5} \frac{ \overline{ \langle \Phi^\eps_{b}, \D_{x_\ell}\Phi_{(r)}^\eps \rangle}\langle \Phi^\eps_{b}, \D_{x_m}\Phi_{(s)}^\eps \rangle}{\mu_{b}^\eps - \mu_S^\eps}.
\end{split}\label{dphi-exp}
\end{equation}
Observe the terms of \eqref{dphi-exp}, for  $q \in \{+1, -1\}$, are of order $\mathcal{O}(\frac{1}{\eps})$, while the contributions from the higher-order bands $b \geq 5$ are $\mathcal{O}(1)$. From the expansion of  $a^{ (r), (s) }_{\ell, m}$, we show, in fact:
\begin{prop} \label{appF-albega}
Assume $|V_{11}| \ne|V_{01}|$. Given $\eps$ sufficiently small and non-zero, 
\begin{align}
\alpha^\eps &= 4 a_{1,1}^{1,1}= 4 a_{1,1}^{(+i),(+i)} = \frac{32\pi^2}{\eps} \left(\frac{V_{11}}{V_{11}^2 - V_{01}^2}\right) + \mathcal{O}(1);\label{alph-app}\\
\beta^\eps &= 4 a_{1,2}^{1,2}= 4 a_{1,2}^{(+i),(-i)}= \frac{32\pi^2}{\eps}\left( \frac{V_{11}}{V_{11}^2 - V_{01}^2}\right)+ \mathcal{O}(1);\label{beta-app}\\
\gamma^\eps &= 4a_{1,1}^{1,2}= 4 a_{1,1}^{(+i),(-i)}= -\frac{32\pi^2}{\eps} \ i  \left(\frac{V_{01}}{V_{11}^2 - V_{01}^2}\right)+ \mathcal{O}(1) . \label{gamma-app}
\end{align}

\end{prop}

\subsubsection{Derivation of Terms in Summand}
First, note that $\bk_1 = (2\pi, 0)^T$ and $\bk_2 = (0, 2\pi)^T$ are the dual lattice basis vectors for $\Lambda = \mathbb{Z}^2$.  
 
\begin{align}
 &\langle \Phi_{(+1)}^\eps,\D_{x_\ell} \Phi_{(+i)}^{\eps} \rangle = (\bk_1 - \bk_2)_\ell + i(\bk_1 + \bk_2)_\ell + \mathcal{O}(\eps); \label{phi3-phi1}\\
&\langle \Phi_{(-1)}^\eps,\D_{x_\ell} \Phi_{(+i)}^{\eps} \rangle = (\bk_2 - \bk_1)_\ell + i(\bk_1 + \bk_2)_\ell + \mathcal{O}(\eps);\label{phi4-phi1}\\
 &\langle \Phi_{(+1)}^\eps,\D_{x_\ell} \Phi_{(-i)}^{\eps} \rangle  = (\bk_2 - \bk_1)_\ell + i(\bk_1 + \bk_2)_\ell + \mathcal{O}(\eps);\label{phi3-phi2}\\
&\langle \Phi_{(-1)}^\eps,\D_{x_\ell} \Phi_{(-i)}^{\eps} \rangle  = (\bk_1 - \bk_2)_\ell + i(\bk_1 + \bk_2)_\ell + \mathcal{O}(\eps).\label{phi4-phi2}
\end{align}

We have derived the expressions
\begin{equation*}
\Phi_{(+i)}^{\eps}(\bx) = \sum_{\bfm \in \mathcal{S}} c_{(+i)}^\eps(\bfm) ( e^{i \bM^\bfm \cdot \bx}-i e^{i R\bM^\bfm \cdot \bx} - e^{i R^2\bM^\bfm \cdot \bx} + i e^{i R^3\bM^\bfm \cdot \bx})
\end{equation*}
\begin{equation*}
\begin{split}
\D_{x_\ell} \Phi_{(+i)}^{\eps}(\bx) &= \sum_{\bfm \in \mathcal{S}} c^\eps_{(+i)}(\bfm) \bigg[ i(\bM + m_1 \bk_1 + m_2 \bk_2)_\ell e^{i \bM^\bfm \cdot \bx}-i^2(\bM  + m_2 \bk_1 - (m_1+1) \bk_2)_\ell  e^{i R\bM^\bfm \cdot \bx} 
\\&\hspace{0mm}- i(\bM -(m_1+1)\bk_1 - (m_2 +1)\bk_2)_\ell e^{i R^2\bM^\bfm \cdot \bx} + i^2 (\bM -(1+m_2) \bk_1 + m_1 \bk_2)_\ell e^{i R^3\bM^\bfm \cdot \bx}\bigg]
\end{split}
\end{equation*}
\begin{equation*}
\Phi_{(\pm1)}^{\eps}(\bx) = \sum_{\bfm \in \mathcal{S}} c_{(\pm1)}^\eps(\bfm) ( e^{i \bM^\bfm \cdot \bx}\pm  e^{i R\bM^\bfm \cdot \bx} + e^{i R^2\bM^\bfm \cdot \bx} \pm e^{i R^3\bM^\bfm \cdot \bx})
\end{equation*}

Now we proceed to prove the results in \eqref{phi3-phi1}-\eqref{phi4-phi2}.
\begin{equation*}\label{phi-sum1}
\begin{split}
&\langle \Phi_{(+1)}^\eps,\D_{x_\ell} \Phi_{(+i)}^{\eps} \rangle = \\&\int_{\Omega} \sum_{\bn \in \mathcal{S}}\sum_{\bfm \in \mathcal{S}} \overline{c^\eps_{(+1)}(\bn)}c^\eps_{(+i)}(\bfm)
( e^{i \bM^\bn \cdot \bx}+  e^{i R\bM^\bn \cdot \bx} + e^{i R^2\bM^\bn \cdot \bx} + e^{i R^3\bM^\bn \cdot \bx}) 
 \\&\hspace{2mm}\ast
 \{ i (\bM^\bfm)_\ell e^{i \bM^\bfm \cdot \bx} + (R\bM^\bfm)_\ell  e^{i R\bM^\bfm \cdot \bx} -i(R^2\bM^\bfm)_\ell  e^{i R^2\bM^\bfm \cdot \bx} - (R^3\bM^\bfm)_\ell e^{i R^3\bM^\bfm \cdot \bx}\}d\bx\\
 &=\int_{\Omega} \sum_{\bn \in \mathcal{S}}\sum_{\bfm \in \mathcal{S}} \overline{c^\eps_{(+1)}(\bn)}c^\eps_{(+i)}(\bfm)
( e^{i \bM^\bn \cdot \bx}+  e^{i R\bM^\bn \cdot \bx} + e^{i R^2\bM^\bn \cdot \bx} + e^{i R^3\bM^\bn \cdot \bx}) 
 \\&\hspace{2mm}\ast
  \bigg[ i(\bM + m_1 \bk_1 + m_2 \bk_2)_\ell e^{i \bM^\bfm \cdot \bx}-i^2(\bM  + m_2 \bk_1 - (m_1+1) \bk_2)_\ell  e^{i R\bM^\bfm \cdot \bx} 
\\&\hspace{0mm}- i(\bM -(m_1+1)\bk_1 - (m_2 +1)\bk_2)_\ell e^{i R^2\bM^\bfm \cdot \bx} + i^2 (\bM -(1+m_2) \bk_1 + m_1 \bk_2)_\ell e^{i R^3\bM^\bfm \cdot \bx}\bigg].
  \end{split}
 \end{equation*}

\nit The leading-order term is  associated with $\bn = \bfm = (0,0) \in \mathcal{S}$, for which $c^\eps_{(+1)} (\bn) =  c^\eps_{(-1)} (\bn) = c^\eps_{(+i)} (\bfm) = 1.$ Then,
\begin{align*}
\langle \Phi_{(+1)}^\eps,\D_{x_\ell} \Phi_{(+i)}^{\eps} \rangle =& \int_{\Omega} 
( e^{i \bM^\bn \cdot \bx}+  e^{i R\bM^\bn \cdot \bx} + e^{i R^2\bM^\bn \cdot \bx} + e^{i R^3\bM^\bn \cdot \bx}) 
 \\&\hspace{5mm}  
 \big( i \bM_\ell e^{i \bM^\bfm \cdot \bx} + (\bM-\bk_2)_\ell  e^{i R\bM^\bfm \cdot \bx} -i(\bM -\bk_1 - \bk_2)_\ell  e^{i R^2\bM^\bfm \cdot \bx}
 \\&\hspace{15mm} - (\bM -\bk_1)_\ell e^{i R^3\bM^\bfm \cdot \bx}\big)d\bx + \mathcal{O}(\eps) \\
 &= [(\bM -\bk_2) - (\bM -\bk_1)]_\ell + i [ \bM -(\bM -\bk_1 - \bk_2)]_\ell + \mathcal{O}(\eps)\\&  = (\bk_1 - \bk_2)_\ell + i(\bk_1 + \bk_2)_\ell + \mathcal{O}(\eps).
 \end{align*}
 Similarly, we have 
 \begin{align*}
\langle \Phi_{(-1)}^\eps,\D_{x_\ell} \Phi_{(+i)}^{\eps} \rangle = &\int_{\Omega} 
( e^{i \bM^\bn \cdot \bx}-  e^{i R\bM^\bn \cdot \bx} + e^{i R^2\bM^\bn \cdot \bx} - e^{i R^3\bM^\bn \cdot \bx})
\\&\hspace{5mm}  
 \big( i \bM_\ell e^{i \bM^\bfm \cdot \bx}  + (\bM-\bk_2)_\ell  e^{i R\bM^\bfm \cdot \bx} -i(\bM -\bk_1 - \bk_2)_\ell  e^{i R^2\bM^\bfm \cdot \bx} 
 \\&\hspace{15mm}- (\bM -\bk_1)_\ell e^{i R^3\bM^\bfm \cdot \bx}\big)d\bx  \ + \ \mathcal{O}(\eps) \\
 &=- [(\bM -\bk_2) - (\bM -\bk_1)]_\ell + i [ \bM -(\bM -\bk_1 - \bk_2)]_\ell \  + \mathcal{O}(\eps)\\&  = (\bk_2 - \bk_1)_\ell + i(\bk_1 + \bk_2)_\ell + \mathcal{O}(\eps).
 \end{align*}
The inner products of $\Phi_{(+1)}$ and $\Phi_{(-1)}$ with $\D_{x_\ell}\Phi_{(-i)}$ are calculated similarly. Recall that 
\begin{equation*}
\Phi_{(-i)}^{\eps}(\bx) = \sum_{\bfm \in \mathcal{S}} c^\eps_{(-i)}(\bfm) ( e^{i \bM^\bfm \cdot \bx}+i e^{i R\bM^\bfm \cdot \bx} - e^{i R^2\bM^\bfm \cdot \bx} - i e^{i R^3\bM^\bfm \cdot \bx})
\end{equation*}
Then, 
 \begin{equation*}
\begin{split}
\D_{x_\ell} \Phi_{(-i)}^{\eps}(\bx) &= \sum_{\bfm \in \mathcal{S}} c^\eps_{(-i)}(\bfm) \bigg[ i(\bM + m_1 \bk_1 + m_2 \bk_2)_\ell e^{i \bM^\bfm \cdot \bx}+i^2(\bM  + m_2 \bk_1 - (m_1+1) \bk_2)_\ell  e^{i R\bM^\bfm \cdot \bx} 
\\&\hspace{-5mm}- i(\bM -(m_1+1)\bk_1 - (m_2 +1)\bk_2)_\ell e^{i R^2\bM^\bfm \cdot \bx} - i^2 (\bM -(1+m_2) \bk_1 + m_1 \bk_2)_\ell e^{i R^3\bM^\bfm \cdot \bx}\bigg].
\end{split}
\end{equation*}

\nit Expanding  $\langle \Phi_{(+1)}^\eps,\D_{x_\ell} \Phi_{(-i)}^{\eps} \rangle$,   we obtain
\begin{align*}
\langle \Phi_{(+1)}^\eps,\D_{x_\ell} \Phi_{(-i)}^{\eps} \rangle=& \int_{\Omega} 
( e^{i \bM^\bn \cdot \bx}+  e^{i R\bM^\bn \cdot \bx} + e^{i R^2\bM^\bn \cdot \bx} + e^{i R^3\bM^\bn \cdot \bx}) 
\\&\hspace{5mm}  
 \big(  i \bM_\ell e^{i \bM^\bfm \cdot \bx} - (\bM-\bk_2)_\ell  e^{i R\bM^\bfm \cdot \bx} -i(\bM -\bk_1 - \bk_2)_\ell  e^{i R^2\bM^\bfm \cdot \bx}
 \\&\hspace{15mm} + (\bM -\bk_1)_\ell e^{i R^3\bM^\bfm \cdot \bx}\big)d\bx \ + \ \mathcal{O}(\eps)\\
 &=-[(\bM -\bk_2) - (\bM -\bk_1)]_\ell + i [ \bM -(\bM -\bk_1 - \bk_2)]_\ell + \mathcal{O}(\eps)\\&  = (\bk_2 - \bk_1)_\ell + i(\bk_1 + \bk_2)_\ell + \mathcal{O}(\eps).
 \end{align*}
 Similarly,
 \begin{align*}
\langle \Phi_{(-1)}^\eps,\D_{x_\ell} \Phi_{(-i)}^{\eps} \rangle=& \int_{\Omega} 
( e^{i \bM^\bn \cdot \bx}-  e^{i R\bM^\bn \cdot \bx} + e^{i R^2\bM^\bn \cdot \bx} - e^{i R^3\bM^\bn \cdot \bx}) 
\\&\hspace{5mm}  
 \big( i \bM_\ell e^{i \bM^\bfm \cdot \bx} - (\bM-\bk_2)_\ell  e^{i R\bM^\bfm \cdot \bx} -i(\bM -\bk_1 - \bk_2)_\ell  e^{i R^2\bM^\bfm \cdot \bx}
 \\&\hspace{15mm} + (\bM -\bk_1)_\ell e^{i R^3\bM^\bfm \cdot \bx}\big) d\bx \ +\  \mathcal{O}(\eps). \\
 &= [(\bM -\bk_2) - (\bM -\bk_1)]_\ell + i [ \bM -(\bM -\bk_1 - \bk_2)]_\ell + \mathcal{O}(\eps)\\&  = (\bk_1 - \bk_2)_\ell + i(\bk_1 + \bk_2)_\ell + \mathcal{O}(\eps).
 \end{align*}
 
Collecting the results, we have the equations \eqref{phi3-phi1}-\eqref{phi4-phi2}. That is,
\begin{align*}
 &\langle \Phi_{(+1)}^\eps,\D_{x_\ell} \Phi_{(+i)}^{\eps} \rangle = (\bk_1 - \bk_2)_\ell + i(\bk_1 + \bk_2)_\ell + \mathcal{O}(\eps);\\
&\langle \Phi_{(-1)}^\eps,\D_{x_\ell} \Phi_{(+i)}^{\eps} \rangle  = (\bk_2 - \bk_1)_\ell + i(\bk_1 + \bk_2)_\ell+ \mathcal{O}(\eps);\\
 &\langle \Phi_{(+1)}^\eps,\D_{x_\ell} \Phi_{(-i)}^{\eps} \rangle = (\bk_2 - \bk_1)_\ell + i(\bk_1 + \bk_2)_\ell+ \mathcal{O}(\eps);\\
&\langle \Phi_{(-1)}^\eps,\D_{x_\ell} \Phi_{(-i)}^{\eps} \rangle = (\bk_1 - \bk_2)_\ell + i(\bk_1 + \bk_2)_\ell + \mathcal{O}(\eps).
\end{align*}

\subsubsection{Coefficient Calculations: $\alpha^\eps$, $\beta^\eps$, $\gamma^\eps$}

For $\ell, m  \in \{1,2\}$ and $r, s \in \{+i, -i\}$
\begin{align*}
a_{\ell,m}^{(r), (s)}&=\langle \D_{x_\ell} \Phi_{(r)}^\eps, \mathscr{R}(\mu_S^\eps) \D_{x_m}\Phi_{(s)}^\eps \rangle 
\\&= \sum_{q \in \{+1, -1\}} \frac{ \overline{ \langle \Phi^\eps_{(q)}, \D_{x_\ell}\Phi_{(r)}^\eps \rangle}\langle \Phi^\eps_{(q)}, \D_{x_m}\Phi_{(s)}^\eps \rangle}{\mu_{(q)}^\eps - \mu_S^\eps} + \mathcal{O}(1) .
\end{align*}

Further, recall $\alpha^\eps = 4a^{(+i),(+i)}_{1,1} =4\langle \D_{x_1} \Phi_{(+i)}^\eps, \mathscr{R}(\mu_S^\eps) \D_{x_1}\Phi_{(+i)}^\eps \rangle.$  From the analysis above, 
\begin{align}
\frac{ \overline{ \langle \Phi_{(+1)}, \D_{x_1}\Phi_{(+i)}^\eps \rangle}\langle \Phi_{(+1)}, \D_{x_1}\Phi_{(+i)}^\eps \rangle}{\mu_{(+1)}^\eps - \mu_S^\eps} 
&= \frac{ \overline{ (\bk_1 - \bk_2)_1+ i(\bk_1 + \bk_2)_1 }  [(\bk_1 - \bk_2)_1 + i(\bk_1 + \bk_2)_1 ]+ \mathcal{O}(\eps)}{\eps ( V_{00} + 2 V_{01} + V_{11}) - \eps (V_{00} - V_{11}) + \mathcal{O}(\eps^2)} \notag
\\&= \frac{4\pi^2 \ (\overline{ 1+i })  (1 + i )+ \mathcal{O}(\eps)}{2\eps (  V_{01} + V_{11}) + \mathcal{O}(\eps^2)}  \notag
\\&= \frac{2\pi^2}{\eps} \left(\frac{\abs{ 1+i }^2 + \mathcal{O}(\eps)}{ (  V_{01} + V_{11}) + \mathcal{O}(\eps)}\right)  \notag
\\&= \frac{2\pi^2}{\eps}\left(\frac{2+ \mathcal{O}(\eps)}{   V_{01} + V_{11} + \mathcal{O}(\eps)}\right).\label{a-phi1}
\end{align}
We calculate the inner products now with $q = -1$:
\begin{align}
\frac{ \overline{ \langle \Phi_{(-1)}, \D_{x_1}\Phi_{(+i)}^\eps \rangle}\langle \Phi_{(-1)}, \D_{x_1}\Phi_{(+i)}^\eps \rangle}{\mu_{(-1)}^\eps - \mu_S^\eps} 
&= \frac{ \overline{ (\bk_2 - \bk_1)_1+ i(\bk_1 + \bk_2)_1 }  [(\bk_2 - \bk_1)_1 + i(\bk_1 + \bk_2)_1 ]+ \mathcal{O}(\eps)}{\eps ( V_{00} - 2 V_{01} + V_{11}) - \eps (V_{00} - V_{11}) + \mathcal{O}(\eps^2)} \notag
\\&= \frac{4\pi^2 \ (\overline{ -1+i })  (-1 + i )+ \mathcal{O}(\eps)}{2\eps ( V_{11}- V_{01}  ) + \mathcal{O}(\eps^2)} \notag
\\&= \frac{2\pi^2}{\eps} \left(\frac{\abs{- 1+i }^2+ \mathcal{O}(\eps)}{ (  V_{11}-V_{01} ) + \mathcal{O}(\eps)}\right) \notag
\\&= \frac{2\pi^2}{\eps}\left(\frac{2+ \mathcal{O}(\eps)}{  V_{11}- V_{01} + \mathcal{O}(\eps)}\right).\label{a-phi2} 
\end{align}

Substituting \eqref{a-phi1} and \eqref{a-phi2} into the expression for $a_{1,1}^{(+i),(+i)}(\eps)$,
\begin{align*}
\alpha^\eps = 4 a_{1,1}^{(+i),(+i)}(\eps) &= \frac{8\pi^2}{\eps}\left[\left(\frac{2+ \mathcal{O}(\eps)}{   V_{01} + V_{11} + \mathcal{O}(\eps)}\right)+\left(\frac{2+ \mathcal{O}(\eps)}{  V_{11}- V_{01} + \mathcal{O}(\eps)}\right)\right] +\mathcal{O}(1) \notag
\\&=  \frac{8\pi^2}{\eps}\left( \frac{2(V_{11}-V_{01}) + 2(V_{01}+V_{11}) + \mathcal{O}(\eps)}{(V_{11}^2 - V_{01}^2) + \mathcal{O}(\eps)}\right)+\mathcal{O}(1)
\\&=  \frac{32\pi^2}{\eps}\left( \frac{V_{11}+ \mathcal{O}(\eps)}{(V_{11}^2 - V_{01}^2) + \mathcal{O}(\eps)}\right) +\mathcal{O}(1). \notag
\end{align*}
Therefore, we have \eqref{alph-app}:
\begin{align*}
\alpha^\eps =  \frac{32\pi^2}{\eps}\left( \frac{V_{11}+ \mathcal{O}(\eps)}{(V_{11}^2 - V_{01}^2) + \mathcal{O}(\eps)}\right)+\mathcal{O}(1) = \frac{32\pi^2}{\eps} \left(\frac{V_{11}}{V_{11}^2 - V_{01}^2}\right) + \mathcal{O}(1).
\end{align*}

\nit The results for $\beta^\eps$ and $\gamma^\eps$ follow similarly. First, consider $\beta^\eps.$
\begin{equation}
\begin{split}
\beta^\eps &= 4a^{(+i),(-i)}_{1,2}(\eps) = 4 \langle \D_{x_1} \Phi_{1}^\eps, \mathscr{R}(\mu_S^\eps) \D_{x_2}\Phi_{(-i)}^\eps \rangle\\
&= \frac{4 \overline{ \langle \Phi_{(+1)}, \D_{x_1}\Phi_{(+i)}^\eps \rangle}\langle \Phi_{(+1)}, \D_{x_2}\Phi_{(-i)}^\eps \rangle}{\mu_{(+1)}^\eps - \mu_S^\eps}
 +\frac{4\overline{ \langle \Phi_{(-1)}, \D_{x_1}\Phi_{(+i)}^\eps \rangle}\langle \Phi_{(-1)}, \D_{x_2}\Phi_{(-i)}^\eps \rangle}{\mu_{(-1)}^\eps - \mu_S^\eps} + \mathcal{O}(1).
 \end{split}
\end{equation}
Following the analysis as for $\alpha^\eps$, we find the following.
\begin{align}
\frac{ \overline{ \langle \Phi_{(+1)}, \D_{x_1}\Phi_{(+i)}^\eps \rangle}\langle \Phi_{(+1)}, \D_{x_2}\Phi_{(-i)}^\eps \rangle}{\mu_{(+1)}^\eps - \mu_S^\eps} 
&= \frac{2\pi^2}{\eps}\left(\frac{2+ \mathcal{O}(\eps)}{   V_{01} + V_{11} + \mathcal{O}(\eps)}\right);\label{beta-1}
\end{align}
\begin{align}
\frac{ \overline{ \langle \Phi_{(-1)}, \D_{x_1}\Phi_{(+i)}^\eps \rangle}\langle \Phi_{(-1)}, \D_{x_2}\Phi_{(-i)}^\eps \rangle}{\mu_{(-1)}^\eps - \mu_S^\eps} 
&= \frac{2\pi^2}{\eps}\left(\frac{2+ \mathcal{O}(\eps)}{   V_{11}- V_{01}  + \mathcal{O}(\eps)}\right).\label{beta-2}
\end{align}
Combining \eqref{beta-1} and \eqref{beta-2},
\begin{align*}
\beta^\eps = 4 \langle \D_{x_1} \Phi_{(+i)}^\eps, \mathscr{R}(\mu_S^\eps) \D_{x_2}\Phi_{(-i)}^\eps \rangle&= \frac{8\pi^2}{\eps}\left[\left(\frac{2+ \mathcal{O}(\eps)}{   V_{01} + V_{11} + \mathcal{O}(\eps)}\right)+\left(\frac{2+ \mathcal{O}(\eps)}{  V_{11}- V_{01} + \mathcal{O}(\eps)}\right)\right] 
\\&=  \left(\frac{8\pi^2}{\eps}\right) \frac{2(V_{11}-V_{01}) + 2(V_{01}+V_{11}) + \mathcal{O}(\eps)}{(V_{11}^2 - V_{01}^2) + \mathcal{O}(\eps)}
\\&=  \left(\frac{8\pi^2}{\eps}\right) \frac{4V_{11}+ \mathcal{O}(\eps)}{(V_{11}^2 - V_{01}^2) + \mathcal{O}(\eps)}.
\end{align*}

Thus we have \eqref{beta-app}:
\begin{equation*}
\beta^\eps =  \left(\frac{32\pi^2}{\eps}\right)  \frac{V_{11}+ \mathcal{O}(\eps)}{(V_{11}^2 - V_{01}^2) + \mathcal{O}(\eps)} =  \frac{32\pi^2}{\eps} \left(\frac{V_{11}}{V_{11}^2 - V_{01}^2}\right) + \mathcal{O}(1).
\end{equation*}

\noindent Finally, consider the decomposition for $\gamma^\eps$.
\begin{equation}\label{gamma-full}
\begin{split}
\gamma^\eps &= 4a^{(+i),(-i)}_{1,1}(\eps) = 4\langle \D_{x_1} \Phi_{(+i)}^\eps, \mathscr{R}(\mu_S^\eps) \D_{x_1}\Phi_{(-i)}^\eps \rangle
\\& = \frac{4 \overline{ \langle \Phi_{(+1)}, \D_{x_1}\Phi_{(+i)}^\eps \rangle}\langle \Phi_{(+1)}, \D_{x_1}\Phi_{(-i)}^\eps \rangle}{\mu_{(+1)}^\eps - \mu_S^\eps}
 +\frac{4 \overline{ \langle \Phi_{(-1)}, \D_{x_1}\Phi_{(+i)}^\eps \rangle}\langle \Phi_{(-1)}, \D_{x_1}\Phi_{(-i)}^\eps \rangle}{\mu_{(-1)}^\eps - \mu_S^\eps} +\mathcal{O}(1).
 \end{split}
\end{equation}
Again, similarly as for $\alpha^\eps$, the first two terms of \eqref{gamma-full} yield the following.
\begin{align}
\frac{ \overline{ \langle \Phi_{(+1)}, \D_{x_1}\Phi_{(+i)}^\eps \rangle}\langle \Phi_{(+1)}, \D_{x_1}\Phi_{(-i)}^\eps \rangle}{\mu_{(+1)}^\eps - \mu_S^\eps} 
&= \frac{2\pi^2}{\eps}\left(\frac{2i+ \mathcal{O}(\eps)}{   V_{01} + V_{11} + \mathcal{O}(\eps)}\right);\label{gam-1}
\end{align}

\begin{align}
\frac{ \overline{ \langle \Phi_{(-1)}, \D_{x_1}\Phi_{(+i)}^\eps \rangle}\langle \Phi_{(-1)}, \D_{x_1}\Phi_{(-i)}^\eps \rangle}{\mu_{(-1)}^\eps - \mu_S^\eps} 
&= \frac{2\pi^2}{\eps}\left(\frac{-2i+ \mathcal{O}(\eps)}{   V_{11}- V_{01}  + \mathcal{O}(\eps)}\right).\label{gam-2}
\end{align}
Combining \eqref{gam-1} and \eqref{gam-2},
\begin{align*}
 \langle \D_{x_1} \Phi_{(+i)}^\eps, \mathscr{R}(\mu_S^\eps) \D_{x_1}\Phi_{(-i)}^\eps \rangle&= \frac{2\pi^2}{\eps}\left[\left(\frac{2i+ \mathcal{O}(\eps)}{   V_{01} + V_{11} + \mathcal{O}(\eps)}\right)+\left(\frac{-2i+ \mathcal{O}(\eps)}{  V_{11}- V_{01} + \mathcal{O}(\eps)}\right)\right] 
\\&=  \left(\frac{2\pi^2}{\eps}\right) \frac{2i(V_{11}-V_{01}) - 2i(V_{01}+V_{11}) + \mathcal{O}(\eps)}{(V_{11}^2 - V_{01}^2) + \mathcal{O}(\eps)}
\\&=  \left(\frac{2\pi^2}{\eps}\right) \frac{(-4i)V_{01}+ \mathcal{O}(\eps)}{(V_{11}^2 - V_{01}^2) + \mathcal{O}(\eps)}.
\end{align*}

\noindent Therefore, we have \eqref{gamma-app}:
\begin{equation*}
\gamma^\eps = 4  \langle \D_{x_1} \Phi_{(+i)}^\eps, \mathscr{R}(\mu_S^\eps) \D_{x_1}\Phi_{(-i)}^\eps \rangle= -\left(\frac{32\pi^2}{\eps} \ i\right)  \frac{V_{01}+ \mathcal{O}(\eps)}{(V_{11}^2 - V_{01}^2) + \mathcal{O}(\eps)}+\mathcal{O}(1) = - \frac{32\pi^2}{\eps} i\left(\frac{V_{01}}{V_{11}^2 - V_{01}^2}\right) + \mathcal{O}(1).
\end{equation*}
With the coefficients $\alpha^\eps, \beta^\eps,$ and $\gamma^\eps$, expanded, the description for the expression \eqref{mu_eps-diff}:

\begin{equation*}
\mu_{\pm}^\eps(\bM + \kappa) - \mu_S^\eps =
(1-\alpha^\eps)|\kappa|^2+\mathscr{Q}^\eps_6(\kappa)\ \pm \sqrt{\Big|\ \gamma^\eps(\kappa_{1}^{2} - \kappa_{2}^{2})+ 2\beta^\eps \kappa_{1}\kappa_{2}\ \Big|^2\ +\ \mathscr{Q}^\eps_8(\kappa)}
,\end{equation*}
described in Corollary \ref{small-eps-disp} is complete.

\bibliographystyle{alpha}
\bibliography{kmow}

\end{document}